\documentclass[12pt, a4paper, parskip=half, abstracton, bibliography=totoc]{scrartcl}
\pdfoutput=1

\usepackage{array}
\usepackage{marginnote}
\usepackage{xcolor}
\usepackage{amscd,amssymb,amsfonts,amsmath,latexsym,amsthm}
\usepackage{hyperref}
\usepackage[all,cmtip]{xy}
\usepackage{mathrsfs}
\usepackage{tikz}
\usepackage{tikz-cd}
	\tikzcdset{row sep/normal=1.3em}
	\tikzcdset{column sep/normal=1.3em}
\usepackage{graphicx}
\usepackage{bm} 

\usepackage{etoolbox}

\def\hB{\hspace*{\fill}$\qed$}

\usepackage{defs_pp1}
\usepackage{slashed}
\usepackage[utf8]{inputenc}
\usepackage{microtype}
\usepackage[english]{babel}
\usepackage{mathtools}
\usepackage{bm}
\usepackage{esvect}
\usepackage{rotating}

\title{A stable $\infty$-category for equivariant $\mathrm{K\!K}$-theory}
\author{
Ulrich Bunke\thanks{Fakult{\"a}t f{\"u}r Mathematik,
Universit{\"a}t Regensburg,
93040 Regensburg,
GERMANY\newline
\href{mailto:ulrich.bunke@mathematik.uni-regensburg.de}{ulrich.bunke@mathematik.uni-regensburg.de}}
\and Alexander Engel\thanks{Institut f{\"u}r Mathematik und Informatik, Universit{\"a}t Greifswald, 17489 Greifswald, GERMANY\newline
\href{mailto:alexander.engel@uni-greifswald.de}{alexander.engel@uni-greifswald.de}}
\and Markus Land\thanks{Mathematisches Institut, Ludwig-Maximilians-Universit\"at M\"unchen, 80333 M\"unchen, GERMANY\newline \href{mailto:markus.land@math.lmu.de}{markus.land@math.lmu.de}}
}

\numberwithin{equation}{section}
\newtheorem{theorem}{Theorem}[section] 
\newtheorem{prop}[theorem]{Proposition}
\newtheorem{lem}[theorem]{Lemma}
\newtheorem{ddd}[theorem]{Definition}
\newtheorem{kor}[theorem]{Corollary}

\theoremstyle{remark}
\theoremstyle{definition}

\newtheorem{ex}[theorem]{Example}
\newtheorem{rem}[theorem]{Remark}

\newtheorem{construction}[theorem]{Construction}

\newcommand{\inj}{\mathrm{inj}}
\newcommand{\GJ}[1]{\mathrm{GJ}^{#1}}
\newcommand{\JG}[1]{\mathrm{GJ}_{#1}}

\newcommand{\kkGeq}{W_{\kkGsk}}
\newcommand{\kkAi}{\mathrm{kk}_{C^{*}\mathbf{Cat},\infty}}
\newcommand{\kkA}{\mathrm{kk}_{C^{*}\mathbf{Cat}}}
\newcommand{\kkGA}{\mathrm{kk}_{C^{*}\mathbf{Cat}}^{G}}
\newcommand{\kkHA}{\mathrm{kk}_{C^{*}\mathbf{Cat}}^{H}}

\newcommand{\kkGAi}{\mathrm{kk}_{C^{*}\mathbf{Cat},\infty}^{G}}
\newcommand{\Bd}{\mathrm{Bd}}
\newcommand{\npCat}{{}_{\mathrm{pre}}C^{*}\mathbf{Cat}^{\mathrm{nu}}}
\newcommand{\compl}{\mathrm{compl}}
\newcommand{\nClincat}{{}^{*}\Cat^{\mathrm{nu}}_{\C}}

\newcommand{\npAlg}{{}_{\mathrm{pre}}C^{*}\mathbf{Alg}^{\mathrm{nu}}}
\newcommand{\nsAlg}{{}^{*}\mathbf{Alg}^{\mathrm{nu}}}
\newcommand{\ccpl}{\mathrm{ccpl}}
\newcommand{\nsCat}{{}^{*}\mathbf{Cat}_{\mathbb{C}}^{\mathrm{nu}}}
\newcommand{\CAT}{\mathbf{CAT}}

\newcommand{\KKHs}{\mathrm{KK}_{\mathrm{sep}}^{H}}
\newcommand{\kks}{\mathrm{kk}_{\mathrm{sep}}}
\newcommand{\kkHs}{\mathrm{kk}^{H}_{\mathrm{sep}}}

\newcommand{\kk}{\mathrm{kk}}

\renewcommand{\Ind}{\mathrm{Ind}}

\newcommand{\std}{\mathrm{std}}

\newcommand{\triv}{\mathrm{triv}}

\newcommand{\topp}{\mathrm{top}}

\newcommand{\Homol}{\mathrm{Hg}}

\newcommand{\bJ}{\mathbf{J}}

\newcommand{\Coind}{\mathrm{Coind}}

\newcommand{\ho}{\mathrm{ho}}

\newcommand{\Res}{\mathrm{Res}}

\newcommand{\Orb}{\mathbf{Orb}}

\newcommand{\Hilb}{\mathbf{Hilb}}

\newcommand{\Fin}{\mathbf{Fin}}

\newcommand{\Ob}{\mathrm{Ob}}

\newcommand{\bB}{{\mathbf{B}}}

\newcommand{\incl}{\mathrm{incl}}

\newcommand{\An}{\mathrm{An}}

\newcommand{\bM}{\mathbf{M}}

\renewcommand{\SS}{\mathrm{SS}}

\newcommand{\bA}{{\mathbf{A}}}

\newcommand{\const}{{\mathtt{const}}}

\newcommand{\cD}{{\mathcal{D}}}

 \newcommand{\Vect}{{\mathbf{Vect}}}

 \newcommand{\Cat}{{\mathbf{Cat}}}

\newcommand{\lf}{G,\mathrm{an}}

\newcommand{\lto}{\longrightarrow}

\newcommand{\Ccat}{{\mathbf{C}^{\ast}\mathbf{Cat}}}
\newcommand{\Calg}{{\mathbf{C}^{\ast}\mathbf{Alg}}}

\newcommand{\op}{\mathrm{op}}

\newcommand{\add}{\mathrm{add}}

\newcommand{\nCcat}{C^{*}\mathbf{Cat}^{\mathrm{nu}}}
\renewcommand{\Ccat}{C^{*}\mathbf{Cat}}
\newcommand{\alg}{\mathrm{alg}}

\renewcommand{\Calg}{C^{*}\mathbf{Alg}}
\newcommand{\nCalg}{C^{*}\mathbf{Alg}^{\mathrm{nu}}}

\newcommand{\Kcat}{K^{C^{*}\mathbf{Cat}}}

\newcommand{\Ass}{\mathrm{Asmbl}}
\newcommand{\Kast}{K^{C^{*}\mathbf{Alg}}}
\newcommand{\an}{\mathrm{an}}

\newcommand{\sepa}{\mathrm{sep}}

\newcommand{\kkG}{\mathrm{kk}^{G}}
\newcommand{\kkGp}{\mathrm{kk}^{G,\prime}}
\newcommand{\kkH}{\mathrm{kk}^{H}}
\newcommand{\KKG}{\mathrm{KK}^{G}}
\newcommand{\KKGp}{\mathrm{KK}^{G,\prime}}
\newcommand{\KKH}{\mathrm{KK}^{H}}
\newcommand{\KK}{\mathrm{KK}}
\newcommand{\KKs}{\mathrm{KK}_{\sepa}}
\newcommand{\KKGs}{\mathrm{KK}_{\sepa}^{G}}
\newcommand{\KKHsk}{\mathrm{KK}_{0}^{H}}
\newcommand{\KKGsk}{\mathrm{KK}_{0}^{G}}
\newcommand{\kkGsk}{\mathrm{kk}_{0}^{G}}
\newcommand{\kkHsk}{\mathrm{kk}_{0}^{H}}
\newcommand{\kkGs}{\mathrm{kk}_{\sepa}^{G}}

\newcommand{\KKth}{K\!K}
\newcommand{\nCalgs}{C^{*}\mathbf{Alg}^{\mathrm{nu}}_{\mathrm{sep}}}
 
 \newcommand{\kkGstensora}{\mathrm{kk}_{\sepa}^{G,\otimes}}
 \newcommand{\kkHstensora}{\mathrm{kk}_{\sepa}^{H,\otimes}}
  \newcommand{\KKGstensora}{\mathrm{KK}_{\sepa}^{G,\otimes}}
 \newcommand{\KKHstensora}{\mathrm{KK}_{\sepa}^{H,\otimes}}
   \newcommand{\KKGtensora}{\mathrm{KK}^{G,\otimes}}
 \newcommand{\KKHtensora}{\mathrm{KK}^{H,\otimes}}
   
\newcommand{\KKGstensor}{\mathrm{KK}_{\sepa}^{G,\otimes_?}}
\newcommand{\kkGstensor}{\mathrm{kk}_{\sepa}^{G,\otimes_?}}
\newcommand{\KKGtensor}{\mathrm{KK}^{G,{\otimes}_?}}
\newcommand{\kkGtensor}{\mathrm{kk}^{G,{\otimes}_?}}
\newcommand{\kkGAtensor}{\mathrm{kk}_{C^{*}\mathbf{Cat}}^{G,{\otimes}_?}}
\newcommand{\nCcattensor}{C^{*}\mathbf{Cat}^{\mathrm{nu},\otimes_?}}
 
\newcommand{\la}{\mathrm{}}
 \newcommand{\exa}{\mathrm{ex}}

\newcommand{\nsttCat}{{}^{*} \mathbf{Cat}^{\mathrm{nu}}}
\newcommand{\pGTop}{G\mathbf{Top}^{\mathrm{prop}}}
\newcommand{\ppGTop}{G\mathbf{Top}^{\mathrm{prop}}_{+}}

\newcommand{\ppGTops}{G\mathbf{Top}_{\mathrm{2nd},+}^{\mathrm{prop}}}

\newcommand{\kkGC}{\mathrm{kk}^{G}C_{0}}
\newcommand{\kkGCtensor}{\mathrm{kk}^{G}C_{0}^{\otimes_?}}
\newcommand{\kkGCs}{\mathrm{kk}_{\mathrm{sep}}^{G}C_{0}}

\newcommand{\countably}{countably }

\begin{document}
	
\maketitle

\vspace*{-3ex}

\begin{abstract}
For a countable group $G$ we construct a small, idempotent complete, symmetric monoidal, stable $\infty$-category $\KKGs$ whose homotopy category recovers the triangulated equivariant Kasparov category of separable $G$-$C^*$-algebras, and exhibit its universal property. Likewise, we consider an associated presentably symmetric monoidal, stable $\infty$-category $\KKG$ which receives a symmetric monoidal functor $\kkG$  from possibly non-separable $G$-$C^*$-algebras and discuss its universal property. In addition to the symmetric monoidal structures, we construct various change-of-group functors relating these KK-categories for varying $G$. We use this to define and establish key properties of a (spectrum valued) equivariant {analytic} $K$-homology theory on proper and locally compact $G$-topological spaces, allowing for coefficients in arbitrary $G$-$C^*$-algebras. Finally, we extend the functor $\kkG$ from $G$-$C^*$-algebras to $G$-$C^*$-categories. These constructions are key in a companion paper about a form of equivariant Paschke duality and assembly maps.
\end{abstract}

\tableofcontents
\setcounter{tocdepth}{5}

\paragraph{Acknowledgements.}

Ulrich Bunke was supported by the SFB 1085 (Higher Invariants) funded by the Deutsche Forschungsgemeinschaft (DFG).

Alexander Engel acknowledges financial support by the Deutsche Forschungsgemeinschaft (DFG, German Research Foundation) through the Priority Programme SPP 2026 ``Geometry at Infinity'' (EN 1163/5-1, project number 441426261, Macroscopic invariants of manifolds) and through Germany's Excellence Strategy EXC 2044-390685587, Mathematics Münster: Dynamics -- Geometry -- Structure.

 {Markus Land was supported by the research fellowship DFG 424239956, and by the Danish National Research Foundation through the Copenhagen Centre for Geometry and Topology (DNRF151). }

\section{Introduction and statements} 

\subsection{A small stable \texorpdfstring{$\infty$}{infty}-category of separable \texorpdfstring{$G$-$C^*$}{G-Cstar}-algebras}

Let $G$ be a countable group and $A$, $B$ be  {$G$}-$C^{*}$-algebras, i.e.\ $C^*$-algebras with an action of $G$ by automorphisms.   In this situation, 
we have the abelian group $\KKth^{G}(A,B)$ introduced  by Kasparov   in his work on the Novikov conjecture \cite{kasparovinvent}. This group depends contravariantly on the first algebra $A$ and covariantly on the second algebra $B$. 

The construction of $\KKth^{G} {(A,B)}$ can be generalized  to graded $G$-$C^{*}$-algebras, to families of $C^{*}$-algebras parametrized by a space as in \cite{kasparovinvent}, or to the case where $G$ is  a locally compact groupoid \cite{MR1686846}.  Though we think that many of our constructions  also work in more general situations,
in the present paper we will stick to the classical situation since this is what immediately generalizes to $C^{*}$-categories, and  what we need for the applications to Paschke duality and assembly maps in  \cite{bel-paschke}.   

If we restrict to separable $G$-$C^{*}$-algebras, then the Kasparov product
$$\KKth^{G}(A,B)\otimes \KKth^{G}(B,C)\to \KKth^{G}(A,C)\, ,$$
which has also been introduced in  \cite{kasparovinvent},  can be considered as the composition law of an 
$\Ab$-enriched category  $\KKGsk$. The objects of this category   are the  separable $G$-$C^{*}$-algebras  
and the morphism groups are given by 
$$\Hom_{\KKGsk}(A,B) \coloneqq \KKth^{G}(A,B)\, .$$
The category $\KKGsk$ is additive  {and the sum} is represented by the direct sum of $C^{*}$-algebras. 
As observed and exploited in \cite{MR2193334}, $\KKGsk$  has a  {canonical} refinement to a triangulated category.

{Often, triangulated categories arise as   homotopy categories
of 
stable $\infty$-categories  \cite[{Sec.\ 1.1.2}]{HA}.  The first objective of this paper is to show that the triangulated categories $\KKGsk$ are no exception to this principle. Standard references for the language of $\infty$-categories in general are \cite{htt}, \cite{Cisinski:2017}. For the definition and properties of   stable $\infty$-categories we refer to
  \cite[Ch.\ 1]{HA}.}

The first  result of this paper is 
 the construction of a stable $\infty$-category  $\KKGs$ whose homotopy category  is canonically equivalent,  {as a triangulated category}, to $\KKGsk$.
This generalizes a construction  of \cite{Land:2016aa}  from the non-equivariant to the equivariant case. 
In the following we provide the precise statement.

 Let $\Fun(BG,\nCalgs)$ denote the category of separable, possibly non-unital $C^{*}$-algebras with $G$-action.  
We then have a canonical functor
\begin{equation}\label{erwwergwergrwegrgerwgwergrew}
\kkGsk \colon \Fun(BG,\nCalgs)\to \KKGsk
\end{equation}
which is the identity on objects and sends $f \colon A\to B$ to the element $[f]$ in $\KKth^{G}(A,B)$ represented by the Kasparov $(A,B)$-module
$(B,f,0)$. 

\begin{ddd}\label{werigwoergreewfefewrf}
A morphism $f$ in $ \Fun(BG,\nCalgs)$ is called a $\kkGsk$-equivalence if $\kkGsk(f)$ is an isomorphism. \end{ddd}
We let $\kkGeq$  denote the  collection  of all $\kkGsk$-equivalences. 
The following definition is the direct generalization of    \cite{Land:2016aa}  to the equivariant case.
\begin{ddd}\label{wtpohwopggergeg}
We define the $\infty$-category
$$\KKGs \coloneqq \Fun(BG,\nCalgs)[\kkGeq^{-1}] $$
and 
  let $$\kkGs \colon \Fun(BG,\nCalgs)\to \KKGs$$
  denote the canonical functor.
\end{ddd}
Here the  $\infty$-category $\Fun(BG,\nCalgs)[\kkGeq^{-1}]$ denotes the Dwyer--Kan localization   
of
$\Fun(BG,\nCalgs)$ at the collection $\kkGeq$ of the $\kkGsk$-equivalences. {Such a Dwyer--Kan localization exists and is characterized by a universal property, }see Remark \ref{weroigujweogerwgregfw} for  {the precise statement}.  We have the following theorem:
\begin{theorem}\label{weighigregwgregw9}\mbox{}
\begin{enumerate}
\item The $\infty$-category $\KKGs$ is stable.
\item \label{weoitgjowergwegregw9} We have a canonical factorization
\begin{equation}\label{fewpoqkopfqewffwqwefqf1}
\xymatrix{
\Fun(BG,\nCalg_{\sepa})\ar[rr]^-{\kkGsk}\ar[rd]_-{\kkGs}&&\KKGsk\\
&\KKGs\ar@{-->}[ur]_-{\ho}&
}
\end{equation} 
and $\ho$ is an equivalence of triangulated categories.
\end{enumerate}
\end{theorem}
The proof of Theorem \ref{weighigregwgregw9} is a  modification of the argument given for the non-equivariant case in   \cite{Land:2016aa}. It will be given 
in Section \ref{wetgijweogwegwerger}.\footnote{Note that the link points to the end of the proof.}

The  specific motivation for the present paper was the  need,  {in the companion paper \cite{bel-paschke}}, to refine the classical equivariant {analytic} $K$-homology functor with coefficients in a $G$-$C^{*}$-algebra $A$
$$X\mapsto K_{A,*}^{\lf}(X) \coloneqq \KKth_{*}^{G}(C_{0}(X),A)$$
 to a spectrum valued functor.  Using  Theorem \ref{weighigregwgregw9} we get such a refinement  by setting \begin{equation}\label{qwefiojfoiwfweqfqwefewfeqw}
X\mapsto K_{A}^{\lf}(X) \coloneqq \KKGs( C_{0}(X), A)\, ,
\end{equation}
where for the moment $A$ must be separable, and is $X$ a locally compact and second countable topological $G$-space 
so that $C_{0}(X)$ is separable, too. Here $\KKGs( B, A)$ is a short-hand notation for  the spectrum $\map_{\KKGs}(\kkGs(B),\kkGs(A))$. 
In Definition \ref{qffohfiuwehfiowefqewfqwfqwefqewf} below we will   remove the restrictions  on $A$ and $X$.

The next theorem lists properties
of the functor $\kkGs$ and the $\infty$-category $\KKGs$
which reflect well-known properties of Kasparov's bifunctor $\KKth^{G}$.
We  first explain some of the notions appearing in the statement. We consider 
 a functor from   $G$-$C^{*}$-algebras  to a stable $\infty$-category.  It is called reduced if it sends the zero algebra to a zero object. It  is    semiexact if  it   sends every    semisplit exact   sequence    to a fibre sequence, where an exact sequence is semisplit if it admits an equivariant    cpc  
 (completely positive {and} contractive) split (see also Definition \ref{qwroigjqrwgqwrfqewfqewfq}.\ref{qeriughqeiruferwqkfjqr9ifuq1}). It   is   $\mathbb{K}^{G}$-stable
if it sends morphisms of the form 
\begin{equation}\label{342uihuihfewfwefweferf}
A\otimes K(H)\to A\otimes K(H')
\end{equation} to equivalences, where
 $H\to H'$ is an equivariant  isometric inclusion of non-zero separable $G$-Hilbert spaces. It is homotopy invariant if  it
   sends  the morphisms {of the form} \begin{equation}\label{ewgegreggergegeggegwergw}
A\to C([0,1])\otimes A
\end{equation}  given by the embedding $a\mapsto {1\otimes a}$
 to  equivalences.
Finally, we refer to \cite[Def.\ 2.5]{MR2193334}  for the notion of admissibility of a diagram $A \colon \nat\to \Fun(BG,\nCalgs)$.

\begin{theorem}\label{qroifjeriogerggergegegweg}\mbox{}
\begin{enumerate}
\item \label{qoirwfjhqoierggrg0}  $\kkGs$ is reduced.
\item \label{qoirwfjhqoierggrg1} $\kkGs$ is  semiexact.
\item  \label{qoirwfjhqoierggrg2}  $\kkGs$ is $\mathbb{K}^{G}$-stable.
\item  \label{qoirwfjhqoierggrg3}  $\kkGs$ is homotopy invariant.
\item  \label{qoirwfjhqoierggrg4}  $\KKGs$ admits countable colimits  and is therefore   idempotent complete.
\item  \label{qoirwfjhqoierggrg5}  $\kkGs$ preserves countable sums.
\item  \label{qoirwfjhqoierggrg6} $\kkGs$ preserves colimits of admissible diagrams $A \colon \nat\to \Fun(BG,\nCalgs)$.
\end{enumerate}
\end{theorem}

The proof of this theorem will be given 
 in Section \ref{woigwgwgwerg9}.

 The functor \eqref{erwwergwergrwegrgerwgwergrew}  has a characterization by universal properties
 for functors to additive categories   \cite{higsondiss}, \cite{Thomsen-equivariant}, \cite{Meyer:aa}, see Proposition \ref{wtohiwthwgrgwerg}.
Our next theorem states that $\kkGs$ has a similar universal property  for functors 
from  $\Fun(BG,\nCalg_{\sepa})$ to 
objects of 
the large category $\Cat^{\exa}_{\infty}$  of small stable $\infty$-categories and exact functors. 

\begin{theorem}\label{wtohwergerewgrewgregwrg1}
The functor $\kkGs \colon \Fun(BG,\nCalg_{\sepa})\to \KKGs$ is initial among functors
from $\Fun(BG,\nCalg_{\sepa})$ to 
 objects of $\Cat^{\exa}_{\infty}$ which are reduced, 
 semi\-exact and $\mathbb{K}^{G}$-stable.
\end{theorem}
The proof of this theorem will be given in Section \ref{qeroijgoiergeregergergewrgergewg}.
We also have a version  {for} additive targets,  {see} Theorem~\ref{wtohwergerewgrewgregwrg}, {in which semiexactness is replaced by split-exactness.}

\begin{rem}  
 Note that the universal property stated in Theorem \ref{wtohwergerewgrewgregwrg1} is different from 
 the obvious {one} 
 stating that $\kkGs$ is the initial functor to $\infty$-categories which inverts ${\kkG_{0}}$-equivalences. The latter   holds  true by definition of $\KKGs$ as a Dwyer--Kan localization.
 
 { Note also that in most references, the universal property is stated for functors which are in addition homotopy invariant. In the unequivariant situation it was first proven by Higson that homotopy invariance follows from split-exactness and $\mathbb{K}$-stability, and in the equivariant situation the same is true, see also Remark~\ref{rem:homotopy-invariance}}.  {\hB}
\end{rem}

In the following, we consider the minimal and maximal tensor products $\otimes_{\min}$ and $\otimes_{\max} $ of $C^{*}$-algebras. Both of them equip the category $\Fun(BG,\nCalg)$ of  possibly non-unital $C^{*}$-algebras with $G$-action   with a symmetric monoidal structure and preserve separable algebras.  

\begin{prop}[Proposition \ref{eqrgoerjgpergwegerg1}]\label{eqrgoerjgpergwegerg}\mbox{}
The tensor product $\otimes_{?}$ for $?$ in $\{\min,\max\}$  descends to a  {bi-}exact symmetric monoidal structure  on $\KKGs$, and  $\kkGs$  refines to a 
symmetric monoidal functor 
$$\kkGstensor \colon \Fun(BG,\nCalgs)^{\otimes_{?}}\to \KKGstensor\, .$$ 
{Moreover, the tensor structure $\otimes_{?}$ on $\KKGs$ preserves countable colimits in each variable.}
\end{prop}

\subsection{A presentable stable \texorpdfstring{$\infty$}{infty}-category for \texorpdfstring{$G$-$C^*$}{G-Cstar}-algebras}

For the purpose of the application in  \cite{bel-paschke}, the restriction of the definition of {analytic} $K$-homology in  \eqref{qwefiojfoiwfweqfqwefewfeqw} 
to separable coefficient algebras $A$ is not sufficient.
Therefore  we must extend the functor $\kkGs$ from separable $C^{*}$-algebras to all $C^{*}$-algebras with $G$-action. In order to fix size issues
we choose an increasing  sequence  of three Grothendieck universes whose elements will be called  small, large and  very large sets. 
All $C^{*}$-algebras are assumed to be small.  The category $\Fun(BG,\nCalgs)$ is essentially small, and it follows from the details of the proof of Theorem \ref{weighigregwgregw9} that $\KKGs$ is also essentially small. In contrast, 
the category $\Fun(BG,\nCalg)$  is large, but locally small. 
\begin{ddd}\label{wtrohjwrthrthehrth}
We define the 
$\infty$-category $$\KKG \coloneqq \Ind(\KKGs)$$ as the $\Ind$-completion of $\KKGs$ and let 
\begin{equation}\label{ewgfuqgwefughfqlefiewfqwe}
 y^{G} \colon \KKGs \to \KKG 
\end{equation}
denote the canonical functor. 
\end{ddd}  
 

\begin{rem}\label{weoiguheijwogergfregwerf}
If $\cC$ is a small, stable $\infty$-category, then by
 \cite[Prop.\ 3.2]{MR3070515}  an explicit model for the canonical functor
 $\cC\to \Ind(\cC)$ is given by the Yoneda embedding
 $$\cC\to \Fun^{\exa}(\cC^{\op},\Sp^{\la})\, , \quad C\mapsto  \map_{\cC}(-,C)\, .  $$ 
 The $\infty$-category  $\Ind(\cC)$ is   compactly generated, presentable and  stable. {Moreover,} the functor $\cC\to \Ind(\cC)^{\omega}$ exhibits the full subcategory  $\Ind(\cC)^{\omega}$ of compact objects  in $\Ind(\cC)$ as 
 the idempotent completion of $\cC$.
 \hB
\end{rem}

\begin{ddd}\label{qeroigoergrgwg}
We define the functor $$\kkG \colon \Fun(BG,\nCalg)\to \KKG$$ as the  left Kan extension 
of $y^{G}\circ \kkGs$ along the inclusion $\incl$
as indicated in
\begin{equation}\label{qfwoefjkqwpoefkewpfqewfqwfqfwefq}
\xymatrix{
\Fun(BG,\nCalg_{\sepa})\ar[dr]_-{\incl} \ar[r]^-{\kkGs}&\KKGs\ar[r]^-{y^{G} }&\KKG\,.\\
&\Fun(BG,\nCalg)\ar[ur]_-{\kkG}\ar@{}[u]^{\Downarrow}&
}
\end{equation}
\end{ddd}


\subsection{Equivariant {analytic} \texorpdfstring{$K$}{K}-homology}\label{kophetrhrtgetrgtget}

\renewcommand{\pGTop}{G\mathbf{LCH}^{\mathrm{prop}}}
\newcommand{\pTop}{\mathbf{LCH}^{\mathrm{prop}}}

\renewcommand{\ppGTop}{G\mathbf{LCH}_{+}^{\mathrm{prop}}}
 
Using the Definitions \ref{wtrohjwrthrthehrth} and \ref{qeroigoergrgwg} we can  now extend the definition of the spectrum valued {analytic} $K$-homology functor \eqref{qwefiojfoiwfweqfqwefewfeqw}
to all $G$-$C^{*}$-algebras (or even objects of $\KKG$)  $A$ and locally compact $G$-spaces $X$. 

We let  $\pGTop$ denote the category of locally compact Hausdorff spaces with $G$-action and  {equivariant, continuous and} proper maps. We further consider the category $\pGTop_{+}$  with the same objects, but with the larger set of maps 
$$\Hom_{\ppGTop}(X,Y) \coloneqq \Hom_{\pGTop}((X^{+},\infty_{X}),(Y^{+},\infty_{Y}))\, ,$$
where $X^{+}$ and $Y^{+}$ are the one-point compactifications of $X$ and $Y$,
respectively, and a map
$f \colon (X^{+},\infty_{X})\to (Y^{+},\infty_{Y})$ 
is a continuous equivariant map $X^{+}\to Y^{+}$ {with} $f(\infty_{X})=\infty_{Y}$.  
  Equivalently,
a morphism $f \colon X\to Y$ in $\ppGTop$ is a partially defined   map $X\supseteq U\stackrel{f}{\to} Y$ {on an open subset $U$ of $X$} with $f$ in $\pGTop$ which corresponds to the map $f^{+} \colon X^{+}\to Y^{+}$
such that $f^{+}_{|U}=f$ and $f(X^{+}\setminus U)=\{\infty_{Y}\}$.

Two morphisms  $X\to Y$ in $ \ppGTop $
are called properly homotopic if there exists a homotopy
$[0,1]\times X\to Y$  in $ \ppGTop $ between them.

The category $\ppGTop$ is the natural domain of the
 functor
$$C_{0}(-) \colon (\ppGTop)^{\op}\to \Fun(BG,\nCalg)$$
which sends a locally compact $G$-space to the $C^{*}$-algebra
of continuous functions vanishing at $\infty$  {with} the induced $G$-action. 
More precisely, $C_{0}(X)$ is the kernel of the evaluation map $$C_{0}(X) \coloneqq \ker(C(X^{+})\to \C)\, ,$$
where  the evaluation map takes the value at the point $\infty$.
Since we do not assume that $X$ is second countable, the algebra  $C_{0}(X)$ is in general not separable.

If $Y$ is an invariant closed subset of $X$, then we have an exact sequence
\begin{equation}\label{fqefweddqwedwd}
0\to C_{0}(X\setminus Y)\to C_{0}(X)\to C_{0}(Y)\to 0
\end{equation} 
 in $\Fun(BG,\nCalg)$.

\begin{ddd}\label{weotigjoewgfreggw9}
{We} say that $Y$ is split-closed {in $X$}, if the sequence 
\eqref{fqefweddqwedwd} is semisplit{, i.e.\ if it admits an equivariant cpc (completely positive and contractive) split.}
\end{ddd}
The following proposition provides sufficient criteria for $Y$ being split-closed.
 \begin{prop}[Proposition \ref{qeroifjqeriofewqewdwedqwdqewdqwedwed}] The closed invariant subset $Y$ of $X$ is split-closed in the following cases:
 \begin{enumerate}
  \item $G$ acts properly on an  invariant neighbourhood of $Y$ in $X$ {and $Y$ is second countable.}
  \item $Y$ admits a $G$-invariant tubular neighbourhood.
 \end{enumerate}
 \end{prop}
  
 The notion of split-closedness is introduced since $\kkG$ only sends semisplit exact sequences to fibre sequences. 

 \begin{ex}
 Let $\nat \to \beta\nat$ be the Stone--\v{C}ech compactification of the discrete space $\nat$ and $\partial\nat \coloneqq \beta\nat\setminus \nat$. Then
 $\partial \nat $ is not split-closed in $\beta\nat$. In fact, it is known that  $C_{0}(\nat) $  does not even have a closed  linear complementary  subspace in $C_{0}(\beta\nat)=C_{b}(\nat)$.\hB
 \end{ex}


\begin{ddd}\label{qffohfiuwehfiowefqewfqwfqwefqewf}
We define the equivariant {analytic} $K$-homology functor 
$$K^{\lf} \colon \ppGTop\times  \KKG\to \Sp^{\la}$$ by
$$(X,A)\mapsto K^{\lf}_{A}(X) \coloneqq \KKG(C_{0}(X),A)\, .$$
\end{ddd}

The following theorem states the basic properties of equivariant  {analytic} $K$-homology. All  spaces in the statement
belong to $\ppGTop$, and $A$ is in $\KKG$ or $\Fun(BG,\nCalg)$, where in the latter case we  drop the functor $\kkG$ in order to simplify the notation. {In  Remark \ref{rem_explain_thm_Khom} below we will explain these statements in more detail.}
 \begin{theorem}\label{wergojowtpgwergregwreg}\mbox{}
\begin{enumerate}
\item\label{qoeirjgwergwergweg} If $X$ is second countable and $A$ is a $\sigma$-unital $G$-$C^{*}$-algebra, then we have an isomorphism of $\Z$-graded abelian groups
$$ K_{A,*}^{\lf}(X)\cong \KKth_{*}^{G}(C_{0}(X),A)\, .$$
\item \label{qoeirjgwergwergweg1} The functor $K^{\lf}$ is homotopy invariant.
\item \label{qoeirjgwergwergweg2} If $Y$ is   a split-closed $G$-invariant subspace of $X$, then we have a fibre sequence  
\begin{equation}\label{eq_thm_intro_Khom_fibre}
K^{\lf}_{A}( Y )\to K^{\lf}_{A}(X) \to K^{\lf}_{A}( X\setminus Y)\,.
\end{equation}
\item   \label{wetoigjwotigwgergwrerfcerwffsfvf}
We consider an exact sequence   $0\to A\to B\to C\to 0$  in $\Fun(BG,\nCalg)$. 
If the sequence 
 is    semisplit, or if
 $X$ is  {properly} homotopy equivalent to a  finite $G$-CW complex  {with finite stabilizers}, then we have a fibre sequence
\begin{equation}\label{fwefewrfrefwed}
K^{\lf}_{A}( X )\to K^{\lf}_{B}( X )\to K^{\lf}_{C}(   X)\ .
\end{equation} 
If $X$ is second countable, then
\begin{equation}\label{fwefewrfrefwed1}
\KKG\ni A\mapsto K^{\lf}_{A}(X)
\end{equation}  preserves filtered colimits.
 \item \label{qoeirjgwergwergweg3} We have $$K^{\lf}_{A}([0,\infty)\times X)\simeq 0\, .$$ Furthermore, if
$(X_{n})_{n\in \nat}$ is a family of second countable spaces and $A$ is separable, then we have a canonical equivalence. 
$$K^{\lf}_{A}\big(\bigsqcup_{n\in \nat} X_{n}\big)\xrightarrow{\simeq} \prod_{n\in \nat} K^{\lf}_{A}(X_{n})\, .$$ 
\item \label{jfqoifjoerfqewf}
If $X_{0}\supseteq X_{1}\supseteq \ldots \supseteq X_{n} \supseteq \ldots$
is a decreasing sequence of  closed  invariant subspaces of a second countable space $X_{0}$ such that $\bigcap_{n}X_{n}\to {X_0}$ is split-closed, and $A$ is separable,  then we have an equivalence 
$$K_{A}^{\lf}\big(\bigcap_{n\in \nat} X_{n}\big) \xrightarrow{\simeq} \lim_{n\in \nat} K_{A}^{\lf}(X_{n})\, .$$
 \item \label{qoeirjgwergwergweg4}  If $H$ is a finite subgroup of $G$, then we have an equivalence 
$$
K^{\lf}_{A}(G/H)\simeq \Kast( \Res^{G}_{H}(A)\rtimes  H)\,.
$$ 
\item \label{eqrgiheigowergergergerw}  The functor $K^{\lf}$ has a lax symmetric monoidal refinement  $${\ppGTop}^{,\otimes}\times \Fun(BG,\nCalg)^{\otimes_{?}}\to \Sp^{\la \otimes}$$ for $?$ in $\{\min,\max\}$.
\end{enumerate}
\end{theorem}

\begin{rem}\label{rem_explain_thm_Khom}
In this remark we explain the meaning of the assertions of Theorem \ref{wergojowtpgwergregwreg} in greater detail.

The Assertion \ref{wergojowtpgwergregwreg}.\ref{qoeirjgwergwergweg} shows that, under the conditions on $X$ and $A$ as stated, the  functor $K^{\lf}_{A}$ is a spectrum valued refinement of the  classical equivariant  {analytic} $K$-homology functor. 

The homotopy invariance in $X$ stated in Assertion  \ref{wergojowtpgwergregwreg}.\ref{qoeirjgwergwergweg1} first of all means that the functor $K^{\lf}_{A}$ sends  the projection $[0,1]\times X\to X$ to an equivalence. Since we work in the category of proper equivariant maps,
this implies invariance of $K^{\lf}_{A}$ under proper equivariant  homotopies. Furthermore, for fixed $X$ the functor $K^{\lf}_{-}(X)$ is also homotopy invariant in the algebra variable.  

The Assertion  \ref{wergojowtpgwergregwreg}.\ref{qoeirjgwergwergweg2}
implies that $K^{\lf}_{A}$  satisfies excision for invariant split-closed decompositions $(Z,Y)$ of $X$, i.e., invariant decompositions  such that $Y$ is split-closed in $X$ and $Y\cap Z$ is split-closed in $Z$. Furthermore, the fact that
the fibre of $K^{\lf}_{A}(Y)\to K^{\lf}_{A}(X)$ only depends on the complement $X\setminus Y$ 
is often referred to as the strong excision axiom.
Note that   the map {$K^{\lf}_{A}(X) \to K^{\lf}_{A}(X\setminus Y)$ in \eqref{eq_thm_intro_Khom_fibre}} is induced by 
the partially defined map
$X\supseteq X\setminus Y\xrightarrow{\id_{X\setminus Y}} X\setminus Y$.


It immediately follows from Definition \ref{qffohfiuwehfiowefqewfqwfqwefqewf} that for fixed $X$,
the functor $$\KKG\ni A\mapsto K^{\lf}_{A}(X)\in \Sp$$ preserves limits. 
The Assertion  \ref{wergojowtpgwergregwreg}.\ref{wetoigjwotigwgergwrerfcerwffsfvf} states 
additional exactness properties of $K^{\lf}_{A}(X)$ as a functor  on $\Fun(BG,\nCalg)$ {rather than $\KKG$}. 

 The Assertion  \ref{wergojowtpgwergregwreg}.\ref{qoeirjgwergwergweg3}
  is our expression of local finiteness of $K^{\lf}_{A}$.  The second part is often referred to as the cluster axiom.

The Assertion  \ref{wergojowtpgwergregwreg}.\ref{jfqoifjoerfqewf} is also called the continuity axiom. 

Following \cite{Land:2016aa}, in  Assertion  \ref{wergojowtpgwergregwreg}.\ref{qoeirjgwergwergweg4}
we use the spectrum valued $K$-theory functor for $C^{*}$-algebras \begin{equation}\label{erwfoijoiwefrweffwe}
\Kast(-) \coloneqq  \KK(\C,-) \colon \nCalg\to \Sp
\end{equation}
which is equivalent to the one constructed in e.g.\ \cite{joachimcat}, \cite{mijo}; see \cite[Prop.\ 3.7.1]{Land:2016aa}.   Let $G\Orb$  {be the orbit category of $G$, i.e.\ the full subcategory of $G$-sets consisting of transitive $G$-sets, and let} $G_{\Fin}\Orb$ denote  {its full subcategory on orbits with finite stabilizers.} 
Considering $G$-sets as discrete topological spaces, we have an embedding $G_{\Fin}\Orb\to \ppGTop$ and can define a functor
\begin{equation}\label{wregiojoiwjerogwergrewgwerge}
K^{G,\an}_{A} \colon G_{\Fin}\Orb\to \ppGTop\xrightarrow{K^{\lf}_{A}} \Sp^{\la}\, .
\end{equation}

The Assertion  \ref{wergojowtpgwergregwreg}.\ref{qoeirjgwergwergweg4} implies that 
this functor has the same values as the {Davis--L\"uck} functor  {used} in \cite{davis_lueck}, \cite{joachimcat}, \cite{Land:2017aa} (for $A=\C$) and \cite{kranz} (in general).  In \cite{bel-paschke}, we will  {upgrade this to} an equivalence of functors, see  \eqref{wergoiweforefrfwe} for a precise statement.  
 {In the Appendix \ref{weoigjegregewfr}} we explain {how the functor $K^{G,\an}_{A}$, which is {involved in the definition of the} domain of the spectrum valued Baum--Connes assembly map, features in a comparison of assembly maps, see Diagram \ref{refoijoergwergrefw}.} {Its construction   is  {one of the} motivations for  {the companion paper} \cite{bel-paschke}.}  
  
{In Assertion \ref{eqrgiheigowergergergerw} the symmetric monoidal structure on ${\ppGTop}^{,\otimes}$ is given by the cartesian product of the underlying topological spaces.}
 \hB
\end{rem}

The functor $K^{\lf}$ will be derived  in \eqref{ewrgoeijoiwejgfvsdfvsfdvfdv}
 from the  more fundamental functor 
$$\kkGC \coloneqq \kkG\circ C_{0} \colon \ppGTop \to \KKG $$
whose properties  will be stated in Theorem \ref{wetoighjwtiogewgregregweg}.
  
  The proof of  Theorem \ref{wergojowtpgwergregwreg} (and of Theorem \ref{wetoighjwtiogewgregregweg})   will employ almost all of the general results about $\kkG$ stated below.
It will be completed in Section \ref{uihiueghwegeregewg}.


\subsection{The s-finitary extension}
We now come back to the properties of the  functor $\kkG$ {from Definition \ref{qeroigoergrgwg}.}
\begin{ddd}\label{wtgpkowegrrefrwferf}
A functor $F$ defined on $\Fun(BG,\nCalg)$ is called s-finitary  if for every 
$A$  in $\Fun(BG,\nCalg)$ the canonical map
\begin{equation}\label{efrgveveeerwvfevdsfvsdvfdvdsv}
\colim_{A'\subseteq_{{\sepa}} A} F(A')\to F(A)
\end{equation}
is an equivalence, where $A'$ runs through the separable $G$-invariant subalgebras of $A$.
\end{ddd}
The  prefix `s' stands for separable.
In contrast, in the literature a  finitary functor  is usually required to preserve all filtered colimits. 
The following theorem lists the basic properties of $\kkG$.

\begin{theorem}\label{qeroigjqergfqeewfqewfqewf1}\mbox{}
\begin{enumerate}
\item  \label{qrfghqirwogrgqfqwefq} $\kkG$ is s-finitary. 
\item\label{qeirojqwoifqwefwefqewfqw}  $\kkG$ is reduced.
\item \label{weroigjwergwergwrgwgre1} 
$\kkG$ is   semiexact.
\item \label{jfowejfoiwefewfqwecvqwcwecwcew}   $\kkG$ is homotopy invariant.
\item   \label{jfowejfoiwefewfqwecvqwcwecwcew1}  $\kkG$ is $\mathbb{K}^{G}$-stable.
\end{enumerate}
\end{theorem}
The proof of this theorem will be finished in Section \ref{weotigjwotgwregergegw}.

In order to formulate the universal property of $\kkG$ we consider the very large $\infty$-category category $\CAT_{\infty}^{\ccpl\cap \exa}$ of cocomplete (with respect to small colimits) stable $\infty$-categories  and  functors preserving
small colimits.  As noted in Remark \ref{weoiguheijwogergfregwerf}  the $\infty$-category $\KKG$ is presentable  stable and
therefore belongs to $\CAT_{\infty}^{\ccpl\cap \exa}$.

\begin{theorem}[Theorem \ref{weoigjwoegerggegwegergecv}]\label{eoiruggwerregergegwo}
\label{ojroiwfewfqefqef_intro}
$\kkG$ is  initial among functors from $\Fun(BG,\nCalg)$ to objects of $\CAT_{\infty}^{\ccpl\cap \exa}$
which are s-finitary, reduced, 
$\mathbb{K}^G$-stable and semiexact. 
\end{theorem}

Assume that $  A, B$ are in $\Fun(BG,\nCalg)$. Then on the one hand, we 
 have the abelian group $\KKth_{0}^{G}(A,B)$  defined by Kasparov in \cite{kasparovinvent}. On the other
 hand, we have  the abelian group 
 $\pi_{0}\KKG(  A, B)$  
 defined by the abstract categorical procedure {in Definition \ref{wtrohjwrthrthehrth}.} 
If both $A$ and $B$ are separable, then  these  two groups  are identified by the morphism $\ho$ in   \eqref{fewpoqkopfqewffwqwefqf}.
The next proposition extends this isomorphism to all degrees and from separable to   $\sigma$-unital   $B$.\footnote{Using \cite[Sec.~3]{zbMATH03973627} on can remove the assumption that $B$ is $\sigma$-unital. For separable $A$ the functor $\KKth_{0}(A,-)$ is $s$-finitary. Thereby various definitions of the functor coincide.}
We let $ \nCalg_{\sigma}$ be the full subcategory of $\nCalg$ of $\sigma$-unital $C^{*}$-algebras. 

\begin{prop}[Proposition \ref{eroigowregwergregwgreg11}]\label{eroigowregwergregwgreg1}
  For any objects $A$   in $\Fun(BG,\nCalg_{\sepa})$ and $ B$ in $ \Fun(BG,\nCalg_{\sigma})$,  the functor $\ho$ induces 
 an isomorphism of $\Z$-graded abelian groups
  $$\pi_{*}\KKG(  A,  B)\cong  \KKth^{G}_{*}( A,  B)\, .$$
  \end{prop}

  Note that this proposition for separable  $B$
   is a direct consequence of the compatibility of $\ho$ with the 
  triangulated structure stated in Theorem \ref{weighigregwgregw9}.\ref{weoitgjowergwegregw9}. 
   
The following result extends Proposition \ref{eqrgoerjgpergwegerg} from the separable to the general case.

 \begin{prop}[Proposition \ref{togjoigerggwgergwgr1}]\label{togjoigerggwgergwgr}
 The   symmetric monoidal structure $\otimes_{?}$ on $\KKGs$  for $?$ in $\{\min,\max\}$ canonically  induces a
 presentably symmetric monoidal structure on $\KKG$  and $\kkG$ refines to a symmetric monoidal functor
 $$\kkGtensor \colon \Fun(BG,\nCalg)^{\otimes_{?}} \to \KKGtensor\, .$$ 
 \end{prop}
  
As an immediate consequence of   Proposition \ref{togjoigerggwgergwgr}
we can define for $?$ in $\{\min,\max\}$ an internal morphism functor
\begin{equation}\label{quhfiufhqewwefqwef}
\kkG_{?}(-,-) \colon (\KKG)^{\op}\times \KKG\to \KKG\, . 
\end{equation}
It is 
characterized by a natural equivalence 
 $\KKG(A,\kkG_{?}(B,C))\simeq \KKG(A\otimes_{?}B, C)$ for all $A,B,C$ in $\KKG$
 and preserves limits in both arguments. 
 
 \subsection{Change of groups functors}

In the following we consider various change of groups functors.    
If $H\to G$ is a homomorphism of groups, then we have an obvious restriction functor
$$\Res^{G}_{H} \colon \Fun(BG,\nCalg)\to \Fun(BH,\nCalg)\, .$$
If $H$ is a subgroup of $G$, then we have an induction functor
$$\Ind_{H}^{G} \colon \Fun(BH,\nCalg) \to \Fun(BG,\nCalg)	$$
which will be explained in Construction \ref{qroigjoerwgergrewgregweg}.
Finally, we have maximal and reduced crossed product functors
$$-\rtimes_{\max} G, -\rtimes_{r}G \colon  \Fun(BG,\nCalg)\to  \nCalg$$  whose details will be recalled in
 Construction \ref{qreogijqeroigergwgwegwegwergw}.
The following results say that these functors descend to functors (denoted by the same symbols) between the corresponding
 stable $\infty$-categories.

\begin{theorem}\label{reguergiweogwergegwergwerv}\mbox{}
\begin{enumerate}
\item  \label{reguergiweogwergegwergwerv1}There exists a factorization $$\xymatrix {\Fun(BG,\nCalg)\ar[r]^{\Res_{H}^{G}}\ar[d]^{\kkG}&\Fun(BH,\nCalg)\ar[d]^{\kkH}\\\KKG\ar@{..>}[r]^{\Res_{H}^{G}}&\KKH}$$
and  $\Res_{H}^{G} \colon \KKG\to \KKH$ preserves colimits and compact objects.
\item   \label{reguergiweogwergegwergwerv2}There exists a factorization $$\xymatrix {\Fun(BH,\nCalg)\ar[r]^{\Ind_{H}^{G}}\ar[d]^{\kkH}&\Fun(BG,\nCalg)\ar[d]^{\kkG}\\\KKH\ar@{..>}[r]^{\Ind_{H}^{G}}&\KKG}$$
and  $\Ind_{H}^{G} \colon \KKH\to \KKG$ preserves colimits and compact objects.
\item   \label{reguergiweogwergegwergwerv3} There exists a factorization 
$$\xymatrix {\Fun(BG,\nCalg)\ar[r]^-{ \rtimes_{?}G}\ar[d]^{\kkG}& \nCalg\ar[d]^{\kk}\\\KKG\ar@{..>}[r]^-{-\rtimes_{?}G}&\KK}$$ for $?\in \{r,\max\}$ 
and  $-\rtimes_{?}G \colon \KKG\to \KK$ preserves colimits and compact objects.
\end{enumerate}
\end{theorem}
The functor in Assertion \ref{reguergiweogwergegwergwerv3}  {induces on mapping spectra a} spectrum level version of Kasparov's descent morphism \cite{kasparovinvent}.

Since $\KKGs$ is idempotent complete by Theorem \ref{qroifjeriogerggergegegweg}.\ref{qoirwfjhqoierggrg4}, the functor $y^{G}$ {in \eqref{ewgfuqgwefughfqlefiewfqwe}} identifies $\KKGs$ with the full subcategory of $\KKG$ of compact objects{, see Remark \ref{weoiguheijwogergfregwerf}.} 
Hence the  assertion   that the functors preserve compact objects means that they induce functors between the 
separable versions of the respective $\KK$-categories. It then follows from the assertion about preservation of colimits
that the functors  {appearing in Theorem~\ref{reguergiweogwergegwergwerv}} are  equivalent to the canonical extensions of their  separable versions.
The proof of Theorem \ref{reguergiweogwergegwergwerv} will be given  in Section~\ref{reogijweogerewg}.

There exists a canonical natural transformation
\begin{equation}\label{wergekgkjgergewg}
 \iota \colon \id\to \Res^{G}_{H}\circ \Ind_{H}^{G} \, ,
\end{equation}
 of endofunctors of $\Fun(BH,\nCalg)$ which will be explained in detail after \eqref{wervwihiowecwecdscac}. It looks like the unit of an adjunction, and it becomes one  {after application of $\kkG$ (see \eqref{qewfoijoqwefewqfwefqewf} below), but it is not one before. }
 The transformation $\iota$  induces the first transformation   of functors in
\begin{equation}\label{fqwfoiofqwefqewfqewf}
{-}\rtimes_{?} H\to \Res^{G}_{H}\circ \Ind^{G}_{H}(-)\rtimes_{?}H\to   \Ind^{G}_{H}(-)\rtimes_{?}G\, ,
\end{equation}
where the second is canonically induced by the inclusion of $H$ into $G$, see \eqref{regerwgwgwergergwergerg}.


We now   assume that $H$ is a finite group. If $A$ is in $\nCalg$, then we
consider the homomorphism
$$\epsilon_{A} \colon A\to \Res_{H}(A)\rtimes H\,, \quad a\mapsto\frac{1}{|H|}  \sum_{h\in H} (a,h)\,,$$ 
 {where $\Res_{H}(A)$ denotes $A$ equipped with the trivial $H$-action} {and we refer to Construction \ref{qreogijqeroigergwgwegwegwergw} for the notation $(a,h)$ for elements in $\Res_{H}(A)\rtimes H$}.
The family 
$\epsilon=(\epsilon_{A})_{A\in \nCalg}$ is a natural transformation
\begin{equation}\label{wregoijwoierfwerfwerf1}
\epsilon \colon \id\to \Res_{H}(-)\rtimes H
\end{equation}
of endofunctors of $\nCalg$.

Let $B$ be in $\nCalg$  {and $G$ as previously a countable discrete group}. Then we have a canonical homomorphism 
 $$\lambda_{B} \colon \Res_{G}(B)\rtimes_{\max}G\to B$$
of $C^{*}$-algebras  which corresponds to the covariant representation 
$(\id_{B},\triv)$ consisting of the identity of $B$ and the trivial representation of $G$. The family 
$\lambda=(\lambda_{B})_{B\in \nCalg}$ is a natural transformation
\begin{equation}\label{qweoifjiowefwqfewfdewd}
\lambda \colon \Res_{G}(-)\rtimes_{\max}G\to \id
\end{equation}
of endofunctors of $\nCalg$.

The Assertions \ref{qrougoegwergreegerewg} and   \ref{gjeriogjreerwgwregwerg} 
of the following theorem 
are     $\infty$-categorical  level versions of Green's Imprimitivity Theorem \cite{greens} (for $?=\max$), \cite{kasparovinvent} and the Green--Julg theorem \cite{julg}. 
Both statements generalize the versions for the triangulated categories  stated in \cite{MR2193334}.
Assertion  \ref{werpogkwergrwegwef} is known as the dual Green--Julg theorem.  
 \begin{theorem}\label{wtgijoogwrewegewgrg}
\mbox{} 
\begin{enumerate}
\item\label{qrougoegwergreegerewg}  The natural transformation \eqref{wergekgkjgergewg} induces the unit of    
an adjunction
\begin{equation}\label{qewfoijoqwefewqfwefqewf}
\Ind_{H}^{G}:\KKH\rightleftarrows \KKG:\Res^{G}_{H}\, .
\end{equation}
\item  \label{eogiowpergwrewrege} The transformation \eqref{fqwfoiofqwefqewfqewf} naturally induces an equivalence of functors 
$${-}\rtimes_{?} H\to \Ind_{H}^{G}(-)\rtimes_{?}G \colon \KKH\to \KK$$
for $?$ in $\{{r},\max\}$.
\item   \label{gjeriogjreerwgwregwerg}
If $H$ is  finite, then  the natural transformation 
\eqref{wregoijwoierfwerfwerf1} induces the unit of an adjunction   
$${\Res_{H}}
:\KK\rightleftarrows \KKH: -\rtimes H\, .$$ 
\item \label{werpogkwergrwegwef} 
The natural transformation \eqref{qweoifjiowefwqfewfdewd}
induces the counit of an adjunction
$$-\rtimes_{\max} G:\KKG \rightleftarrows  \KK :{\Res_{G}}\, .$$
\end{enumerate}
\end{theorem}
Since $H$ is finite, the crossed product in Assertion \ref{gjeriogjreerwgwregwerg} need not be decorated. It is simply the algebraic  crossed product which happens to be equal to the reduced and maximal one in this case.  
Note that Assertion \ref{eogiowpergwrewrege} in the case $?=\max$ is also a consequence of Assertions \ref{qrougoegwergreegerewg} and 
\ref{werpogkwergrwegwef}.
The proof of this theorem will be completed 
in Section \ref{qwfefqqefqffeqwewrgpeoj2oigrtgwffwffqfqewfef}.

Note that $\Ind_{H}^{G}(\C)\cong C_{0}(G/H)$. Let
\begin{equation}\label{eq_ind_res_adjunction_IC}
r^{G}_{H} \colon \KKG(C_{0}(G/H),-)\simeq \KKH(\C,\Res^{G}_{H}(-))
\end{equation}
denote the equivalence of the mapping spectrum functors given by the adjunction in the Theorem \ref{wtgijoogwrewegewgrg}.\ref{qrougoegwergreegerewg}.  
Furthermore, let
$$\GJ{H} \colon \KKH(\C,\Res^{G}_{H}(-))\simeq
\KK(\C,\Res^{G}_{H}(-)\rtimes H)$$
 {denote} the equivalence of the mapping spectrum functors  given by the adjunction 
in the Theorem \ref{wtgijoogwrewegewgrg}.\ref{gjeriogjreerwgwregwerg}.

\begin{kor}\label{qeroigjwoegwergergwerg}
If $H$ is a finite subgroup of $G$, then we have an equivalence
 \begin{equation}\label{weogjwoegwegergwergewrg_intro} 
 \GJ{H}\circ r^{G}_{H}\colon \KKG(C_{0}(G/H),-)\to  { \Kast(\Res_H^G(-)\rtimes H)}
 \end{equation}
 of functors from $\Fun(BG,\nCalg)$ to $\Sp^{\la}$.
\end{kor}

By Theorem \ref{qeroigjqergfqeewfqewfqewf1}.\ref{weroigjwergwergwrgwgre1} the functor $\kkG$ sends exact sequences in $\Fun(BG,\nCalg)$  admitting equivariant cpc splits to fibre sequences. As an immediate  consequence,  the functor
$$\KKG(A,-) \colon \Fun(BG,{\nCalg})\to \Sp^{\la}$$
has the same property for every $A$ in $\KKG$. The Theorem \ref{wtoihgwgregegwerg}
shows that under certain restrictions on $A$ this functor in fact sends all exact sequences to fibre sequences. The conditions are formulated  so that   we can deduce this exactness using Corollary \ref{qeroigjwoegwergergwerg} from the fact that the usual $K$-theory functor for $C^{*}$-algebras in \eqref{erwfoijoiwefrweffwe}
sends all exact sequences to fibre sequences. 

Recall that a thick subcategory of a stable $\infty$-category
is a full stable subcategory which is closed under
taking retracts. A localizing subcategory of a presentable
stable $\infty$-category is a full stable subcategory which is closed under
all colimits.

\begin{ddd}\label{qrogiqoirgergwergergwergwergwrg} \mbox{}
\begin{enumerate}
\item \label{ergwgergwergeg25t} The objects of  the thick subcategory of 
  $\KKG$  generated by the objects $\kkG( C_{0}(G/H))$ for all finite subgroups $H$ of $G$
are called $G$-proper.
\item  \label{weroigjwegrfrw} The objects of  the  localizing subcategory of 
  $\KKG$  generated by the $G$-proper objects   are called ind-$G$-proper.
\end{enumerate}
\end{ddd}


Note that $G$-proper objects are ind-$G$-proper.
The following result provides examples of $G$-proper objects in $\KKG$.
\begin{prop}[Proposition \ref{eroigowregwergregwgreg1neu}]\label{eroigowregwergregwgreg}
If $X$  is in $\pGTop$ and  homotopy equivalent   (in $\pGTop$)  to
a finite $G$-CW complex with finite stabilizers, then 
$\kkG(C_{0}(X))$ is $G$-proper.
\end{prop}

Let $P$ be an object of $\kkG$.
 
\begin{theorem}[Theorem \ref{opcsjoisjdcoijsdcoiasdcasdcadc}]\label{wtoihgwgregegwerg} \mbox{}
\begin{enumerate}
\item \label{ergowejworefrefwerfwerfrefw}
If $P$   is  ind-$G$-proper, then the functor
$$\KKG(P,-) \colon \Fun(BG,\nCalg)\to \Sp^{\la}$$
sends all exact sequences to fibre sequences.
\item \label{ergowejworefrefwerfwerfrefw1} If $P$ is $G$-proper, then the functor $$\KKG(P,-) \colon \Fun(BG,\nCalg)\to \Sp^{\la}$$
preserves filtered colimits. 
\end{enumerate}
\end{theorem}

 \begin{rem}\label{weriogjoewgrffwerf}{
We define the
 thick subcategory of $G$-nuclear objects 
of $\KKG$  generated by objects of the form $\kkG(\Ind_{H}^{G}(A))$ for finite subgroups $H$ and  nuclear $A$  in $\nCalg_{\sepa}$ equipped with the trivial $H$-action. We further define the localizing subcategory of $\KKG$
of ind-$G$ nuclear objects generated by the $G$-nuclear objects.
 Since $\C$ is nuclear,  (ind)-$G$-proper objects are (ind)-$G$-nuclear.
It follows from \cite{Skand} that the functor
$\KKs(A,-):\nCalg_{\sepa}\to \Sp$ sends all exact sequences to fibre sequences provided $A$ is nuclear. 
 Using this fact one can show that the assertions of Theorem \ref{wtoihgwgregegwerg}
remain true (with essentially the same proof) with  (ind)-$G$-proper replaced by  
 (ind)-$G$-nuclear.}

 {For instance, if $X$  in $\pTop$ is separable and $H$ is a finite subgroup of $G$, then 
$\kkG(C_{0}(G/H\times X))$   is $G$-nuclear, but  not necessarily (ind)-$G$-proper. \hB}
\end{rem}


\subsection{Extensions to \texorpdfstring{$C^*$}{Cstar}-categories}

Again motivated by the applications in \cite{bel-paschke}, we extend the functor
$\kkG$ from $C^{*}$-algebras to $C^{*}$-categories with $G$-action. 
We refer to the beginning of Section \ref{ergoijwogergregwergwrg} for a more detailed introduction to $C^{*}$-categories.
We have a  fully faithful inclusion 
 of the category $\nCalg$ of {(possibly non-unital)} $C^{*}$-algebras into the category $\nCcat$ of small  (possibly non-unital) $C^{*}$-categories  which considers a $C^{*}$-algebra as a $C^{*}$-category with a single object.
 This inclusion fits into an adjunction 
  \begin{equation}\label{rwegkjnkjgngjerkjwgwergregegwrege}
A^{f}:\nCcat\rightleftarrows \nCalg:\incl
\end{equation} 
 {first considered} by \cite{joachimcat} {(see \cite{crosscat} for the non-unital case).}  

\begin{ddd}\label{jgorijgqoiqfewfqwfewf}
We define the functor
$$\kkGA \colon \Fun(BG,\nCcat)\to \KKG$$
as the composition
$$\Fun(BG,\nCcat)\xrightarrow{A^{f}} \Fun(BG,\nCalg)\xrightarrow{\kkG} \KKG\, .$$
\end{ddd}

{Let $\bC$ be in $\nCcat$.
\begin{ddd}\label{dqciojsaoidcjasdoc}
We 
 call $\bC$   separable if $\bC$ has countably many objects and all morphism spaces of $\bC$ are separable. \end{ddd}}

The following extends 
  Definition \ref{wtgpkowegrrefrwferf} from $C^{*}$-algebras to $C^{*}$-categories.
  
  \begin{ddd}  \label{wergoijweiogerfrefe}   {A functor $F$ defined on $\Fun(BG,\nCcat)$ is called }
  s-finitary, if for every $\bC$ in $\Fun(BG,\nCcat)$ the canonical map
\begin{equation}\label{eq_defn_sfinitary}
\colim_{\bC'} F(\bC')\to F(\bC)
\end{equation}
is an equivalence, where $\bC'$ runs through the poset of {separable
$G$-invariant subcategories of $\bC$}. 
\end{ddd}
As in the case of $C^{*}$-algebras, we added the letter `s' (for separable), since this notion differs from the definition of a finitary functor in \cite[Def.\ {13.7}]{cank}. The latter  requires the preservation of all filtered colimits. 
{As an example, using that $G$ is countable one can check that the identity functor on $\Fun(BG,\nCcat)$ is $s$-finitary.}

In the following theorem we list the properties of the functor $\kkGA$ for $C^{*}$-categories.  {For notions appearing in its statement we  {refer to} the following sources: unitary equivalence: \cite[Def.\ 3.19]{cank}, weak Morita equivalence:  \cite[Def.\ 18.3]{cank} and Definition \ref{igjweogierfwerfwerf}, exact sequence: \cite[Def.\ 8.5]{crosscat}, flasque: Definition \ref{weiogjegewrrferfwef} and  \cite[Def.\ 11.3]{cank}, maximal crossed product $-\rtimes G$: \cite[Def.\ 5.9]{crosscat}, reduced crossed product $-\rtimes_{r}G$:  \cite[Thm.\ 12.1]{cank}
 and relative Morita equivalence:  \cite[Def.\ 17.1]{cank}.}
\begin{theorem}\label{qroihjqiofewfqwefqwefqwefqwefqewfq}\mbox{}
\begin{enumerate}
\item  \label{ergoijogwegefwerf} The functor $\kkGA$ is s-finitary.
\item \label{oijqoifefdwefwedqwdqewdqe} The functor $\kkGA$ sends unitary 
{equivalences 
to} 
equivalences. 
\item\label{fiuqwehfiewfeewqfwefqwefqwefqewf} 
The functor $\kkGA$ 
sends weak Morita equivalences to equivalences.
\item \label{wtigjoepgrgrggergwegwergrweg} We have an equivalence
\[\kkA( -\rtimes_{ {?}} G)\simeq  (-\rtimes_{ {?}} G)\circ  \kkGA\]
of functors 
$\Fun(BG,\nCcat)\to \KK$ for $?\in \{r,\max\}$. 
\item\label{efoijqweofiqofweewfewfqfqfefe} If $  P$ in $\KKG$ is ind-$G$-proper, then the functor $\KKG( P, \kkGA(-))$
sends all exact sequences in $\Fun(BG,\nCcat)$ to fibre sequences. 
\item\label{ergoijweiogegfrfwfrefwf} If $P$ in {$\KKG$} is $G$-proper, then
$\KKG(P,\kkGA(-))$ preserves filtered colimits.
\item \label{erogijqorgwefqfewfqfewfqef}  {
If $  P$  in $\KKG$ is ind-$G$-proper,   then  $\KKG(  P,   \kkGA(-))$ annihilates flasques in $\Fun(BG,\Ccat)$.}
  \item \label{ifejhgiowergerdsfvs}
{If $P$ in $\KKG$ is ind-$G$-proper, then   $\KKG(P,\kkGA(-)) $ 
sends relative Morita equivalences in $\Fun(BG,\nCcat)$ to
 equivalences.}
\end{enumerate}
\end{theorem}
The proof of Theorem \ref{qroihjqiofewfqwefqwefqwefqwefqewfq} will be 
{given in Section \ref{ergoijwogergregwergwrg}.}

{
\begin{rem}\label{wtkoigwergrefw}
Using Remark \ref{weriogjoewgrffwerf} in Assertions 
\ref{efoijqweofiqofweewfewfqfqfefe} 
we could replace the conditions  {(ind)-$G$-proper by (ind)-$G$-nuclear.}
\end{rem}}

\begin{rem}
Let us consider here the case of the trivial group.
In \cite[Def.\ {13.4}]{cank} we introduced the notion of a homological functor $\Homol \colon \nCcat\to \bM$. {By definition, it is a functor to a stable $\infty$-category $\bM$
which   sends unitary equivalences  of $C^{*}$-categories to  equivalences 
and 
exact sequences of $C^{*}$-categories to fibre sequences.} If $\Homol$  {in addition} preserves  {filtered} colimits, then it is called finitary.

If $P$ in  {$\KK$} is ind-proper (or more generally ind-nuclear),  then  it follows from
 Theorem \ref{qroihjqiofewfqwefqwefqwefqwefqewfq}  {(and Remark \ref{wtkoigwergrefw})}
 {that}  $\KK(P,\kkA(-)) \colon \nCcat\to \Sp$ is a homological functor.
If $P$ is proper ({or more generally nuclear}), then 
this functor is also finitary. 

As a{n} example, if $X$ is locally compact and homotopy equivalent to  a finite CW-complex, then 
$P \coloneqq \kk(C_{0}(X))$ is proper by Proposition \ref{eroigowregwergregwgreg}, and therefore
$\KK(C_{0}(X),\kkA(-))$ is a homological functor. 
{More generally, if $A$  in $\nCalg_{\sepa}$ is such that   $\kk(A)$ is nuclear (e.g. if $A$ is $\kk^{G}_{0}$-equivalent to a separable nuclear  $C^{*}$-algebra), then
$\KK(A,\kkA(-))$ is a homological functor.}

The functor $\kkA$ itself is not a $\KK$-valued homological functor since it does not send all exact sequences to fibre sequences. \hB
\end{rem}

The minimal and maximal tensor products for $C^{*}$-algebras can naturally be extended   to $C^{*}$-categories, see  \cite{DellAmbrogio:2010aa} for $\otimes_{\max}$, and  \cite{antoun_voigt} for both. In Section \ref{regijweoijergregrgrev} we will give a comprehensive treatment of both cases. The  definition of the maximal tensor product  in terms of its universal property is  stated in Definition \ref{weogjgewgwrgwg}, while the analoguous definition of the minimal tensor product is given in Definition \ref{toigwjeriogregergergre}.

\begin{theorem} \label{ioeghwjoigrgwergrggrgew}
The functor $\kkGA$ refines to a symmetric monoidal functor
$$\kkGAtensor \colon \Fun(BG,\nCcat)^{\otimes_{?}}\to \KKGtensor$$
for $?$ in $\{\min,\max\}$.
\end{theorem}

The proof of this theorem will be given  in Section \ref{weoigjwoegffrefgrgwgf}.

\section{A stable \texorpdfstring{$\boldsymbol{\infty}$}{infty}-category of separable \texorpdfstring{$\boldsymbol{G}$-$\boldsymbol{C^{*}}$}{G-Cstar}-algebras}
\label{qwkfhifqefewfqfqwefqwefq}

Let $\nCalg_{\sepa}$ and $\nCalg_{\sigma}$ {denote} the full subcategories of $\nCalg$ of separable and of $\sigma$-unital $C^{*}$-algebras, respectively. 
In the present paper we work with  the  bivariant $G$-equivariant Kasparov $K\!K$-theory functor \cite{kasparovinvent} $$\KKth_{*}^{G}\colon \Fun(BG,\nCalg_{\sigma})^{\op}\times  \Fun(BG,\nCalg_{\sigma})\to \Ab^{\Z}\, , $$
where $\Ab^{\Z}$ denotes the category of $\Z$-graded abelian groups.
 {Further} good references for this functor and its properties are
 \cite{Meyer:aa} and  \cite{blackadar}.  


We will use the symbol
$\KKth^{G} \coloneqq \KKth_{0}^{G}$ for the  $\Ab$-valued functor obtained from $\KKth^{G}_{*}$ by extracting the degree-zero component. Using the Kasparov intersection product in order to define the composition we construct the $\Ab$-enriched category $\KKGsk$ and the functor 
$$\kkGsk \colon \Fun(BG,\nCalgs)\to\KKGsk$$
 (appearing in \eqref{erwwergwergrwegrgerwgwergrew})
as explained in the introduction. 
The functor $\kkGsk$ is reduced, homotopy invariant and $\mathbb{K}^{G}$-stable since it inherits these properties from 
$\KKth^{G}$.

{Recall the  Definition \ref{werigwoergreewfefewrf} of a $\kkGsk$-equivalence.}
The functor $\kkGsk$ can be characterized  by the following universal property:

 \begin{prop}\label{eoigjeoigergergwergergegergweg}
The functor $\kkGsk \colon \Fun(BG,\nCalg_\sepa) \to \KKGsk$ 
exhibits the target as    localisation of  $\Fun(BG,\nCalg_\sepa)$ 
   at the set of ${\kkGsk}$-equivalences in the sense of ordinary categories.
\end{prop}
 \begin{proof}
Let  
\[F\colon \Fun(BG,\nCalg_\sepa) \to \cD\]  be any functor
to an ordinary category $\cD$ which sends $\KKGsk$-equivalences to isomorphisms in $\cD$. We must show that there exists a unique factorisation through $\KKGsk$, indicated by the dashed arrow in the following diagram
\begin{equation}\label{qwefqknkjllnvafvsdfvfv}
\xymatrix{
	\Fun(BG,\nCalg_\sepa) \ar[r]^-{F} \ar[d]^{\kkGsk}& \cD\,. \\
	\KKGsk \ar@{-->}[ur]_{\bar{F}}&}
\end{equation}
 Since   $\kkGsk \colon \Fun(BG,\nCalg_\sepa) \to \KKGsk$ is the identity on objects we are forced to define $\bar{F}( A) \coloneqq F(  A)$ for all objects $A$ of $\Fun(BG,\nCalg_\sepa)$.

For the discussion of morphisms we use 
 the Cuntz picture of $\KKth^{G}$ due to   \cite[Sec.\ 6]{Meyer:aa}.
 Let \begin{equation}\label{vwervewwerjvklvervwevwv}
  K \coloneqq  K(\ell^{2}\otimes L^{2}(G))
\end{equation}  denote the algebra of  compact operators on the $G$-Hilbert space $\ell^{2}\otimes L^{2}(G)$, where $G$ acts on $L^{2}(G)$ by the left-regular representation, and by conjugation on the operators.  We furthermore let  
\[  q_{s}( A):=\ker( A\otimes  K\sqcup  A\otimes  K\to  A\otimes  K)\] be the kernel of the fold map.  Note that the coproduct in $C^{*}$-algebras is realized by the free product. 
By \cite[Thm.\ 6.5]{Meyer:aa}  we have a bijection
\begin{equation}\label{qfrojqoiefjqewfefefqfqewfqewf}
\KKth^{G}(  A,  B) \cong [  q_{s}( A)\otimes  K ,   q_{s}( B)\otimes  K ]\, ,
\end{equation}
where the notation $[-,-]$ stands for norm-continuous homotopy classes of morphisms in $ \Fun(BG,\nCalg_{\sepa})$. 
  Moreover, the composition in $\KKGsk$ corresponds to the composition of homotopy classes.

 We consider {the 
  $G$}-Hilbert space 
 $  H' \coloneqq \C\oplus \ell^{2}\otimes L^{2}(G)$ and the associated $G$-$C^{*}$-algebra $ K' \coloneqq K(  H')$ of compact operators. The equivariant inclusion $\C\to  H'$   induces an inclusion $ d:\C\to  K'$ in $ \Fun(BG,\nCalg_{\sepa})$, where $ \C$ has the trivial $G$-action. We also have a canonical inclusion $e \colon K\to  K'$. 
 
 Let $\pi_{ A} \colon q_{s}( A)\to  A\otimes  K$ be the  restriction of  {$\id_{ A\otimes  K}\sqcup 0 $} to $q_{s}( A)$.  It is shown in  \cite[Sec.\ 6]{Meyer:aa} that $\pi_{ A}$ is a ${\kkGsk}$-equivalence. 
 For every $ A$ in $ \Fun(BG,\nCalg_{\sepa})$  we have a  zig-zag    \begin{equation}\label{rqfhiowfjiowefqwefqwfewqfqwef}
q_{s}( A)\otimes  K\stackrel{\pi_{ A}\otimes \id_{ K}}{\to} { A}\otimes  K\otimes  K \stackrel{\id_{ A}\otimes e\otimes \id_{ K}}{\to}  A\otimes  K' \otimes  K\stackrel{\id_{ A}\otimes d\otimes \id_{ K}}{\leftarrow}  A\otimes  K \stackrel{\id_{ A}\otimes e }{\to} A\otimes  K'\stackrel{\id_{ A}\otimes d\otimes \id_{ K}}{\leftarrow}   A
\end{equation} 
which is natural in  $ A$. By the 
$\mathbb{K}^{G}$-stability of $\kkGsk$ all these maps are ${\kkGsk}$-equivalences.
  We now consider the   diagram  
\begin{align}
\label{ewgregergggwggerg}\\
\notag
\mathclap{
\xymatrix{
\Hom_{\cD}( A, B)\ar@{-}[d]_{\eqref{rqfhiowfjiowefqwefqwfewqfqwef}}^{\cong}& \Hom_{\Fun(BG,\nCalg_\sepa)}(  A    ,  B )\ar[l]_-{F  } \ar[r]^-{\kkGsk}&\KKth^{G}( A, B)\ar@{-->}@/_1.5cm/[ll]_{\bar F_{A,B} }\ar@{-}[d]_{\eqref{rqfhiowfjiowefqwefqwfewqfqwef}}^{\cong}\\
\Hom_{\cD}(q_{s}( A)\otimes  K, q_{s}( B)\otimes  K)&\ar[l]^-{!!!}_-{F} \Hom_{\Fun(BG,\nCalg_\sepa)}(q_s( A)\otimes  K,q_s( B) \otimes  K)\ar[r]_-{!}^-{\kkGsk}\ar@{..>}[d]& \KKth^{G}(q_s( A)\otimes  K,q_s( B) \otimes  K) \ar@{-->}@/_-3cm/[ll]^{!!}_{\bar F_{{q_s( A)\otimes  K,q_s( B)\otimes  K}}}\\
&[q_s( A)\otimes  K,q_s( B) \otimes  K] \ar@{..>}[ul]^{F'} \ar@{..>}[ur]^{ \eqref{qfrojqoiefjqewfefefqfqewfqewf}}_{\cong} &
}
}
\end{align}
The left  (or right) vertical isomorphism is induced by applying $F$ (or $\kkGsk$) to the zig-zag    \eqref{rqfhiowfjiowefqwefqwfewqfqwef} using the fact that $F$ (or $\kkGsk$) sends $\kkGsk$-equivalences to isomorphisms. If $\bar F$ exists, then the outer square 
and the two triangles involving  dashed and bold arrows
commute.

Using the lower right triangle we see that the arrow marked by $!$   
is surjective.
This shows that  the {arrow} 
marked by $!!$ is uniquely determined if it exists. 

Since $F$ inverts $\KKGsk$-equivalences, it also sends homotopic maps to equal maps. 
Consequently, the arrow $F $ marked by $!!!$
  factorizes uniquely through the dotted arrow  
$F' $  as indicated. 
We are therefore forced to define 
$$\bar F_{A,B} : \KKth^{G}( A, B)\to \Hom_{\cD}( A, B)$$ as the composition
going clockwise from the upper right-corner to the upper left corner.
The whole diagram is bi-natural in $ A$ (contravariant) and $ B$ (covariant). 
This implies that
the family $(\bar F_{ A, B})_{ A, B\in \Fun(BG,\nCalg_\sepa) }$
  provides the action of the desired functor  $\bar F$ on morphisms.  The triangle \eqref{qwefqknkjllnvafvsdfvfv} then commutes by construction.
%
\end{proof}

The following universal property of $\kkGsk$ (which differes from the one in Proposition  \ref{eoigjeoigergergwergergegergweg}) was
shown  in \cite[Thm.\ 2.2]{Thomsen-equivariant}, see also
\cite[Thm.\ 6.6]{Meyer:aa}. A
  functor from  $\Fun(BG,\nCalg_\sepa) $ to an additive category is called split exact if it   preserves split-exact sequences, see also Definition \ref{qwroigjqrwgqwrfqewfqewfq}.\ref{qeriughqeiruferwqkfjqr9ifuq2}.
%
%
%
\begin{prop}\label{wtohiwthwgrgwerg}
The functor
\[\kkGsk:\Fun(BG,\nCalg_\sepa) \lto \KKG_0\]  
is initial among all functors to an additive $1$-category which  {are reduced,  
$\mathbb{K}^{G}$-stable  and split exact.  }
\end{prop}

\begin{rem}\label{rem:homotopy-invariance}
The statement of Meyer requires additivity  {and homotopy invariance} of the functors. But note that  the condition  {of being }split exact and additive  {are equivalent to the condition of being split exact and reduced.} {Furthermore, we use that 
split-exactness and $\mathbb{K}^{G}$-stability  together imply homotopy invariance (see \cite{higsa} for the non-equivariant case). We refer to \cite[Thm. 3.35]{cmr} for an argument which applies {verbatim} to the equivariant case.}
%
\hB
\end{rem}

\begin{kor}\label{ioqerjgqergrqfefqewfq}
A functor from $\Fun(BG,\nCalg_{\sepa})$ to an additive  
{$\infty$-}category which  {is reduced}, 
$\mathbb{K}^{G}$-stable and split exact sends ${\kkGsk}$-equivalences to {equivalences}.
\end{kor}
\begin{proof}
{Let $F:\Fun(BG,\nCalg_{\sepa})\to \cD$ be a functor  as in the statement of the corollary.
Since $\ho(\cD)$ is an additive $1$-category we can apply Proposition \ref{wtohiwthwgrgwerg} to the composition 
$ \Fun(BG,\nCalg_{\sepa})\stackrel{F}{\to} \cD\to \ho(\cD)$ in order to conclude that it factorizes over $\kkGsk$.   In particular, it sends ${\kkGsk}$-equivalences to isomorphisms.
Since the canonical functor $\cD\to \ho(\cD)$ detects equivalences we conclude that 
$F$ sends ${\kkGsk}$-equivalences to equivalences.} 
\end{proof}

\begin{rem}\label{weroigujweogerwgregfw}  Let $\cC$ be a small $\infty$-category and $W$ be a set of morphisms in $\cC$. Then we can form
the Dwyer--Kan localization  \cite[Def.~1.3.4.1, Rem.~1.3.4.2]{HA}
\begin{equation}\label{qewfqwefqffqwfqfqew}
\ell \colon \cC\to \cC[W^{-1}]\, .
\end{equation} 
It is characterized by the universal property that for any $\infty$-category $\cD$ the restriction along $\ell$ induces an equivalence \begin{equation}\label{erwgoijowiergwergwregerg}
\Fun( \cC[W^{-1}],\cD)\xrightarrow{\simeq} \Fun^{W}(\cC,\cD)\,,
\end{equation}
where $ \Fun^{W}(\cC,\cD)$ is the full subcategory of 
$ \Fun(\cC,\cD)$ of functors which send morphisms in $W$ to equivalences. \hB
\end{rem}

Note that we consider ordinary categories as $\infty$-categories using the nerve, but we will not write the nerve explicitly.   
With these conventions we can consider 
the functor
$$\kkGs:\Fun(BG,\nCalgs)\to \KKGs$$ introduced in Definition \ref{wtpohwopggergeg}.


We now start with the proof of Theorem \ref{weighigregwgregw9} following the lines of  \cite{Land:2016aa} using \cite{Meyer:aa} and \cite{Uuye:2010aa}. 
First of all it follows from the universal property of the Dwyer--Kan localization that  there exists a factorization of $\kkGsk$ over a functor $\ho$ as indicated by the following commuting triangle: 
\begin{equation}\label{fewpoqkopfqewffwqwefqf}
\xymatrix{
\Fun(BG,\nCalg_{\sepa})\ar[rr]^-{\kkGsk}\ar[rd]_-{\kkGs}&&\KKGsk\,.\\
&\KKGs\ar[ur]_-{\ho}&
}
\end{equation} 
This is the commutative triangle in  \eqref{fewpoqkopfqewffwqwefqf1}.

In order to show that ${\KKGs}$ is stable we consider the functor
\begin{equation}\label{qwefewfwefwfqwefeewqfe}
S \colon \Fun(BG,\nCalg_{\sepa})\to \Fun(BG,\nCalg_{\sepa})\, , \quad  A\mapsto C_{0}((0,1))\otimes  A
\end{equation}
which later turns out to be an explicit model for the suspension functor in ${\KKGs}$.

\begin{lem}\label{qefoijhfoiqfejoewfqwef}\mbox{}
\begin{enumerate}
\item \label{weiogjwergwergwegrgwgreg} The functor $S$ uniquely descends to an equivalence $S_{0}$ of additive categories   such that  
\begin{equation}\label{qevlkqmlkqwvqwrv}
 \xymatrix{\Fun(BG,\nCalg_{\sepa})\ar[r]^{S}\ar[d]^{\kkGsk}&\Fun(BG,\nCalg_{\sepa}) \ar[d]^{\kkGsk}\\ \KKGsk\ar@{..>}[r]^{S_{0}}&\KKGsk}
\end{equation} commutes. 
\item \label{qeroigjqergegwegre}The functor 
$S$    essentially uniquely  descends to a functor (also  denoted by) $S$ completing the square \begin{equation}\label{qefljoiewfqwefqewfwefqf}
\xymatrix{\Fun(BG,\nCalg_{\sepa})\ar[d]^{\kkGs}\ar[r]^{S}&\Fun(BG,\nCalg_{\sepa}\ar[d]^{\kkGs}\\ \KKGs\ar@{..>}[r]^{S}&\KKGs}
\end{equation}
 \end{enumerate}
\end{lem}

\begin{proof}
In order to obtain $S_{0}$,
we apply {Proposition \ref{wtohiwthwgrgwerg}} 
to the composition $$\kkGsk\circ S \colon \Fun(BG,\nCalg_{\sepa})\to \KKGsk \, .$$ 
 It straightforward to check that the functor   $\kkGsk\circ S$  is  {reduced,} 
 $\mathbb{K}^{G}$-stable and split exact, using that $\kkGsk$ has these properties.
  {The functor $\kkGsk\circ S$}
  therefore factorizes over an additive endofunctor  
   $S_{0}$ as required.  It  {then} follows from Bott periodicity that $S_{0}$ is an equivalence of categories.
   
 Since  {by definition} $\kkGsk$ detects ${\kkGsk}$-equivalences, 
 we conclude from Assertion   \ref{weiogjwergwergwegrgwgreg} that $S$ {in \eqref{qwefewfwefwfqwefeewqfe}} preserves $\kkGsk$-equivalences. Applying the universal property of $\kkGs$  being a Dwyer-Kan localization (see Definition \ref{wtpohwopggergeg}) we obtain the essentially unique factoriaztion of $\kkGs\circ S$ through $\kkGs$ as asserted in Assertion \ref{qeroigjqergegwegre}. 
 \end{proof}

As shown in \cite{kasparovinvent}, for every semisplit  exact sequence
  \begin{equation}\label{fqwefwoijoiqewfeqwfewqfewfqfef1111}
0\to  I\to  A\to  B \to 0 
\end{equation}  we have    a boundary operator
$\partial$ in $ \KKth^{G}(S( B ), I)$ such that
the sequence \begin{equation}\label{qwfeewfqefwfeqwf}
 \KKth^{G}( D,S( A))\to  \KKth^{G}( D,S( B))\stackrel{\partial}{\to}  \KKth^{G}( D, I)\to  \KKth^{G}( D, A)\to  \KKth^{G}( D, B)
\end{equation}
  is exact for every $ D$ in $\Fun(BG,\nCalg_{\sepa})$. 
  The boundary operator  is natural in the sequence in the following sense.
If $$\xymatrix{0\ar[r]& I\ar[d]\ar[r]& A\ar[r]\ar[d]& B\ar[r]\ar[d]&0\\
0\ar[r]& I'\ar[r]& A'\ar[r]& B'\ar[r]&0}$$
is a morphism to a second sequence of this kind, then the square
$$\xymatrix{ \KKth^{G}( D,S( B))\ar[r]^{\partial}\ar[d]&\KKth^{G}( D, I)\ar[d]\\\KKth^{G}( D,S( B'))\ar[r]^{\partial'}&\KKth^{G}( D, I')}$$
commutes.  
%

The main problem in the proof of   Theorem \ref{weighigregwgregw9} is   to control  finite limits  in $\KKGs$. 
To this end we calculate the Dwyer-Kan localization in Definition \ref{wtpohwopggergeg} using  a category of fibrant objects structure on $\Fun(BG,\nCalgs)$.
We now  recall the definition of a category of fibrant objects in the form  used in  \cite[Def.\ 1.1]{Uuye:2010aa}.
Let $\cC$ be a category with a terminal object, and let $F $ and $W$ be collections of morphisms in $\cC$.
\begin{ddd}
The triple $(\cC,F,W)$ is a category of fibrant objects if
the following conditions hold.
 \begin{enumerate}
\item F0: $F$ is closed under compositions.
\item F1: $F$   contains all isomorphisms.
\item F2: Pull-backs of morphisms in $F$ exist and belong again to $F$.
\item F3: For every $C$ in $\cC$ the morphism $C\to *$ belongs to $F$.
\item W1: $W$   contains all isomorphisms.
\item W2: $W$ has the $2$-out-of-$3$-property.
\item FW1: $W\cap F$ is stable under forming pull-backs.
\item  FW2: For every $C$ in $\cC$ the morphism
$C\to C\times C$  has a factorization
$C\stackrel{w}{\to} C^{I}\stackrel{f}{\to} C\times C$, where $w$ is in $W$ and $f$ is in {$F$.}
\end{enumerate}
\end{ddd}
Note that  the product in the formulation of Condition  FW2 exists by a combination of Conditions F3 and F2.

\begin{rem}\label{factorization}
The object $C^I$ appearing in FW2 is called a path object. 
  {From the existence of a path object $C^I$ }
one can construct a factorisation of an arbitrary morphism $\phi\colon A \to  C$ in $\cC$ as   a weak equivalence followed by a fibration as follows.   We choose a factorization $C\stackrel{w}{\to} C^{I}\stackrel{f}{\to} C\times C$ as in $FW2$. Then we 
consider the commut{ative} diagram
\begin{equation}\label{fqwfojofqweqwefwefwef}
\begin{tikzcd}
	A\ar[r,"u"] \ar[ddr,"\phi"] &A \times_{C} C^{I} \ar[r] \ar[d,"f'"] & C^I \ar[d,"f"] \\&
	A \times C\ar[d,"\pr"] \ar[r,"{\phi \times \id}"] & C \times C\\&C&
 \end{tikzcd}
\end{equation} 
 where the square is a pull-back.  It exists by $F2$ which also implies that $f'$ is a fibration.  Using $F3$ and $F2$ one checks that the projection $\pr$ is also a fibration. We conclude by $F0$ that $\pr\circ f'$ is a fibration.
Furthermore, the map $u$ 
 is given by $ A\cong A\times_{C}C\stackrel{A\times w}{\to} A\times_{C}
C^{I}$ and therefore a weak equivalence by $FW1$ and $F3$ (applied to $A\to *$). 
Consequently
\begin{equation}\label{ergfgferqfqfqef}
A\stackrel{u}{\lto} A\times_{C}C^{I}\stackrel{\pr\circ f'}{\lto} C
\end{equation}
is the  desired factorization of $\phi$ as a composition of a weak equivalence and a fibration. \hB
\end{rem}

\begin{rem}\label{cylinder object}


The category $\Fun(BG,\nCalg_\sepa)$ admits a canonical path object  $ C\mapsto C([0,1], C)$. The factorization
\[  C \lto C([0,1], C) \lto  C \times  C\] 
required in $FW2$ is given by the inclusion of constant maps and the evaluation at $0$ and $1$.  
The first map is a homotopy equivalence. Furthermore,
the map $ C^{I}\to  C\times  C$ is surjective and admits an  {equivariant}  contractive completely positive split. For example, one can take the split 
\begin{equation}\label{ewfqewffqeff}
(c_{0},c_{1})\mapsto (t\mapsto (1-t)c_{0}+tc_{1})\, .
\end{equation}

Applying the general factorization \eqref{ergfgferqfqfqef} to a  morphism $ A \to  C$, one obtains the usual cone sequence
\[ 0\lto C( \phi) \lto \mathrm{Cyl}( \phi) \lto  C\to 0 \, .\]
Explicitly,
\begin{equation}\label{qewfpojkopqwefqewf}
\mathrm{Cyl}(  \phi) \coloneqq \{(a,c)\in A\oplus  C([0,1],C)\:|\: c(0)= \phi(a)\}\, , 
\end{equation}  the  homomorphism $\mathrm{Cyl}(  \phi)\to   C$ is given by $(a,c)\mapsto c(1)$, and  
\[C( \phi):= \{(a,c)\in \mathrm{Cyl}(  \phi)\:|\: c(1)=0\} \]
with the induced $G$-actions. 
\hB
\end{rem}

A morphism $ \phi \colon A \to  C$ in $\Fun(BG,\nCalg_\sepa)$ is  {called }a semisplit surjection if it is a surjection and
admits an equivariant cpc split.
We let $\SS$ denote the set of semisplit surjections{.}  Furthermore recall that $ \kkGeq$ is the  set of ${\kkGsk}$-equivalences.

\begin{prop}\label{prop:fibration-structure}
The triple $(\Fun(BG,\nCalg_\sepa),\SS , \kkGeq)$ is a category of fibrant objects.
\end{prop}
\begin{proof}
The axioms $F0$,  $F1$, $F3$, $W1$, and $W2$ are obvious.

We  show $F2$. We consider a pullback 
\begin{equation}\label{qweflkjqlekwfqwfqwefefqefe}
\xymatrix{ A\ar[r]\ar[d]^{ \psi}& C\ar[d]^{ \phi}\\  B\ar[r]& D}
\end{equation}
in $\Fun(BG,\nCalg_\sepa)$ such that $ \phi$ is in $\SS$ with equivariant cpc split $s:D\to C$. The square of the underlying   Banach spaces is a pullback  square of   Banach spaces. Thus the  equivariant  contractive split $s $ of $\phi$ induces an equivariant  contractive split $t:B\to A$.  We now use the fact that an element in $A$ is positive if its images in $C$ and $B$ are positive. The same is true for the extension of the square to finite matrices. Using that $s$ is completely positive we can now conclude that $t$ is completely positive as well. Hence $ \psi$ is in $\SS$, too.

To show $FW2$, we make use of the cylinder discussed in Remark~\ref{cylinder object}.  The map $ C \to  C^I$ is  a homotopy equivalence, hence    in particular a ${\kkGsk}$-equivalence. Furthermore, as explained in Remark~\ref{cylinder object} the projection $ C^{I}\to  C\times  C$ has an equivariant  contractive completely positive split and therefore belongs to $\SS$.

 In order to prepare the proof of   $FW1$ we consider   an  exact sequence admitting an equivariant  cpc split.
\begin{equation}\label{fqwefwoijoiqewfeqwfewqfewfqfef}
0\lto  I\lto  A\stackrel{ \psi}{\lto}  B\lto 0 \, .
\end{equation} 
As explained above,  
then we have  the  boundary operator 
$\partial$ in $\KKth^{G}(S( B ), I)$ 
such that
the  sequence
 \begin{equation}\label{qefefefefqewfqef}
\KKth^{G}(D,S( A))\to \KKth^{G}(D,S(  B ))\stackrel{\partial}{\to} \KKth^{G}( D, I)\to \KKth^{G}( D, A)\stackrel{ \psi_{*}}{\to} \KKth^{G}( D, B )
\end{equation} is exact for every $ D$ in $\Fun(BG,\nCalg_{\sepa})$.
By the Yoneda Lemma, the morphism $ \psi$ is a ${\kkGsk}$-equivalence if and only if $ \psi_{*}$ is an isomorphism for every  $ D$ in $\Fun(BG,\nCalg_\sepa)$.
We conclude from the exactness of \eqref{qefefefefqewfqef} that  this is the case  if and only if $\KKth^{G}( D, I)=0$  for every $ D$.

{We now consider a pull-back square \eqref{qweflkjqlekwfqwfqwefefqefe} in 
$\Fun(BG,\nCalg_\sepa)$ such that $ \phi$ belongs to $\SS  \cap \kkGeq$. By $F2$ we know that $ \psi$ is in $\SS$ and it remains to show that $ \psi$ belongs to $\kkGeq$. This follows from the fact that $\ker( \phi)\cong \ker( \psi)$. 
Since $  \phi$ is in $\kkGeq$ we have 
$\KKth^{G}( D, \ker( \psi))= \KKth^{G}( D, \ker( \phi))=0$ for all $ D$ in $\Fun(BG,\nCalg_\sepa)$. Consequently, $ \psi$ belongs to   $\kkGeq$.}
%
%
\end{proof}

For the formulation of universal properties  it turned out to be useful to reformulate various exactness   properties of functors defined on $G$-$C^{*}$-algebras  using squares so that they make sense not only for stable targets, but also for pointed ones.  But we will show in Lemma \ref{eroigjowegwerreggwerg} below that these new definitions are equivalent  to the old  same-named 
conditions in the situations where they have been used above.
We do these considerations at this point since we will use squares already in the proof of Proposition \ref{KK-stable} below.

We consider a cartesian square 
 \begin{equation}\label{f23rfipojoierfwefrwefwerf}
 \xymatrix{ A\ar[r]\ar[d]^-{ p}& B\ar[d]^-{ q}\\
 \ar[r] C& D}
\end{equation}
in $\Fun(BG,\nCalg)$. We will consider the following two 
conditions on the  
square in \eqref{f23rfipojoierfwefrwefwerf}.

%
%

 \begin{ddd}\label{qoirgqrfqfefqf}\mbox{} \begin{enumerate}
\item   \label{qroifjoqrfewewfqf}
The square   \eqref{f23rfipojoierfwefrwefwerf} is called  semisplit if there exist   equivariant     cpc's   $ s \colon  C\to  A$ and $ t \colon  D\to  B$ such that $ p\circ  s=\id_{ C}$ and $ q\circ  t=\id_{ D}$.

\item \label{wtiogjergergerwgergwerg}
 The square   \eqref{f23rfipojoierfwefrwefwerf} is called  split if there exist equivariant 
 morphisms    $ s \colon  C\to  A$ and $ t \colon  D\to  B$ such that $ p\circ  s=\id_{ C}$ and $ q\circ  t=\id_{ D}$.
\end{enumerate}
\end{ddd}

\begin{rem}
The conditions on $ q$ appearing in Definition~\ref{qoirgqrfqfefqf} imply the corresponding conditions on $ p$, see the verification of the axiom F2 in the proof of Proposition \ref{prop:fibration-structure} for a proof of this {fact.}
\hB
\end{rem}

Note that an exact sequence 
\begin{equation}\label{f23rfipojoierfwefrwefwerfq}
0\to   I\to A\xrightarrow{ q}  B \to 0 
\end{equation}
in $\Fun(BG,\nCalg)$
is semisplit  or split, if and only if the cartesian square 
\begin{equation}\label{vsadvlknalskdvasdvsd}
\xymatrix{ I\ar[r]\ar[d]& A\ar[d]^-{q}\\0\ar[r]&\tilde B}
\end{equation}  
has the corresponding property as defined in Definition \ref{qoirgqrfqfefqf}. 

Let $F \colon \Fun(BG,\nCalg_{\sepa})\to \cD$ (or $F\colon \Fun(BG,\nCalg)\to \cD$) be a 
functor to a pointed $\infty$-category $\cD$. 
Corresponding to the list of conditions in Definition \ref{qoirgqrfqfefqf} we introduce  the following  two  exactness conditions on the functor $F$.


\begin{ddd}\label{qwroigjqrwgqwrfqewfqewfq}\mbox{}
\begin{enumerate}

\item \label{qeriughqeiruferwqkfjqr9ifuq1}
$F$ is semiexact 
if it   is reduced and   sends any    semisplit    {cartesian square} 
to a {cartesian square}  in $\cD$.
 \item  \label{qeriughqeiruferwqkfjqr9ifuq2} F is split exact 
if it   is reduced and sends  any split  cartesian square 
to a cartesian square in $\cD$.
\end{enumerate}
\end{ddd}

%
%
%
%
%
%
   
 Let $F$ be a  functor as in Definition \ref{qwroigjqrwgqwrfqewfqewfq}.    
 \begin{lem}\mbox{}\label{eroigjowegwerreggwerg}
 \begin{enumerate}
\item \label{qeoigwergwerggergewrgwergrwergwg}  If $\cD$ is stable, then  $F$ is semiexact in the sense of Definition  \ref{qwroigjqrwgqwrfqewfqewfq}.\ref{qeriughqeiruferwqkfjqr9ifuq1}  if and only if $F$ sends all  semisplit exact sequences to  fibre sequences.
\item \label{qeoigwergwergwergwg} If $\cD$ is additive, then  $F$ is  split exact   in the sense of Definition  \ref{qwroigjqrwgqwrfqewfqewfq}.\ref{qeriughqeiruferwqkfjqr9ifuq2}  if and only if
 $F$   sends  all  split exact sequences to  fibre sequences.
 \end{enumerate}
  \end{lem}
\begin{proof}
In each case, the ''only if`` implication is obvious. To see the ''if`` implication, we extend the square \eqref{f23rfipojoierfwefrwefwerf}  to the left by adding a pullback of $ p$ along $0 \to  C$:
\begin{equation*}\label{rqgrwefwqfewfw}
\begin{tikzcd}
	  I  \ar[r] \ar[d] &  A  \ar[r] \ar[d," p"] &   B \ar[d," q"] \\
	0 \ar[r] &  C \ar[r] &  D
\end{tikzcd}
\end{equation*}
The left and the outer squares are exact sequences which are send by $F$ to fibre sequences.
Thus applying $F$, {in both cases} 
  we obtain the diagram
\begin{equation}\label{F-diagram}
\begin{tikzcd}
	F( I) \ar[r] \ar[d] & F( A) \ar[r] \ar[d,"F( p)"] & F( B) \ar[d,"F( q)"] \\
	0 \ar[r] & F( C) \ar[r] & F( D)
\end{tikzcd}
\end{equation}
in which   the left  and the outer  squares are cartesian and the lower left corner is a zero object. In particular we can  conclude that $F$ is reduced. The right square is cartesian since $F(I)$ identifies with the fibre of both $F(p)$ and $F(q)$. {The stable case follows at once, and in the additive case we use the existence of the splits of $F(p)$ and $F(q)$ to conclude}.

%
%
%
%
%
%
\end{proof}

The following proposition  shows all  assertions of  Theorem \ref{weighigregwgregw9} except the one about triangulated structures which will be obtained  later in Proposition \ref{woiejgwoegregreggwergw}.

\begin{prop}\label{KK-stable}\mbox{}
\begin{enumerate}
	\item \label{qoifujoqfqewfewfqewfe} The functor $\kkGs$  is   semiexact.
 	\item \label{qoifujoqfqewfewfqewfe2}The   functor  $\ho_0 \colon \ho(\KKGs) \to \KKGsk$ is an equivalence.\item\label{qoifujoqfqewfewfqewfe1}  The $\infty$-category $\KKGs$ is stable. 
\end{enumerate}
\end{prop}
\begin{proof} 
Assertion \ref{qoifujoqfqewfewfqewfe} is an immediate consequence of the fact that $\KKGs$ is the $\infty$-category associated to {a  category of fibrant objects}  whose fibrations are the semisplit exact surjections. Indeed, consider a semisplit cartesian square    \eqref{f23rfipojoierfwefrwefwerf}. Since $ p$ and $  q$ are fibrations this square represents a cartesian square in the $\infty$-category $\KKGs$, see \cite[Prop.\ 7.5.6]{Cisinski:2017}.

We now prove Assertion \ref{qoifujoqfqewfewfqewfe2}. 
We will use that forming  localizations is compatible with going over to homotopy categories.
Let $\cC$ be an $\infty$-category with a set of morphisms $W$. Then  
  have a commutative square
  \begin{equation}\label{vwervevwgretgberter}
\xymatrix{\cC\ar[r]^-{\ell}\ar[d]&\cC[W^{-1}]\ar[d]\\\ho(\cC)\ar[r]^-{\ho(\ell)}&\ho(\cC[W^{-1}])}
\end{equation}
 where $\ho(\ell) \colon \ho(\cC)\to   \ho(\cC[W^{-1}])$ presents its target as the  localization of $\ho(\cC)$ at the set $\ho(W)$ in the sense of ordinary categories.
We apply this to $\Fun(BG,\nCalg_{\sepa})$ in place of $\cC$. Since $\Fun(BG,\nCalg_{\sepa})$ is an ordinary category the left vertical functor in \eqref{vwervevwgretgberter} is an equivalence. Hence 
the functor 
$$\Fun(BG,\nCalg_{\sepa})\xrightarrow{\kkGs} \KKGs \to \ho(\KKGs)$$
presents its target as the localization of 
 $\Fun(BG,\nCalg_{\sepa})$ at the set of ${\kkGsk}$-equivalences in the sense of ordinary categories. By Proposition 
 \ref{eoigjeoigergergwergergegergweg} the  functor  $$\kkGsk \colon \Fun(BG,\nCalg_{\sepa})\to \KKGsk$$ has the same universal property. 
 This implies that $\ho_{0}\colon \ho(\KKGs)\to \KKGsk$ is an equivalence.

The proof Assertion \ref{qoifujoqfqewfewfqewfe1}  that $\KKGs$ is stable can now be copied from \cite[Prop.\ 3.3]{Land:2016aa}.  First we note that $\KKGs$, being the $\infty$-category associated to a category of fibrant objects with a zero object, is pointed and has finite limits. 
It remains to show that the loop functor {$\Omega$} in $\KKGs$ is an equivalence. It is well-known that it suffices to show that  $\Omega$ induces an equivalence on the homotopy category, see for instance \cite[Lem.\ 3.4]{Land:2016aa}.

Using the explicit description of the fibrant replacement of $0\to  A$ given in Remark \ref{cylinder object}  we see that
 \begin{equation}\label{}
\xymatrix{\kkGs(S( A))\ar[r]\ar[d]&0\ar[d]\\0\ar[r]&\kkGs( A)} 
\end{equation}
is a pull-back in $\KKGs$. Therefore  $\Omega$ is {equivalent to  
the suspension functor given by the dotted arrow $S$ in \eqref{qefljoiewfqwefqewfwefqf}.  
In view of Assertion \ref{qoifujoqfqewfewfqewfe2} shown above  the induced action of this functor $S$ on the homotopy category can be identified with the action of the functor $S_{0}$ in Lemma~\ref{qefoijhfoiqfejoewfqwef}.\ref{qeroigjqergegwegre} on $\KKGsk$. As asserted in the same statement $S_{0} $   is an equivalence.} \end{proof}

Our next task is the verification of the properties of $\KKGs$ and $\kkGs$ stated in Theorem \ref{qroifjeriogerggergegegweg}.
Assertion  \ref{qroifjeriogerggergegegweg}.\ref{qoirwfjhqoierggrg1} is already shown as a part of Proposition \ref{KK-stable}.


The next Lemma settles Assertions  \ref{qroifjeriogerggergegegweg}.\ref{qoirwfjhqoierggrg0},   \ref{qroifjeriogerggergegegweg}.\ref{qoirwfjhqoierggrg3} and  \ref{qroifjeriogerggergegegweg}.\ref{qoirwfjhqoierggrg4} together.

\begin{lem}
The functor $\kkGs$ is reduced, $\mathbb{K}^{G}$-stable, and homotopy invariant.
\end{lem}
\begin{proof} We know that 
  $\kkGsk$ is reduced, {$\mathbb{K}^G$}-stable and homotopy invariant.  Therefore 
Proposition \ref{KK-stable}.\ref{qoifujoqfqewfewfqewfe2} together with the fact that  the canonical functor from an $\infty$-category to its homotopy category detects equivalences and zero elements  implies the  assertion.
\end{proof}

The next lemma collects, for later reference, some simple consequences of the details proof of Proposition \ref{KK-stable}.   \begin{lem}\label{woitjgorwergergegergw} \mbox{}
\begin{enumerate}
\item\label{woifhjoqiefewqewfqef} Every morphism in $\KKGs$ is equivalent to a morphism $\kkGs(f)$ for a fibration $f$ in $\Fun(BG,\nCalg_{\sepa})$.  
\item\label{ergpoggergwegregweg}  {Every product in $\KKGs$ is equivalent to the image under $\kkGs$ of  a product in $\Fun(BG,\nCalg_{\sepa})$.}
\item \label{woifhjoqiefewqewfqef1} 
Every fibre sequence in $\KKGs$ is equivalent to the image under $\kkGs$ of a cone sequence.
\end{enumerate}
 \end{lem}
\begin{proof}  These assertions   follow from general facts about the associated $\infty$-category of a category of fibrant objects, where for 
{Assertion} \ref{woifhjoqiefewqewfqef1} we use in addition that  $\KKGs$ is stable. 
{For completeness we give  alternative arguments  which are specific to the  present situation as we will use some of the details later.}

{Assertion \ref{woifhjoqiefewqewfqef} follows from the 
equivalence $\ho_{0}\colon \ho(\KKGs)\stackrel{\simeq}{\to} \KKGsk$  
and the surjectivity of the arrow marked by $!!$ in \eqref{ewgregergggwggerg}.}

For Assertion \ref{ergpoggergwegregweg} we note that any
  object in $\KKGs$ is equivalent to an object of the form $\kkGs(   A)$ for some $   A$ in $\Fun(BG,\nCalg_{\sepa})$. Let $   \kkGs(   B)$ be a second object.  
   Then  the square
  \begin{equation*}
\xymatrix{   A\oplus    B\ar[r]\ar[d]&   A\ar[d]\\    B\ar[r]&0}
\end{equation*} 
     is  a {split} cartesian square in $\Fun(BG,\nCalg_{\sepa})$ (Definition \ref{qoirgqrfqfefqf}.\ref{wtiogjergergerwgergwerg}).
  Since
  $\kkGs$ is semiexact  by  Prop. \ref{KK-stable}\ref{qoifujoqfqewfewfqewfe}   (and hence split exact)
 we conclude that  the square 
  \begin{equation}\label{ertwoijeoigwergeggwg}
\xymatrix{\kkGs(   A\oplus    B)\ar[r]\ar[d]&\kkGs(   A)\ar[d]\\ \kkGs (   B)\ar[r]&0}
\end{equation} is cartesian.
 This implies that
 \begin{equation}\label{ergverwfwerferf}
\kkGs(   A)\times  \kkGs (   B)  \simeq \kkGs(   A\oplus    B)\, .
\end{equation}

   The assertion  now follows from  $$\kkGs(A\times B)\simeq \kkGs(A\oplus B)\stackrel{\eqref{ergverwfwerferf}
}{\simeq}   \kkGs(A)\times \kkGs(B)\, ,$$
 where we use that the product of $C^{*}$-algebras is given by the sum. 
     

We {finally} show Assertion \ref{woifhjoqiefewqewfqef1}.  Since $\KKGs$ is stable, a morphism in $\KKGs$ 
can be extended to a fibre sequence, and every fibre sequence is obtained in this way up to equivalence.  By  Assertion \ref{woifhjoqiefewqewfqef} it suffices to consider fibre sequences obtained by extending the morphism $\kkGs(   p)$ for a morphism $   p:   A\to    B$ in $\Fun(BG,\nCalg_{\sepa})$. 
We next argue that one can further assume that $   p$ is a fibration. To this end we use 
  the cylinder defined in \eqref{qewfpojkopqwefqewf}  in order to construct the  commut{ative} square
\[  \xymatrix{
	     A \ar[r]^{   p} \ar[d]^{   e} &    B \ar@{=}[d] \\
	  \mathrm{Cyl}(   p) \ar[r]^{   q} &    B}
 \]
 in $ \Fun(BG,\nCalg_{\sepa})$, where $   q(a,b)=b(1)$ and $   e(a)=(a,\const_{   p(a)})$. The morphism $   e$ is a ${\kkGsk}$-equivalence since it is an instance of the morphism $u$ in \eqref{fqwfojofqweqwefwefwef}.  Therefore 
   the two morphisms $\kkGs(   p)$ and $\kkGs(   q)$ are equivalent.

Note that $   q$ is surjective and that the exact cone sequence  
\begin{equation}\label{vsfdvfsdvfvsvfdv}
0\to  C(   p) \to  \mathrm{Cyl}(   p)\stackrel{   q}{\to}    B\to 0
\end{equation}  is   semisplit.  Since $\kkGs$ is semiexact by Proposition~\ref{KK-stable}.{\ref{qoifujoqfqewfewfqewfe}},  
it  sends this sequence to a fibre sequence. {This finishes the proof of  Assertion \ref{woifhjoqiefewqewfqef1}}. \end{proof} 

As a homotopy  category of a stable $\infty$-category,  the 
category $\ho(\KKGs)$ acquires   a triangulated structure
 \cite[Thm.~1.1.2.14]{HA}. On the other hand $\KKGsk$ has a triangulated structure 
 described in \cite[Sec. 2.1]{MR2193334}.
 
\begin{prop}\label{woiejgwoegregreggwergw}
The functor
$\ho_{0}\colon \ho(\KKGs)\to \KKGsk$ is an equivalence of triangulated categories.
\end{prop}
 \begin{proof}
 We know from Proposition \ref{KK-stable}.\ref{qoifujoqfqewfewfqewfe2} that $\ho_{0}$ is an equivalence of categories. 
 Since sums in $\ho(\KKGs)$ and $\KKGsk$ are represented by sums in $\Fun(BG,\nCalgs)$ we conclude further
 that $\ho_{0}$ preserves sums and is therefore compatible with the $\Ab$-enrichment.
 
 The inverse shift  functor for $\KKGs$ and therefore on $\ho(\KKGs)$ is implemented by the suspension functor \eqref{qwefewfwefwfqwefeewqfe} on the level of $C^{*}$-algebras via Lemma \ref{qefoijhfoiqfejoewfqwef}.\ref{qeroigjqergegwegre}.
 Similarly, by the description given in   \cite[Sec. 2.1]{MR2193334}   the    inverse shift  functor on $\KKGsk$
 is also implemented  by  \eqref{qwefewfwefwfqwefeewqfe}  via Lemma \ref{qefoijhfoiqfejoewfqwef}.\ref{qeroigjqergegwegre}.
 This shows that $\ho_{0}$ commutes with the  shift functor. 
 
  On the one hand, by  \cite[Sec. 2.1]{MR2193334} the exact triangles in $\KKGsk$ are generated by the mapping cone sequences for morphisms in  $\Fun(BG,\nCalgs)$. 
 On the other hand, by Lemma \ref{woitjgorwergergegergw}.\ref{woifhjoqiefewqewfqef1} 
 the fibre sequences in $\KKGs$, and hence the exact triangles in  $\ho(\KKGs)$
 are also generated by the mapping cone sequences for morphisms in  $\Fun(BG,\nCalgs)$. 
 We conclude that the functor $\ho_{0}$ is compatible with the collections of distinguished triangles.
    \end{proof}
\phantomsection \label{wetgijweogwegwerger} This finishes the proof of Theorem \ref{weighigregwgregw9}.

$\KKGs$ is stable by  Proposition \ref{KK-stable}.\ref{qoifujoqfqewfewfqewfe2}  and therefore admits all finite colimits. The following
proposition strengthens this  from finite to countable colimits and settles the remaining Assertions  \ref{qroifjeriogerggergegegweg}.\ref{qoirwfjhqoierggrg4}, \ref{qroifjeriogerggergegegweg}.\ref{qoirwfjhqoierggrg5} and \ref{qroifjeriogerggergegegweg}.\ref{qoirwfjhqoierggrg6}.
For the notion of an admissible diagram $\nat\to \Fun(BG,\nCalgs)$ we refer to \cite[Def.\ 2.5]{MR2193334}.  We will not repeat the definition here since  the exact  details are not  relevant. We  will only use 
its   consequence  \cite[Prop.\ 2.6]{MR2193334}.

\begin{lem}\label{ugoierugogergwergwregw}\mbox{}
\begin{enumerate}
\item\label{eroigjerwogwregwergwergwerg} The category $\KKGs$ admits all countable colimits and is  {therefore} 
 idempotent complete.
\item\label{regpoqgowergwergwregwregwreg} The functor $\kkGs$ sends countable sums in $\Fun(BG,\nCalg_{\sepa})$ to coproducts.
\item\label{ewrpoigjeogrewgregewf} The {functor} $\kkGs$ preserves colimits of all diagrams $A \colon \nat\to \Fun(BG,\nCalgs)$ which are admissible.
\end{enumerate}
\end{lem}
 \begin{proof}
Since ${\KKGs}$ is stable, in order to show that $\KKGs$ admits all countable colimits   it suffices to show that $\KKGs$ admits countable coproducts. The 
  functor $\kkGs$ is essentially surjective by construction. Hence it suffices to show that for  every countable family $( A_{i})_{i\in I}$ in $\Fun(BG,\nCalg_{\sepa})$ the family
$(\kkGs( A_{i}))_{i\in I}$ in $\KKGs$ admits a coproduct.
We consider the sum $ A \coloneqq \bigoplus_{i\in I} A_{i}$ in $\Fun(BG,\nCalg_{\sepa})$
and let $ e_{i} \colon A_{i}\to  A$ be the canonical inclusion for every $i$ in $I$.
Then we claim that
$(\kkGs( A),(\kkGs( e_{i}))_{i\in I})$ represents the coproduct of the family $(\kkGs( A_{i})_{i\in I})$.  In general, coproducts in a stable $\infty$-category can be detected on the homotopy category.  In view of the equivalence   $\kkGsk\simeq \ho\circ \kkGs$ given by \eqref{fewpoqkopfqewffwqwefqf1} it thus suffices to show that  
$(\kkGsk( A),(\kkGsk( e_{i}))_{i\in I})$ represents the coproduct in 
$\KKGsk$. But this follows from  \cite[Thm. 2.9]{kasparovinvent} stating that the family of maps $( e_{i})_{i\in I}$ induces an isomorphism
$$\KKGsk( A, B)\xrightarrow{\cong} \prod_{i\in I} \KKGsk( A_{i}, B)$$ for every $ B$ in $\Fun(BG,\nCalg_{\sepa})$. 
The proof of Assertion  \ref{eroigjerwogwregwergwergwerg} is finished with the general observation that a stable and countably cocomplete $\infty$-category is idempotent complete.

The  claim also  implies Assertion  \ref{regpoqgowergwergwregwregwreg}.

We finally show the Assertion \ref{ewrpoigjeogrewgregewf}. 
Let $A\colon\nat \to \Fun(BG,\nCalgs)$ be an admissible diagram.
We must show that the canonical map $\colim_{\nat} \kkGs(A)\to \kkGs(\colim_{\nat} A)$ is an equivalence.
To this end it suffices to show that  the map 
$\ho(\colim_{\nat} \kkGs(A){)}\to \ho( \kkGs(\colim_{\nat} A))$ obtained by composing with $\ho$ from \eqref{fewpoqkopfqewffwqwefqf1} is an isomorphism. 
Using the isomorphism $$\ho(\colim_{\nat} \kkGs(A))\cong \hocolim_{\nat} \ho( \kkGs(A)) \cong  \hocolim_{\nat} \kkGsk(A)$$ 
{and Proposition~\ref{woiejgwoegregreggwergw} (to translate homotopy colimits formed in the homotopy category of the stable $\infty$-category $\KKGs$ to homotopy colimits formed in the triangulated category $\KKGsk$ used in \cite{MR2193334})}, this is exactly the assertion of \cite[Prop. 2.6]{MR2193334}. 
%
\end{proof}
The existence of coproducts in $\KKGsk$ and the analogue of  Assertion \ref{ugoierugogergwergwregw}.\ref{regpoqgowergwergwregwregwreg} for $\kkGsk$ has previously been shown in \cite[Prop. 2.1]{MR2193334}. 

The proof of Theorem \ref{qroifjeriogerggergegegweg} is now complete. \phantomsection \label{woigwgwgwerg9}

 
 In the following we consider the minimal and the maximal tensor products $\otimes_{\min}$ and $\otimes_{\max} $  of $C^{*}$-algebras. Both equip $\Fun(BG,\nCalgs)$ with a symmetric monoidal structure.  
Recall that a symmetric monoidal structure on a stable $\infty$-category is called bi-exact  if the tensor product  preserves {co}fibre sequences, {and hence finite colimits}, in each variable. 
In this case the $\infty$-category  {together with its symmetric monoidal structure} is called stably symmetric monoidal.

 \begin{prop}\label{eqrgoerjgpergwegerg1}\mbox{}
The tensor product $\otimes_{?}$ for $?$ in $\{\min,\max\}$  descends to a  {bi-}exact symmetric monoidal structure  on $\KKGs$ and  $\kkGs$  refines to a 
symmetric monoidal functor 
$$\kkGstensor \colon \Fun(BG,\nCalgs)^{\otimes_{?}}\to \KKGstensor \, .$$ 
{Moreover, the tensor structure $\otimes_{?}$ on $\KKGs$ preserves countable colimits in each variable.}
\end{prop}
 \begin{proof} 
In order to show that the tensor product descends 
{along} the localization $$\kkGs \colon \Fun(BG,\nCalg_{\sepa})\to \KKGs$$
such that $\kkGs$ refines to a symmetric monoidal functor 
we use \cite[Prop.\ 3.2.2]{hinich}. By this result,   it suffices to show that for every $ A$ in $\Fun(BG,\nCalg_{\sepa})$ the functor
$ A\otimes_{?}- $ preserves ${\kkGsk}$-equivalences, where $?$ is in $\{\min,\max\}$. 
It is easy to see that the functor
$\kkGsk\circ ( A\otimes_{?}- ):\Fun(BG,\nCalg_{\sepa})\to \KKGsk$ preserves zero objects and is  
$\mathbb{K}^{G}$-stable and split exact since $\kkGsk$
has these properties. We now apply Corollary \ref{ioqerjgqergrqfefqewfq} in order to conclude that this composition sends $\kkGsk$-equivalences to isomorphisms. 
Hence $ A\otimes_{?}-$  preserves  $\kkGsk$-equivalences.

In order to show that the resulting symmetric monoidal structure on $\KKGs$ is exact 
we must show that $\kkGs(A)\otimes_{?}-$ in $\KKGs$ preserves fibre sequences for every $A$ in $\Fun(BG,\nCalgs)$.
 We  will use the observation made in the proof of 
Proposition \ref{woiejgwoegregreggwergw} that every fibre sequence in $\KKGs$ is equivalent to a cone sequence.   

It suffices to show that 
$\kkGs(A \otimes_{?}- )$ sends cone sequences to fibre sequences.

For the maximal tensor product, using the exactness of $ A \otimes_{\max} -$ (see e.g.  \cite[Prop.\ 3.7.1]{brown_ozawa}) and
the explicit description of the cone sequences in Remark \ref{cylinder object} 
one  checks that $ A\otimes_{\max} -$ preserves cone sequences in $\Fun(BG,\nCalg_{\sepa})$.  
We then use that the functor $\kkGs$ sends cone sequences to fibre sequences.

For the minimal tensor product we use the fact 
that    cone sequences admit completely positive contractive splits and that the
minimal tensor product is functorial for 
completely positive contractive maps. 
We conclude that $ A \otimes_{\min}- $ sends cone sequences to semisplit  exact sequences. 
We now use that $\kkGs$ is   semi{exact}.
%
%
%

{Since we already know that $\otimes_{?}$ on $\KKGs$ is bi-exact, in order to show that the corresponding tensor structure  on $\KKGs$ preserves countable colimits in each argument it suffices to show that it preserves countable sums.  As observed in the proof of Lemma  \ref{ugoierugogergwergwregw}, countable sums in $\KKGs$ are presented by countable sums in $\Fun(BG,\nCalg_{\sepa})$. It follows from 
Lemma \ref{eriguhwiegugwergwerg} and the fact that the tensor products of $C^{*}$-algebras preserves finite sums, that 
  $\otimes_{?}$ on $\Fun(BG,\nCalg_{\sepa})$ preserves countable sums in each argument.   Since $\kkGs$ preserves countable sums by Lemma \ref{ugoierugogergwergwregw}.\ref{regpoqgowergwergwregwregwreg} we can conclude that $\otimes_{?}$ on $\KKGs$ preserves 
countable sums, too.
}
 \end{proof}

\begin{rem}
Alternatively, the argument in the proof of Proposition \ref{eqrgoerjgpergwegerg} given for the minimal tensor product also applies to the maximal tensor product
since the latter is also functorial for completely positive contractive maps \cite[Cor.\ 4.18]{Pisier}.
\hB
\end{rem}

%

\begin{rem}
As a consequence of Proposition \ref{eqrgoerjgpergwegerg1} the homotopy category
$\ho(\KKGs)$    and therefore  $ \KKGsk$ has two tensor triangulated structures 
induced by $\otimes_{\min}$ and $\otimes_{\max}$, respectively. For $\otimes_{\min}$ this fact is well known, see e.g.\ \cite[Sec.~2.5]{MR2193334}.  It seems that the symmetric monoidal structure  coming from the maximal tensor product has not been studied so much, but for the non-equivariant case see e.g.\ \cite[Lem.~3.13]{Land:2016aa}.
\hB
\end{rem}

\color{black}

Note that  $\kkGs$  is, by Definition~\ref{wtpohwopggergeg},  the initial   functor from $\Fun(BG,\nCalg_{\sepa})$ to $\infty$-categories which  sends ${\kkGsk}$-equivalences to equivalences.  In 
Theorem   \ref{wtohwergerewgrewgregwrg1}  we 
{stated} a different universal property which better reflects the  standard properties of $\KKth$-theory.  We will now first state an intermediate 
Theorem \ref{wtohwergerewgrewgregwrg} involving additive $\infty$-categories  in order to formulate a direct $\infty$-categorial analog of   \cite[Thm.\ 6.6]{Meyer:aa} which appeared in {the present} paper as Proposition \ref{wtohiwthwgrgwerg}.

Recall from \cite[6.1.6.13]{HA} that an $\infty$-category is called semi-additive if it is pointed, admits finite products and coproducts, and if the canonical morphism
from a coproduct to the product of any two objects is an equivalence. The homotopy category of a semi-additive category is canonically enriched in abelian monoids. If all these morphism monoids are abelian groups then the $\infty$-category is said to be additive. Stable $\infty$-categories are additive. 
In particular, 
since  $\KKGs$ is stable by Proposition \ref{KK-stable},  it is additive.

\begin{theorem}\label{wtohwergerewgrewgregwrg}
The functor $\kkGs \colon \Fun(BG,\nCalg_{\sepa})\to \KKGs$ is initial among functors
from $\Fun(BG,\nCalg_{\sepa})$ to  objects of $\Cat^{\add}_{\infty}$ which are reduced, 
$\mathbb{K}^{G}$-stable, and split exact.
\end{theorem}\begin{proof}
 For an additive $\infty$-category $\cD$ we consider the full subcategory \begin{equation}\label{rgqrgregegergerg}
\Fun^{{rse}}(\Fun(BG,\nCcat_{\sepa}),\cD)
\end{equation}
{of $  \Fun(\Fun(BG,\nCcat_{\sepa}),\cD)$} on functors which are reduced,  
$\mathbb{K}^{G}$-stable, and split exact.
By Corollary \ref{ioqerjgqergrqfefqewfq}   the functor  category \eqref{rgqrgregegergerg}
 is  {a full subcategory  of the category}   
 $\Fun^{{\kkGsk}}(\Fun(BG,\nCalg_\sepa),\cD) $ 
  {of functors} sending ${\kkGsk}$-equivalences to equivalences. 
 
We now {build the following commutative} diagram
 \begin{equation}\label{fwroihoiegerwggergewgerggrewgegeg}
\begin{tikzcd}
	\Fun^{\coprod}(\KKGs,\cD) \ar[r,dashed] \ar[d] & \Fun^{{rse}}(\Fun(BG,\nCalg_\sepa),\cD) \ar[d] \\
	\Fun(\KKGs,\cD) \ar[r,"\simeq"] & \Fun^{{\kkGsk}}(\Fun(BG,\nCalg_\sepa),\cD) 
\end{tikzcd}
\end{equation}
{The  vertical morphisms  are inclusions of full subcategories
and the horizontal functors are induced by precomposition with $\kkGs$.
The superscript $\coprod$ stands for coproduct preserving functors which are the morphisms in $\Cat^{\add}_{\infty}$. By the universal property of 
 $\kkGs$ as a Dwyer-Kan localization 
 the lower horizontal functor is an equivalence.}
We now justify  {that the dashed arrow exists, making the diagram commute.} 
 For this, we note that the localisation functor $\kkGs$ is reduced, 
 $\mathbb{K}^G$-stable and {semiexact by Theorem \ref{qroifjeriogerggergegegweg}.   In particular it is split-exact.}  Thus if $H \colon \KKGs \to \cD$ preserves {co}products, {then    the composition $H \circ \kkGs$ is  reduced, 
 $\mathbb{K}^G$-stable and split exact.}

{In order to prove  Theorem  \ref{wtohwergerewgrewgregwrg} we must show}
that the dashed arrow in \eqref{fwroihoiegerwggergewgerggrewgegeg} is an equivalence.   {As all other functors in {the} diagram \ref{fwroihoiegerwggergewgerggrewgegeg} are fully faithful, so is the dashed one. It remains to show that it is essentially surjective.}

 {To this end we consider 
  $F$ in $\Fun^{{rse}}(\Fun(BG,\nCalg_\sepa),\cD) $. In view of the lower horizontal equivalence in \eqref{fwroihoiegerwggergewgerggrewgegeg}  there exists a 
 functor $\bar F$ in $\Fun(\KKGs,\cD)$  and an equivalence $\bar F\circ \kkGs\simeq F$.
 We must check that $\bar F\in\Fun^{\coprod}(\KKGs,\cD) $, i.e., that   $\bar F$ preserves coproducts.}

 First of all $F$ preserves the empty coproduct since $\kkGs$ and $F$ are  reduced.

  
We now show that $\bar F$ preserves binary coproducts. Since $\kkGs$ and $\cD$ are additive, in both $\infty$-categories coproducts and products coincide.
Therefore it suffices to show that $\bar F$ preserves binary products.
   Since $F$ is split-exact and reduced, it sends a split exact square of the form
   \begin{equation*}
\xymatrix{   A\oplus    B\ar[r]\ar[d]&   A\ar[d]\\    B\ar[r]&0}
\end{equation*} 
 to a cartesian square with zero in the lower right corner. 
  We conclude that $F$ preserves products. 
      Now by Proposition \ref{woitjgorwergergegergw}.\ref{ergpoggergwegregweg}  any product in $\KKGs$ is equivalent to the image of a product in
  $\Fun(BG,\nCalg_{\sepa})$. In view of the equivalence $\bar F\circ \kkGs\simeq F$ we can conclude that $\bar F$ preserves binary products.


{Hence $\bar F$ preserves binary coproducts.}
{Since $F$ preserves   empty and binary coproducts it preserves all finite coproducts.}
%
\end{proof}

Theorem \ref{wtohwergerewgrewgregwrg} says that if $\cD$ is  an object of $\Cat^{\add}_{\infty}$   and $F \colon \Fun(BG,\nCalg_{\sepa})\to 
\cD$ is a reduced,  
$\mathbb{K}^G$-stable, and split exact functor, 
  then there exists  an essentially unique, finite coproduct-preserving factorization $\bar{F}$ as indicated
in the diagram
\[\xymatrix{
	\Fun(BG,\nCalg_\sepa) \ar[r]^-{F} \ar[d]_-{\kkGs}& \cD\,. \\
	\KKGs \ar@{-->}[ur]_-{\bar{F}} 
}
\]

Since $\cD$ and  $\KKGs$ are additive, in both $\infty$-categories products and coproducts coincide. Hence $\bar F$ also  preserves all finite products.
If $\cD$ admits finite  limits, e.g., if $\cD$ is stable, one might wonder whether the functor $\bar{F}$ in addition preserves finite  limits. In general it does not, see Remark~\ref{L-theory} for an example, but we have the following characterization.

 Let $\cD$ be an additive $\infty$-category{,} and let
 $F \colon \Fun(BG,\nCalg_\sepa) \to \cD$ be a functor 
 which  is  reduced,  
  $\mathbb{K}^{G}$-stable and split exact.
  Let $\bar F \colon \KKGs\to \cD$ be the factorization as explained above.
  In the following statement we use the Definition \ref{qwroigjqrwgqwrfqewfqewfq}.\ref{qeriughqeiruferwqkfjqr9ifuq1} of semiexactness for functors with additive targets.

\begin{theorem}\label{4ogijwrtgwergerwgwergegergwerg}
If in addition $\cD$ admits finite  limits and $F$ is semiexact,
  then 
  $\bar{F}$ preserves finite  limits. 
\end{theorem}  
 \begin{proof}  It suffices to prove that $\bar{F}$ is reduced and sends fibre sequences to fibre sequences. {We already know that $\bar F$ is reduced
 by   Theorem  \ref{wtohwergerewgrewgregwrg}.}
 
By Proposition~\ref{woitjgorwergergegergw}.\ref{woifhjoqiefewqewfqef1}, any fibre sequence in $\KKGs$ is the image {under $\kkGs$} of a semisplit exact sequence in $\Fun(BG,\nCalg_\sepa)$. 
   By assumption, $F$ sends such   {a  semisplit exact sequence} to a fibre sequence. {In view of the relation $\bar F\circ \kkGs\simeq F$ we conclude that $\bar F$  sends the image of this  semisplit exact sequence   under $\kkGs$ to a fibre sequence.
 Consequently $\bar F$ preserves fibre sequences.}  
  \end{proof}
 
\phantomsection \label{qeroijgoiergeregergergewrgergewg}
 \begin{proof}[Proof of Theorem \ref{wtohwergerewgrewgregwrg1}]
 The theorem immediately follows by specializing  the Theorems \ref{wtohwergerewgrewgregwrg} and \ref{4ogijwrtgwergerwgwergegergwerg} to stable target categories.
 \end{proof}

  

\begin{rem}\label{L-theory}
A natural example of a functor $\nCalg_\sepa \to \Sp$ which {is reduced,} 
$\mathbb{K}$-stable and split exact is the composite
\[ \nCalg_\sepa \lto \mathrm{Rings}_{\mathrm{inv}}^{\mathrm{nu}} \stackrel{L}{\lto} \Sp \]
where the first functor is the forgetful functor taking the underlying ring with involution of a $C^{*}$-algebra and the second takes the (projective, symmetric) L-theory spectrum of a ring with involution. This functor descends to a functor ${\KKs} \to \Sp$ which  preserves finite products, but is not exact. Indeed, the failure of exactness in this case can be described explicitly, we refer to \cite[Thm.~4.2]{Land:2016aa} for a general treatment.
\hB
\end{rem}

%
%

\section{The s-finitary extension}

%

Let $F \colon \Fun(BG,\nCalg)\to \cD$ be a functor to a target  {category $\cD$} admitting all small filtered colimits.
Then  we have a canonical natural transformation $ \hat F\to   F$, where  $\hat F$ {is}
  the left Kan extension of $F_{s} \coloneqq F_{|\Fun(BG,\nCalg_{\sepa})}$ 
 as indicated in the following diagram:
$$\xymatrix{
\Fun(BG,\nCalg_{\sepa})\ar[rrr]^-{F_{s}}\ar[d]^{\incl} &&& \cD\\
\Fun(BG,\nCalg) \ar@{}[urr]^-{\Downarrow\ } \ar[urrr]^{\hat F}_{\:\Downarrow} \ar@/_1cm/[urrr]^-{F\quad} &&&
}$$


 \begin{lem}\label{ergoijegwergrewergwergwerg}
The functor $F$ is s-finitary if and only if the natural transformation $\hat F\to F$ is an equivalence. 
\end{lem} \begin{proof} 
  The pointwise formula for the left Kan-extension
 shows that
 $$ \colim_{(   A'\to    A)\in {\Fun(BG,\nCalg_{\sepa})}_{/A}} F(   A' )\simeq  \hat F(   A)\,.$$
 If we compare this formula with the morphism \eqref{efrgveveeerwvfevdsfvsdvfdvdsv} appearing in the condition for being s-finitary  it becomes clear that
 we must show that the poset of separable subalgebras of $   A$ is cofinal in ${\Fun(BG,\nCalg_{\sepa})}_{/A}$. This follows from the following observation.
 
 If $f \colon   A'\to    A$ is any   morphism  in $\Fun(BG,\nCalg)$, then  we have a factorization 
$$   A'\to  f(   A') \to     A\, ,$$ 
where the image $f(   A')$ is a $G$-$C^{*}$-subalgebra of $   A$.  If $   A'$ is separable, then    $f(   A')$ is separable, too.  
\end{proof}


The next lemma is the essential step for the derivation of Theorem \ref{qeroigjqergfqeewfqewfqewf1} from Theorem \ref{qroifjeriogerggergegegweg}.
Let $F$ and $F_{s}$ be as above

 \begin{lem}\label{qwoiefuoqrfwewfqwef9}
 Assume that $F$ is s-finitary.
 \begin{enumerate}
 \item \label{qroijoqrfqqewffqewf}  {Assume that $\cD$ is pointed.} If  $F_s$ is reduced, then so is $F$.
 \item  \label{qroijoqrfqqewffqewf2}  If $F_{s}$ is $\mathbb{K}^{G}$-stable, then so is $F$.
  \item \label{qroijoqrfqqewffqewf1}  If $F_{s}$ is homotopy invariant, then so is $F$.
 \item \label{qroijoqrfqqewffqewf3}  {Assume that $\cD$ is pointed, admits fibres, and that filtered colimits preserve fibre sequences.}\footnote{E.g., that $\cD$ is stable or pointed and compactly generated.}
 If $F_{s}$ is semiexact, then so is $F$.
 \item \label{efqiuwehfiuewfewfewfqwef} {Assume that $\cD$ is pointed, admits fibres, and that filtered colimits preserve fibre sequences.}
 If $F_{s}$ sends exact sequences to fibre sequences, then so does $F$. 
 \end{enumerate}
 \end{lem} \begin{proof} 
We {begin with} Assertion \ref{qroijoqrfqqewffqewf1}.
For every $t$ in $[0,1]$ let $\ev_{t}:C([0,1])\to \C$ be the evaluation at $t$. 
 If  $   B$ is a  $G$-invariant (we will also just say {\em invariant}) separable subalgebra of $C([0,1])\otimes    A$, then the values $(\ev_{t}\otimes \id_{   A})(b)$  for all $b$ in  $   B$ and $t$ in $[0,1]$ generate  an invariant separable subalgebra $   A'$ of $   A$ such that $   B\subseteq C([0,1])\otimes    A'$.
 Hence the   invariant subalgebras  of $C([0,1])\otimes    A$ of the form $C([0,1])\otimes    A'$ with $   A'$ an   invariant separable subalgebra of $   A$ are cofinal in all  invariant separable subalgebras of $C([0,1])\otimes    A$.  Since we assume that $F$ is s-finitary
   we have  the chain of equivalences 
 $$F(   A)\simeq   \colim_{   A'{\subseteq_{\sepa}A}}  F_{s}(   A')\simeq \colim_{   A'{\subseteq_{\sepa}A}}  F_{s}(C([0,1])\otimes    A')\simeq 
 F(C([0,1])\otimes    A)\, ,$$ where the colimit runs over the poset of all   invariant separable subalgebras $   A'$ of $   A$ and
 the middle equivalence is a consequence 
 of the assumption on $F_{s}$. Since this morphism is induced by the canonical map \eqref{ewgegreggergegeggegwergw}
 this
 shows that $F$ is homotopy invariant

 {To see} Assertion \ref{qroijoqrfqqewffqewf2}, 
let $   A $ be in $\Fun(BG,\nCalg)$, 
 let $     H\to     H'$ be an equivariant isometric  inclusion of separable  $G$-Hilbert spaces {such that $   H\not=0$}, and $ K(   H)\to  K(    H')$  be the corresponding inclusion of the algebras of compact operators. 
  We note that these algebras  are separable. It follows that the family of  invariant subalgebras $   A'\otimes   K(    H)$ for all invariant separable subalgebras $   A'$ of $   A$ is cofinal in all separable  invariant subalgebras of $    A\otimes   K(   H)$, and similar for $   H'$.
Using that $F$ is s-finitary  we conclude that  the   morphism obtained by applying $F$ to \eqref{342uihuihfewfwefweferf}  has a factorization over the chain of equivalences
\begin{eqnarray*}
F(   A\otimes K(   H))&\simeq& \colim_{   A'} F_{s}(   A'\otimes  K(   H))\\&\simeq&
\colim_{   A'} F_{s}(   A'\otimes     K(   H'))\\&\simeq&
F(   A\otimes  K(   H))\, ,
\end{eqnarray*}
 where the colimit runs over the poset of all  invariant separable subalgebras $A'$ of $   A$
and  the middle equivalence  follows from the assumption on $F_{s}$.

Assertion \ref{qroijoqrfqqewffqewf} is obvious since the zero algebra is separable.
%
%

 {To} show Assertion \ref{qroijoqrfqqewffqewf3}, we  now  consider  a     semisplit  exact sequence  \begin{equation}\label{23rgfj23of3f23f234f34f}
 0\to    A'\stackrel{i}{\to}   A\stackrel{\pi}{\to}    A'' \to 0 
\end{equation}
in $\Fun(BG,\nCalg)$.
Let $s:   A''\to    A$ denote the equivariant  cpc split of $\pi$.
  We consider the 
  family $\cS$  of  exact sequences 
 $$\cC\quad: \quad 0\to    C'\to    C\to    C''\to 0$$       of separable $G$-invariant subalgebras of \eqref{23rgfj23of3f23f234f34f} such that $s(   C'')\subseteq    C$.
 Such a sequence is again equivariantly semisplit since we can use the restriction of $s$. Since $F$ is s-finitary
 it is enough to show that 
the  families of constituents $   C'$, $   C$, and $   C''$ for $\cC$ running in $\cS$ are cofinal families of invariant separable $C^{*}$-algebras in $   A'$, $   A$ and $   A''$, respectively.
Indeed, then  using the assumption on $F_{s}$ the fibre sequence
$$F(   A')\to F(   A)\to F(   A'')\to \Sigma F(A')$$ is the colimit over $\cS$ of the family of fibre sequences
$$F_{s}(   C')\to F_{s}(   C)\to F_{s}(   C'')\to \Sigma F_{s}(   C')\, .$$

Let $   B'$, $   B$ and $   B''$ be separable $C^{*}$-subalgebras of $   A'$, $   A$, and $   A''$, respectively. 

Then  the invariant $C^{*}$-subalgebra $   C'':=   A''(\pi(   B),   B'')$ of $A''$ generated by the subsets $\pi(   B)$  and $    B''$ is separable and contains $   B''$. 
  Then $   C:=   A(   B,s(   C''),i(   B'))$ is a separable subalgebra of $   A$ containing $   B$  such that the restriction of $\pi$ induces a surjection
$    C\to    C''$ and $s(   C'')\subseteq    C$. 
We finally define $   C':= \ker(    C \to    C'')$ considered as a subalgebra of $   A'$. Then 
$   B'\subseteq    C$. 
Since $   B'$ is a closed subspace of the separable 
$C^{*}$-algebra  it is again separable.  

By construction we have an equivariantly  semisplit exact   sequence
$$0\to     C'\to    C\to    C'' \to 0$$
in $\cS$
and $$   B'\subseteq    C'\, , \quad    B\subseteq    C\, , \quad    B''\subseteq    C''\, .$$


 The proof of Assertion \ref{efqiuwehfiuewfewfewfqwef} is a small modification of the proof
 of Assertion \ref{qroijoqrfqqewffqewf3}. Since we do not have a split at all, we define
 $   C:=   A(   B, G\hat C'',i(   B'))$, where $\hat C''$ is a set of choices of preimages of  a countable dense subset of $C''$.
 Otherwise the argument is the same.
\end{proof}

 \phantomsection \label{weotigjwotgwregergegw}
 \begin{proof}[Proof of Theorem \ref{qeroigjqergfqeewfqewfqewf1}]
 Since    the functor  $\kkG$  is equivalent  by    Definition \ref{qeroigoergrgwg} to the left Kan-extension of its restriction $\kkGs$ to 
 separable $G$-$C^{*}$-algebras 
it is s-finitary by Lemma \ref{ergoijegwergrewergwergwerg}.
  This is   Assertion  \ref{qeroigjqergfqeewfqewfqewf1}.\ref{qrfghqirwogrgqfqwefq} which verifies the assumption
  in Lemma \ref{qwoiefuoqrfwewfqwef9} for $F=\kkG$ and $F_{s}=\kkGs$. 

The remaining assertions of Theorem \ref{qeroigjqergfqeewfqewfqewf1} now follow by applying  Lemma \ref{qwoiefuoqrfwewfqwef9}
and   Theorem \ref{qroifjeriogerggergegegweg} asserting the corresponding properties for $\kkGs$.
%
 \end{proof}
%

Our next theorem settles the universal property of $\kkG$ stated as Theorem \ref{eoiruggwerregergegwo} in the introduction.
 \begin{theorem} \label{weoigjwoegerggegwegergecv}
 \label{ojroiwfewfqefqef} The functor  $\kkG$ is  initial  {among} functors from $\Fun(BG,\nCalg)$ to  {objects of} {$\CAT_{\infty}^{\ccpl\cap {\exa}}$}
which  {are s-finitary, reduced,  
$\mathbb{K}^G$-stable and semiexact.} 
\end{theorem}
\begin{proof}
 Let $\cD$ be a  {cocomplete stable} 
 $\infty$-category.
 We then have the following chain of equivalences
\begin{eqnarray*}
\Fun^{\colim}(\KKG,\cD)&\stackrel{y^{G,*},\simeq}{\to}& 
\Fun^{ \exa}(\KKGs,\cD)\\
&\stackrel{(\kkGs)^{*},\simeq}{\to}& 
\Fun^{rse}( \Fun(BG,\nCalg_{\sepa}),\cD)\\&\stackrel{\incl^{*},\simeq}{\leftarrow}& \Fun^{frse}( \Fun(BG,\nCalg ),\cD)\, ,
\end{eqnarray*}
 where the superscript $\colim$ stands  for colimit preserving functors. Furthermore the superscript   ${(f)rse}$  stands   the full subcategory of functors  which are (s-finitary,) reduced,  
 $\mathbb{K}^G$-stable and semiexact, respectively.\footnote{Note that in the proof of Theorem~\ref{wtohwergerewgrewgregwrg}, the $e$ in the superscript $rse$ referred to split exact as opposed to semiexact in the present situation. We apologise for the notational clash.}

Here for the first {two} steps  we use the universal properties of $y^{G}$ in \eqref{ewgfuqgwefughfqlefiewfqwe} and $\kkGs$ {(Theorem \ref{wtohwergerewgrewgregwrg1})}.
The last equivalence is a consequence of the   Lemmas \ref{ergoijegwergrewergwergwerg} and  \ref{qwoiefuoqrfwewfqwef9} which imply that  
if we left Kan-extend a functor in
$\Fun^{{rse}}( \Fun(BG,\nCalg_{\sepa}),\cD)$ along the inclusion $\incl$ in \eqref{qfwoefjkqwpoefkewpfqewfqwfqfwefq},
 then the result is in 
$ \Fun^{{frse}}( \Fun(BG,\nCalg_{\sepa}),\cD)$.  
Since $\kkG\circ \incl\simeq y^{G}\circ \kkGs$ by  \eqref{qfwoefjkqwpoefkewpfqewfqwfqfwefq}
we conclude 
$$\Fun^{\colim}(\KKG,\cD)\stackrel{\simeq}{\lto} \Fun^{frse}(\Fun(BG,\nCalg ),\cD)$$
which is the desired equivalence.
%
%
\end{proof}

\begin{rem}
Note that the functor $y^{G}\colon \KKGs\to \KKG$ in \eqref{ewgfuqgwefughfqlefiewfqwe} only preserves finite colimits. 
Therefore, in contrast to Proposition \ref{ugoierugogergwergwregw}.\ref{regpoqgowergwergwregwregwreg}, the functor $\kkG$ only  sends finite sums to coproducts. 
One could improve this situation by forming the {Bousfield} localization 
$$L:\KKG\rightleftarrows \KKGp:\incl$$
of $\kkG$ at the set of maps $\bigoplus_{n\in \nat} \kkG(A_{n})\to \kkG(\bigoplus_{n\in \nat} A_{n})$ for all families
$(A_{n})_{n\in \nat}$ in $\Fun(BG,\nCalgs)$ and setting
$$\kkGp:=L\circ \kkG:\Fun(BG,\nCalg)\to \KKGp\ .$$
We restrict to families of separable algebras in order have a set of generators.

The functor $\kkGp$ {is}
s-finitary. In addition, it preserves countable sums. Indeed, it  preserves countable sums of
families of separable $G$-$C^{*}$-algebras by construction. But using that $\kkGp$ is s-finitary
we can  remove the assumption on separability.

As an immediate consequence of Theorem  \ref{weoigjwoegerggegwegergecv}  and the definitions the universal property of $\kkGp$ is as follows: The functor 
$\kkGp$ is  initial  {among} functors from $\Fun(BG,\nCalg)$ to  {objects of} {$\CAT_{\infty}^{\ccpl\cap \exa}$}
which  are s-finitary, reduced, 
$\mathbb{K}^G$-stable, semiexact, and
which send countable sums to coproducts.
\hB
\end{rem}

%
%
  The next proposition {is} Proposition \ref{eroigowregwergregwgreg1} from the introduction.
  \begin{prop}\label{eroigowregwergregwgreg11}
  If $   A$ is in $\Fun(BG,\nCalg_{\sepa})$ and $   B$ is in $ \Fun(BG,\nCalg_{\sigma})${,} 
  then {the functor $\ho$ induces}    an isomorphism of $\Z$-graded groups \begin{equation}\label{wergiuh3iugjknegfds}
\pi_{*}\KKG(   A,   B)\cong  \KKth^{G}_{*}(   A,   B)\, .
\end{equation}
  \end{prop}
 \begin{proof} 
Let $   A$ be in $\Fun(BG,\nCalg_{\sepa})$. We consider the functors
\begin{equation}\label{werferferfrferfwfrew}
\pi_{*}\KKG(   A,-), \, \KKth^{G}_{*}(   A,-)  \colon \Fun(BG,\nCalg_{\sigma}) \to \Ab^{\Z}\, .
\end{equation}  
 By the compatibility of $\ho$ with the triangulated structure stated in Theorem \ref{weighigregwgregw9}.\ref{weoitgjowergwegregw9} 
it induces a natural isomorphism  $$\pi_{*}\KKG(   A,(-)_{s})\xrightarrow{\cong} \KKth^{G}_{*}(   A,(-)_{s})$$ between the restriction of these functors to
separable $G$-$C^{*}$-algebras.
It therefore  suffices to show that both functors in \eqref{werferferfrferfwfrew} are s-finitary. 
Since $\kkG(   A)\simeq y^{G}(\kkGs(   A))$ is a compact object of $\KKG$ and $\kkG$ is $s$-finitary, the functor $\KKG(   A,-)$ is  s-finitary, too.   Since $\pi_{*}$ preserves filtered colimits also 
$\pi_{*}\KKG(   A,-)$ is s-finitary. 

In order to see that $\KKth^{G}_{*} (   A,-)$ is s-finitary, using the shift equivalence of the triangulated structure on $\KKGsk$, it suffices to show this for $*=0$.   We now use the formula \cite[Thm.~5.5]{Meyer:aa}\footnote{Alternatively, we could use \cite[Rem.~3.2]{zbMATH03973627}.} (at this point we need the assumption that $B$ is $\sigma$-unital)
stating that
$$\KKth^{G}(   A,   B)\cong [q_{s}(   A)\otimes    K,   B\otimes     K]\, .$$
Here $   K \coloneqq K(\ell^{2}\otimes L^{2}(G))$, $q_{s}(   A) \coloneqq \ker(   A\otimes    K\sqcup    A\otimes    K\to    A\otimes    K)$, and $[-,-]$  stands for taking   homotopy classes of morphisms in $\Fun(BG,\nCalg)$.
We note that  $   K$ is separable and $q_{s}(   A)$ is separable.
Every equivariant homomorphism $q_{s}(   A)\otimes    K \to    B\otimes     K$
and every homotopy between such homomorphisms factorizes over an invariant separable subalgebra of the target. Furthermore, the separable subalgebras of the form $   B'\otimes    K$  for invariant  separable subalgebras $   B'$ of $   B$ are cofinal. 
Consequently, $\KKth^{G}(   A, -)$ is s-finitary, too.
  \end{proof}

 We now turn to the proof of Proposition \ref{togjoigerggwgergwgr} from the introduction (repeated here as Proposition \ref{togjoigerggwgergwgr1}) concerning symmetric monoidal structures on $\kkG$ and $\KKG$.
As a preparation we  show  how left Kan extensions interact with symmetric monoidal structures.
Let $r\colon \cC\to \cC^{\prime}$ be a   symmetric monoidal functor between small symmetric monoidal $\infty$-categories, and let $\cD$  be a presentably symmetric monoidal $\infty$-category. 
 \begin{lem}\label{ergiowerg43frefref}
The left Kan extension $r_{!}F:\cC'\to \cD$ of a lax symmetric monoidal functor $F\colon \cC\to \cD$ along $r$ has a naturally induced lax symmetric monoidal refinement.
\end{lem}
\begin{proof}

 The restriction functor
$r^{*}\colon \Fun(\cC^{\prime},\cD)\to \Fun(\cC,\cD)$ is lax symmetric monoidal (with respect to the Day convolution structures on the functor categories). Furthermore it has a symmetric monoidal left-adjoint
$r_{!}\colon \Fun(\cC,\cD)\to \Fun(\cC^{\prime},\cD)$ which is a symmetric monoidal refinement of the left Kan-extension functor{, see \cite[Cor. 3.8]{Nikolaus:2016aa}.}
This implies that $r_{!}$ preserves commutative algebras. 

By  \cite[Prop. 2.16]{Glasman:aa}, \cite[Sec. 2.2.6]{HA}
commutative algebras in these functor categories correspond to lax symmetric monoidal functors. 
Therefore, if $F\colon \cC\to \cD$ is lax symmetric monoidal, it is a commutative algebra in $\Fun(\cC,\cD)$.  Consequently, $r_{!}F$  is a commutative algebra in $\Fun(\cC^{\prime},\cD)$, and hence $r_{!}F$ is lax symmetric monoidal.
\end{proof}

\begin{rem}\label{ewgoijwgioergwergregwg}
In proof of Proposition  \ref{togjoigerggwgergwgr1} below
we will  use that the $\Ind$-completion $\Ind(\cD)$ of  a symmetric monoidal $\infty$-category $\cD$ admits a canonical symmetric monoidal structure preserving filtered colimits in each argument
 such that the canonical functor $\cD\to \Ind(\cD)$ has a canonical refinement 
 to a symmetric monoidal functor.  If $\cD$ is stable and the monoidal structure on $\cD$ is exact, then   the induced structure on $\Ind(\cD)$  is presentably symmetric monoidal, and the canonical functor is symmetric monoidal. A reference for these statements is \cite[Cor.\ 4.8.1.14]{HA}. 
 \hB 
 \end{rem}

 \begin{prop}\label{togjoigerggwgergwgr1}
 The   symmetric monoidal structure $\otimes_{?}$ on $\KKGs$  for $?$ in $\{\min,\max\}$  induces a
 presentably symmetric monoidal {structure} on $\KKG$  and $\kkG$ refines to a symmetric monoidal functor
 \begin{equation}\label{qregiuhiuwfqef}
\kkGtensor \colon \Fun(BG,\nCalg)^{\otimes_{?}} \to \KKGtensor\, .
\end{equation}
 \end{prop}
   \begin{proof} Applying Remark \ref{ewgoijwgioergwergregwg} to $\KKGs$ with  the symmetric monoidal structure $\otimes_{?}$ for $?$ in $\{\min,\max\}$  constructed in Proposition \ref{eqrgoerjgpergwegerg} 
we get a presentably symmetric monoidal structure $\otimes_{?}$ on $\KKG$ and a    symmetric monoidal refinement 
$$y^{G,\otimes_{?}} \colon \KKGstensor \to \KKGtensor$$
of the functor in \eqref{ewgfuqgwefughfqlefiewfqwe}.
Since $\kkGs$ has a symmetric monoidal structure  by   Proposition \ref{eqrgoerjgpergwegerg},
 the upper horizontal composition in the diagram   
$$\xymatrix{\Fun(BG,\nCalg_{\sepa}) \ar[r]^-{\kkGs}\ar[d]^{\incl}&\KKGs\ar[r]^{y^{G}}&\KKG\\\Fun(BG,\nCalg)\ar[rru]_-{\kkG}&}$$
has a symmetric  monoidal structure, too.
 Applying Lemma \ref{ergiowerg43frefref} to  the symmetric monoidal functor $\incl$ in place of $r$ we conclude that 
 $\kkG$ {acquires a canonical} lax symmetric monoidal refinement \eqref{qregiuhiuwfqef}.

To show that $\kkG$ is in fact symmetric monoidal,  we {have to} show   that for
$   A$ and $   B$ in $ \Fun(BG,\nCalg)$  the canonical map
\begin{equation}\label{qefoqwepofqwefqewffqewfqwef}
\kkG(   A)\otimes_{{?}} \kkG(   B)\to \kkG(   A\otimes_{?}    B)
\end{equation} 
is an equivalence.  Furthermore, we must show that the
unit morphism
\begin{equation}\label{qefoqwepofqwefqewffqewfqwefq}
1^{{\otimes_{?}}}_{\KKG}\to \kkG(\C)
\end{equation}
is an equivalence,
where $1^{\otimes_{?}}_{\KKG}$ is the tensor unit for the structure $\otimes_{?}$ on $\KKG$, and  $\C$ is the tensor unit for $\Fun(BG,\nCalg)$ for both structures.

We start with the discussion of the unit morphism. By  Remark \ref{ewgoijwgioergwergregwg} we have an equivalence $1^{\otimes_{?}}_{\KKG}\simeq y^{G}(1_{\KKGs})$.
By Proposition \ref{eqrgoerjgpergwegerg} we also know that
$1_{\KKGs}\simeq \kkGs(\C)$. If we now apply $y^{G}$ to the second equivalence
and compose  with the first, then we get the desired equivalence {\eqref{qefoqwepofqwefqewffqewfqwefq}}.

    We now give separate arguments for \eqref{qefoqwepofqwefqewffqewfqwef} in the cases  where $?= \max$  and $?=\min$.

We first consider the case of $\otimes_{\min}$. We will employ the following fact: 
 If $A$ and $B$ are in $\nCalg$
and $A'$ is a closed subalgebra of $ A$ and $B'$ is a closed subalgebra of $B'$, then the canonical map
$A'\otimes_{\min}B'\to A\otimes_{\min}B$ is an isometric inclusion {\cite[Prop. 3.6.1]{brown_ozawa}.}


 The desired equivalence \eqref{qefoqwepofqwefqewffqewfqwef} for $\otimes_{\min}$ is given by the  following chain of equivalences:
\begin{eqnarray*}
\kkG(   A) \otimes_{\min} \kkG(   B) & \simeq& \colim\limits_{   A' \subseteq_{{\sepa}}    A,    B' \subseteq_{{\sepa}}   B} \kkGs(   A') \otimes_{\min} \kkGs(   B') \\
	& \stackrel{\simeq}{\to}& \colim\limits_{   A' \subseteq_{{\sepa}}      A,    B' \subseteq_{{\sepa}}   B} \kkGs(   A' \otimes_{\min}    B')  \\
	& \stackrel{(!)}{\simeq} &\kkG(   A \otimes_{\min}    B)\,.
\end{eqnarray*}
In the first step, we have used that $\kkG$ is s-finitary and {that} the tensor product in $\kkG$ commutes with colimits in each variable. In the  second step we use that $\kkGs$ is symmetric monoidal for $\otimes_{\min}$ by Proposition \ref{eqrgoerjgpergwegerg}. To justify the final equivalence $(!)$, it suffices to show that the poset of subalgebras of $   A \otimes_{\min}    B$ of the form $   A' \otimes_{\min}    B'$ (note that here we use the fact from above), for separable subalgebras $   A'$ and $   B'$ of $   A$ and $   B$, respectively, is cofinal in the poset of all separable subalgebras. To see this, let $   C$ be an invariant separable subalgebra of $   A\otimes_{\min}    B$. Let $(c_{i})_{i\in I}$ be a countable dense subset of $   C$.
For every $i$ in $I$ we can choose a countable family of elementary tensors
$(a_{i,j}\otimes b_{i,j})_{j\in J_{i}}$ such that $c_{i}$ belongs to the closure of the linear span of this family. We then let $   A'$ be the subalgebra generated by the $G$-orbits of the elements $a_{i,j}$ for all $i$ in $I$ and $j$ in $J_{i}$.
It is an invariant  separable subalgebra of $   A$. We define $   B'$ similarly using the elements $b_{i,j}$.
By construction we have $   C\subseteq    A'\otimes_{\min}    B'$. 
This finishes the argument in the case of $\otimes_{\min}$.

We now consider the maximal tensor product $\otimes_{\max}$. The problem here is that if $   A'$ and $   B'$ are as above then {in general the} 
canonical map $   A'\otimes_{\max}    B'\to    A\otimes_{\max}    B$ is {not} an isometric inclusion. Denoting the image of this map by 
 $   A'\bar{\otimes}_{\max}    B'$ we get the surjection
 $   A'\otimes_{\max}    B'\to    A'\bar{\otimes}_{\max}    B'$.
 
 We use that $   A\cong \colim_{   A'\subseteq_{{\sepa}}      A}   A'$, $   B\cong \colim_{   B'\subseteq_{{\sepa}}      B}    B'$ and the fact that  the maximal tensor product commutes with filtered colimits {(see  Lemma \ref{eriguhwiegugwergwerg}.\ref{weoigjweorrerggwergr})} in order to conclude that 
 \begin{equation}\label{dvavdfsvdvsvqrwvdsfva1}   A\otimes_{\max}    B\cong \colim_{   A'\subseteq_{{\sepa}}      A,   B'\subseteq_{{\sepa}}      B}    A'\otimes_{\max} B'\, ,
 \end{equation} 
 where as before the colimit runs over the poset of $(   A',   B')$ of pairs of invariant separable subalgebras of $   A$ and $   B$, respectively.
 By a similar cofinality argument as in the case of the minimal tensor product we also get 
 \begin{equation}\label{dvavdfsvdvsvqrwvdsfva}
   A\otimes_{\max}    B\cong \colim_{   A'\subseteq_{{\sepa}}      A,   B'\subseteq_{{\sepa}}      B}    A'\bar{\otimes}_{\max}    B'\, .
\end{equation} 

In the following we will show the claim that 
the inductive systems $(   A'\otimes_{\max}    B')_{   A'\subseteq_{{\sepa}}      A,   B'\subseteq_{{\sepa}}      B}$ and
$(   A'\bar{\otimes}_{\max}    B')_{   A'\subseteq_{{\sepa}}      A,   B'\subseteq_{{\sepa}}      B}$ are {isomorphic} in $\Ind(\nCalg_{\sepa})$.

For the moment we fix $   A'$ and $   B'$.  We  then get the outer part of the following diagram
$$\xymatrix{A'\otimes_{\max}    B'\ar[rr]\ar[dr]\ar[dd]&&   A'\bar{\otimes}_{\max}    B'\ar@{..>}[dl]\ar[dd]\\
&   A''\otimes_{\max}    B''\ar[ld]&\\    A\otimes_{\max}    B\ar@{=}[rr]&& A\otimes_{\max} B}$$
The vertical morphisms are the canonical inclusions into the colimits which have been identified using  \eqref{dvavdfsvdvsvqrwvdsfva1} and  \eqref{dvavdfsvdvsvqrwvdsfva}. 
 
 Since   $   A'\otimes_{\max}    B'$ is separable, the kernel $   I$ of $   A'\otimes_{\max}    B'\to    A'\bar{\otimes}_{\max}    B'$ is {separable. 
 By} Example \ref{wegoijwoegwergwregwreg}
     the poset of separable subalgebras in a $C^{*}$-algebra is \countably filtered\footnote{{This means that} every countable subset admits an upper bound.}. 
     Since $G$ is countable, using Lemma \ref{eroigjowregwrege9} we can find   invariant separable subalgebras $   A''$ and $   B''$ containing $   A'$ and $   B'$, respectively, such that $   I $ is  annihilated by  the map
$   A'\otimes_{\max}    B'\to    A''\otimes_{\max}    B''$. This provides the dotted morphism.  The existence of $(   A'',   B'')$ and  the dotted arrow for any given $(   A',   B') $ proves the claim.

 The claim justifies the equivalence marked by {$(!)$} in the following {chain of equivalences}
\begin{eqnarray*}
\kkG(   A) \otimes_{\max}  \kkG(   B) & \simeq& \colim\limits_{   A' \subseteq_{{\sepa}}      A,    B' \subseteq_{{\sepa}}   B} \kkGs(   A') \otimes_{\max} \kkGs(   B') \\
	& \stackrel{\simeq}{\to}& \colim\limits_{   A' \subseteq_{{\sepa}}      A,    B' \subseteq_{{\sepa}}   B} \kkGs(   A' \otimes_{\max}     B')  \\&{\stackrel{(!)}{\simeq}} & \colim\limits_{   A' \subseteq_{{\sepa}}      A,    B' \subseteq_{{\sepa}}   B} \kkGs(   A' \bar \otimes_{\max}     B')\\
	& \stackrel{}{{\simeq}} &\kkG(   A \otimes_{\max}    B)
\end{eqnarray*}
providing the equivalence \eqref{qefoqwepofqwefqewffqewfqwef} in the case of the maximal tensor product.
The remaining equivalences in this chain are justified in the same way as in the case of the minimal tensor product.
\end{proof}
  
%
%
%

\section{Change of groups functors}
\label{sec:change-of-groups}

In this section we show that the restriction, induction and crossed product functors  of $C^*$-algebras
descend to functors between the corresponding universal KK-theoretic stable $\infty$-categories. The following lemma is the blue-print for
the assertions about the change of groups functors below.
 
 Let $G$ and $H$ be  countable groups and consider a  functor
  $$C \colon \Fun(BG,\nCalg ) \to  \Fun(BH,\nCalg ) \, .$$  
  If $A\to B$ is a morphism in 
 $\Fun(BG,\nCalg ) $, then we let ${C(A)}^{C(B)}$ denote the  image of the induced morphism $C(A)\to C(B)$. If   $C$  preserves separable algebras and $A$ is  separable, then $C(A)^{C(B)}$ is separable, too.
 
 {For any category $\cC$ let $\Ind(\cC)$ denote the category of inductive systems in $\cC$.
 For $A$ in $ \Fun(BG,\nCalg ) $ we consider the inductive system
 $(A')_{A'{\subseteq_{\sepa} A}}$  of the invariant separable subalgebras $A'$ of $A$   in  $\Ind(\Fun(BG,\nCalg ))$.
 We then  have a canonical map of inductive systems
 $(C(A'))_{A'{\subseteq_{\sepa} A}}\to (C(A')^{C(A)})_{A' \subseteq_{\sepa} A}$ in $\Ind(\Fun(BH,\nCalg ))$.
 
 {Assume that $C$ preserves separable algebras.}
 \begin{ddd}\label{wtoijwrgergwrefwef}
 We say that $C$ is $\Ind$-s-finitary if it has the following properties:
 \begin{enumerate}
 \item\label{egijegoergwergwreg1} For every $A$ in $ \Fun(BG,\nCalg ) $ the inductive system $(C(A')^{C(A)})_{A'\subseteq_{\sepa}A}$  is cofinal in the  inductive system of all invariant  separable
 subalgebras  of $C(A)$.
 \item\label{egijegoergwergwreg2} The canonical map $(C(A'))_{A'{\subseteq_{\sepa} A}}\to (C(A')^{C(A)})_{A'{\subseteq_{\sepa} A}}$ is an isomorphism in
 $\Ind(  \Fun(BH,\nCalg ))$. 
 \end{enumerate}
 \end{ddd}}

{This definition is designed in order to ensure the following fact.
\begin{lem}\label{weiogwoegwer9}
If $F$ is some $s$-finitary functor on  $\Fun(BH,\nCalg ) $ and $C$ is $\Ind$-s-finitary, then the composition
$F\circ C$ is   an  $s$-finitary functor on  $\Fun(BG,\nCalg ) $.
\end{lem}
\begin{proof}
{For any $A$  in $\Fun(BH,\nCalg)$ we must show that the canonical morphism \begin{equation}\label{giewruhgwiueghewirfwefwerf}
\colim_{A'\subseteq_{\sepa} A} F(C(A'))\to F(C(A))
\end{equation}
is an equivalence. It first follows from Condition \ref{wtoijwrgergwrefwef}.\ref{egijegoergwergwreg2}  that
$$\colim_{A'\subseteq_{\sepa} A} F(C(A'))\stackrel{\simeq}{\to}
 \colim_{A'\subseteq_{\sepa} A} F(C(A')^{C(A)})\, .$$
By Condition   \ref{wtoijwrgergwrefwef}.\ref{egijegoergwergwreg1}
we have an equivalence 
  $$ \colim_{A'\subseteq_{\sepa} A} F(C(A')^{C(A)})\stackrel{\simeq}{\to} \colim_{B'\subseteq_{\sepa} C(A)} F(B')
 \, .$$
 Finally, since $F$ is $s$-finitary, we have an equivalence 
  $$\colim_{B'\subseteq_{\sepa} C(A)} F(B')\stackrel{\simeq}{\to} F(C(A))\, .$$
  The composition of these equivalences is the desired equivalence \eqref{giewruhgwiueghewirfwefwerf}.}
\end{proof}

}

In order to check that a functor $C\colon \Fun(BG,\nCalg)\to \Fun(BH,\nCalg)$ is $\Ind$-s-finitary we  will use the following lemma.
\begin{lem}\label{qoifjqoifewqqfwedwed}
Assume that $C $  {preserves separable algebras,   satisfies the Condition \ref{wtoijwrgergwrefwef}.\ref{egijegoergwergwreg1}} and  one of the following:
\begin{enumerate}
\item\label{wegiowergwergwerf} $C$ preserves \countably filtered colimits.
\item\label{wegiowergwergwerf1}  $C$ preserves  isometric inclusions.\footnote{Note that inclusions of $C^{*}$-algebras are automatically isometric.}  \end{enumerate}
 Then $C$ is $\Ind$-s-finitary.
\end{lem}
\begin{proof}
We must show that each of the Assumptions 
\ref{wegiowergwergwerf} or \ref{wegiowergwergwerf1} implies Condition \ref{wtoijwrgergwrefwef}.\ref{egijegoergwergwreg2}.

In the case of Assumption \ref{wegiowergwergwerf1} we actually have isomorphisms
$C(A')\stackrel{\cong}{\to} {C(A')^{C(A)}}$ for every invariant separable subalgebra $A'$ of $A$  so that Condition \ref{wtoijwrgergwrefwef}.\ref{egijegoergwergwreg2} is clear. 

We now consider the more complicated case of Assumption \ref{wegiowergwergwerf}.
   The argument is very similar to the case of 
$\otimes_{\max}$ in the corresponding part of the proof of Proposition \ref{togjoigerggwgergwgr1}.
 We consider an invariant  separable subalgebra $A'$ of $A$ giving the outer part of the following diagram 
 \begin{equation}\label{sfgsdgsfgsfgfd}  \xymatrix{C(A') \ar[rr]\ar[dr]\ar[dd]&& {C(A')^{C(A)}} \ar@{..>}[dl]\ar[dd]\\
&   C(A'') \ar[ld]&\\    C(A)  \ar@{=}[rr]&& C(A) }
\end{equation}
 Every $C^{*}$-algebra is isomorphic to the colimit of its separable subalgebras.
 In the case of $A$ in $\Fun(BG,\nCalg)$   for a 
 countable group $G$ we have the same assertion for the system of  invariant separable subalgebras $(A')_{A'{\subseteq_{\sepa} A}}$, i.e., 
we have an isomorphism $ \colim_{A'{\subseteq_{\sepa} A}} A' \cong A$ in $\Fun(BG,\nCalg)$.
The  poset of invariant separable subalgebras of $A$ is \countably filtered  (see Example \ref{wegoijwoegwergwregwreg}). 
Since $C$ preserves  \countably filtered colimits we have
$ \colim_{A'{\subseteq_{\sepa} A}} C(A') \cong C(A)$, and the left vertical arrow is the canonical inclusion into the colimit.  
Let $I$ be the kernel of the map
$C(A')\to {C(A')^{C(A)}} $. Since $C(A')$ is separable, also $I$ is separable. Furthermore it is contained in (in fact equal to)  the kernel of $C(A')\to C(A)$.   
By   Lemma \ref{eroigjowregwrege9} 
we find an invariant  separable subalgebra $A''$ of $A$
such that $I$ is annihilated by $C(A')  \to 
C(A'')$. 
 This implies the existence of the dotted arrow. 
The existence of $A''$ for given $A'$ shows that
the canonical map of inductive systems 
$(C(A'))_{A'{\subseteq_{\sepa} A}}\to ( {C(A')^{C(A)}} )_{A'{\subseteq_{\sepa} A}}$ has an inverse in $\Ind(\Fun(BH,\nCalg))$.
 \end{proof}

\color{black}

{We set  $C_{s} \coloneqq C_{| \Fun(BG,\nCalgs )}$.}
\begin{lem} \label{qoi3rgjowrtgwegwergwerg}Assume:
  \begin{enumerate}
\item  \label{wetoigwtgwwefwefwefergrwegwreg} $C$  is {$\Ind$-s}-finitary. 
   \item \label{wergwrpogkrpogwreg} {$C_{s}$ preserves countable sums.}
   \item  \label{wetoigwtgergrwegwreg}The composition $\kkHs\circ C_{s}$ inverts ${\kkGsk}$-equivalences.   
   \item   \label{wetoigwtgergrwegwreg2} $C_{s}$   preserves   semisplit exact sequences.
     \end{enumerate}
Then {there are functors $F$ and $F_s$ taking part in the} following commutative diagram
\begin{equation}\label{tgiohoigwegregwegw}
\xymatrix{
&&\KKGs \ar[d]_{F_{s}}\ar@/^1cm/[ddd]^{y^{G}}\\
\Fun(BG,\nCalg_{\sepa})\ar[urr]^-{\kkGs} \ar[r]_{C_{s}}\ar[d]^{\incl}&\ar[d]^{\incl}\Fun(BH,\nCalg_{\sepa})\ar[r]_-{\kkHs}&\KKHs\ar[d]^{y^{H}}\\
\ar[drr]_-{\kkG}\Fun(BG,\nCalg )\ar[r]^{C}&\Fun(BH,\nCalg )\ar[r]^-{\kkH}&\KKH\\
&&\KKG\ar[u]^{F}
}
\end{equation}
{In addition:}
 \begin{enumerate}
 \item $F_{s}$ is exact  {and preserves countable coproducts}.
 \item $F$ preserves colimits and compact objects.
 \end{enumerate}
 Both $F$ and $F_s$ are characterized by the diagram  and these properties.
    \end{lem}
 \begin{proof} 

%
%
%
%
 The  existence of the factorization $C_{s}$ and the commutativity of the left square in \eqref{tgiohoigwegregwegw} expresses the fact that $C$ preserves separable algebras
 {as part of Assumption \ref{wetoigwtgwwefwefwefergrwegwreg}.}
 The middle square  is an instance of \eqref{qfwoefjkqwpoefkewpfqewfqwfqfwefq}.  
 
 The existence of the factorization  $F_{s}$ follows from Assumption \ref{wetoigwtgergrwegwreg} and the defining universal property of $\kkGs$ as a Dwyer--Kan localization.   
 Using  Assumption \ref{wetoigwtgergrwegwreg2} and the semiexactness of $\kkH$ 
 we conclude that $\kkHs\circ C_{s}$ is semiexact, too. 
  This implies by Theorem \ref{4ogijwrtgwergerwgwergegergwerg} that $F_{s}$ is exact.  To see that $F_s$ preserves countable coproducts, consider a countable family ${(A_i)_{i\in I}}$ of objects of $\Fun(BG,\nCalg_\sepa)$. 
  As seen in the proof of Lemma \ref{ugoierugogergwergwregw}
  countable sums in $\KKGs$ are presented by countable sums in $\Fun(BG,\nCalg_{\sepa})$. Furthermore, by Lemma~\ref{ugoierugogergwergwregw}{.\ref{regpoqgowergwergwregwregwreg}}, $\kkGs$ sends countable sums to countable coproducts. {This implies the  equivalences marked by $(1)$ in the chain} 
{ \begin{align*} F_s(\bigoplus\limits_{i\in I} \kkGs(A_i))&\stackrel{(1)}{ \simeq}F_s(\kkGs(\bigoplus\limits_{i\in I} A_i)) \stackrel{(2)}{ \simeq} \kkHs(C_s(\bigoplus\limits_{i\in I} A_i)) \stackrel{(3)}{ \simeq}  \kkHs(\bigoplus\limits_{i\in I}  C_s( A_i))\\&\stackrel{(1)}{ \simeq} \bigoplus\limits_{i\in I} \kkHs(C_s(  A_i))\stackrel{(2)}{ \simeq}\bigoplus\limits_{i\in I} F_{s}(\kkGs(A_{i})) \ .\end{align*}}{The equivalences marked by $(2)$ are  given by  the upper triangle in \eqref{tgiohoigwegregwegw}.} 
{In order to argue that $(3)$ is an equivalence we {use} that $C_s$ preserves countable sums of algebras. 
{Using} once more that $\kkHs$ sends countable sums to countable coproducts, we obtain that $F_s$ preserves countable coproducts.}

  It follows from Assumption \ref{wetoigwtgwwefwefwefergrwegwreg} in combination with  {{Lemma} \ref{weiogwoegwer9} and the fact that $\kkH$ is s-finitary}  that also $\kkH\circ C$ is s-finitary. Using the 
   existence of $F_{s}$ together with  Proposition \ref{qwoiefuoqrfwewfqwef9} we then verify that $\kkH\circ C$ satisfies  the assumptions of   Theorem  \ref{ojroiwfewfqefqef}.      We therefore    obtain the colimit-preserving  factorization $F$ from the
  universal property   of  $\kkG$.
  
 We claim that then also the right square commutes, i.e., that $y^{H}\circ F_{s}\simeq F\circ y^{G}$. In order to see this we observe that the fillers of the other squares provide an equivalence  $y^{H}\circ F_{s}\circ \kkGs\simeq F\circ y^{G}\circ \kkGs$. Then the desired equivalence is  again a consequence of the defining universal property of $\kkGs$ as a Dwyer--Kan localization.  
 
 The
   compact objects in $\KKG$ are given by the essential image of $y^{G}$  since $\KKGs$ is idempotent complete by the Theorem \ref{qroifjeriogerggergegegweg}.\ref{qoirwfjhqoierggrg4}. Therefore
  the 
   commutativity of the right square implies that $F$ also preserves compact objects.
\end{proof}

\begin{rem}\label{wetoigjowergerfgrfw} Note that $F$ can alternatively be described as the functor $\Ind(F_{s})$ obtained by the applying the $\Ind$-completion functor to $F_{s}$. 

The fact that $F$ preserves colimits and compact objects implies that
$F$ admits a right-adjoint  which preserves filtered colimits. By stability  
it  hence preserves all colimits and therefore has a further right-adjoint.
Below we will describe these right-adjoints explicitly in some cases. 
\hB
\end{rem}

%
%
%

Let $H\to G$ be a homomorphism of  countable  groups. 
The following lemma verifies Theorem \ref{reguergiweogwergegwergwerv}.\ref{reguergiweogwergegwergwerv1}.

\begin{lem}\label{erwoigjwegerregwefwefewfewfwefwefwefwefwefeg}
We have the following commutative diagram 
\begin{equation}\label{vweivhweviuhevievvwevevevewvev}
 \xymatrix{
 &&\KKGs \ar[d]_{\Res^{G}_{H,s}}\ar@/^1cm/[ddd]^{y^{G}}\\
 \Fun(BG,\nCalg_{\sepa})\ar[urr]^-{\kkGs} \ar[r]_{\Res^{G}_{H,s}}\ar[d]^{\incl}&\ar[d]^{\incl}\Fun(BH,\nCalg_{\sepa})\ar[r]_-{\kkHs}&\KKHs\ar[d]^{y^{H}} \\
 \ar[drr]_-{\kkG}\Fun(BG,\nCalg )\ar[r]^{\Res^{G}_{H}}&\Fun(BH,\nCalg )\ar[r]^-{\kkH}&\KKH\\
 &&\KKG\ar[u]^{\Res^{G}_{H}}
 }
\end{equation} 
where
 \begin{enumerate}
 \item $\Res^{G}_{H,s} {\colon \KKGs \to \KKHs}$ is exact {and preserves countable coproducts}.
 \item $\Res^{G}_{H} {\colon \KKG \to \KKH}$ preserves colimits and compact objects. 
 
 \end{enumerate}
%
%
%
%
\end{lem}
\begin{proof} The assertions follow from  Lemma \ref{qoi3rgjowrtgwegwergwerg} once we have verified its assumptions.
It is obvious that $\Res^{G}_{H}$ preserves separable algebras{,}
 {semisplit exact sequences {and} countable sums.}
   It is furthermore also obvious that the composition ${\kkHs}\circ \Res^{G}_{H,s}$  is reduced, 
  $\mathbb{K}^{G}$-stable,   
  and   semiexact.
 By the Theorem \ref{wtohwergerewgrewgregwrg1} it therefore inverts ${\kkGsk}$-equivalences.

 It remains to show that $C$ is $\Ind$-s-finitary.
 {We first verify}  Condition  \ref{wtoijwrgergwrefwef}.\ref{egijegoergwergwreg1}. 
 Let $A$  be in $\Fun(BG,\nCalg)$ and $B'$ be an $H$-invariant separable subalgebra of $\Res^{G}_{H}(A)$. Since $G$ is countable the $G$-invariant  subalgebra $A'$  generated by 
 $B'$ is again separable and we have $B'\subseteq \Res^{G}_{H}(A')$.
 Condition \ref{wtoijwrgergwrefwef}.\ref{egijegoergwergwreg2} follows from Lemma \ref{qoifjqoifewqqfwedwed} since $\Res^{G}_{H}$ preserves isometric inclusions.
    \end{proof}
 
%
%
%

 \begin{rem} From now one we will use the same notation $\Res^{G}_{H}$ for all restriction functors
induced by $H\to G$ and hope that this does not produce confusion as the argument of the functor determines which version has to be considered. The same convention will later also apply to the functors $\Ind_{H}^{G}$ and $-\rtimes_{{?}} H$ considered below.
\hB
\end{rem}

The following corollary generalizes  \cite[(12)]{MR2193334} from the triangulated to the 
stable $\infty$-categorical level. 
\begin{kor}
The  functor
$\Res^{G}_{H} \colon \KKG\to \KKH$ has a symmetric monoidal refinement for {both symmetric monoidal structures $\otimes_{\min}$ and $\otimes_{\max}$.}
\end{kor}
\begin{proof} 
We let $\otimes $ denote one the symmetric monoidal structures $\otimes_{\min}$ or $\otimes_{\max}$.
 
We first show the claim that  
 the functor $\Res^G_{H,s}\colon \KKGs \to \KKHs$ has a symmetric monoidal refinement. 
 We consider the bold part of the diagram
\[\begin{tikzcd}
	\Fun(BG,\nCalg_\sepa) \ar[r,"F"] \ar[d,"\kkGs"] & \Fun(BH,\nCalg_\sepa) \ar[d,"\kkHs"]  \\
	\KKGs \ar[r,dashed,"\bar{F}"] & \KKHs
\end{tikzcd}\ .\]
Assuming that $\kkHs\circ F$ sends $\kkGsk$-equivalences to equivalences we obtain the factorization $\bar F$ indicated by the dashed arrow. Note that $\kkHs$ and $\kkGs$ have a  symmetric monoidal refinements $\kkHstensora$ and $\kkGstensora$    by Proposition~\ref{togjoigerggwgergwgr}. 
We now use 
  \cite[Prop.\ 3.2.2]{hinich}  saying that
  $\kkGstensora$ has the universal property of a Dwyer-Kan localization
  for symmetric monoidal functors which is compatible with the operation of
  taking underlying functors.
     Applying this to
 the composition $\kkHstensora\circ F^{\otimes}$ we get a symmetric monoidal refinement $\bar F^{\otimes}$ of $\bar F$.
Applying this argument to $\Res^{G}_{H,s}$ in place of $F$ shows that claim. 

%

%

 The ind-completion functor  $\Ind \colon \Cat^{\exa}_{\infty} \to \mathrm{Pr}^{L}_\mathrm{st}$ is a symmetric monoidal functor  and hence induces a functor on commutative algebras. Interpreting the symmetric monoidal functor
   $ \Res^{G,\otimes}_{H,s}:
 \KKGstensora\to \KKHstensora$ as a morphism of {commutative} algebras in $  \Cat^{\exa}_{\infty} $,
 we get the desired symmetric monoidal refinement
  $\Res^{G,\otimes}_{H}:=\Ind( \Res^{G,\otimes}_{H,s}):
   \KKGtensora\to  \KKHtensora$ of the functor $\Res^{G}_{H}$.
\end{proof}

 We  now assume that $H$ is a subgroup of $G$.  
 The following construction describes the induction functor $\Ind_{H}^{G}$.
\begin{construction}\label{qroigjoerwgergrewgregweg}
 For $B$ in $\Fun(BH,\nCalg)$ we let $C_{b}(G,   B)$ denote the $C^{*}$-algebra of  bounded $B$-valued functions on $G$ with the $\sup$-norm.  The group $G$ acts on $C_{b}(G,   B)$ by $(g,f)\mapsto (g'\mapsto f(g^{-1}g'))$. 
We define the induction functor
$$ \Ind_{H}^{G} \colon \Fun(BH,\nCalg)\to \Fun(BG,\nCalg)$$ as follows:
\begin{enumerate}
\item  objects: The induction functor 
  sends $   B$ in $\Fun(BH,\Calg)$ to the invariant subalgebra   $ \Ind_{H}^{G}(   B)$ of $C_{b}(G,B)$ generated    by the  functions
$f \colon G\to    B$ satisfying
\begin{enumerate}
\item \label{ergjweiogergregwgr} $f(gh)=h^{-1}f(g)$ for all $h$ in $H$ and $g$ in $G$.
\item  \label{ergjweiogergregwgr1} The 
  projection of $\supp(f)$ to $G/H$ is finite.
  \end{enumerate}
  \item 
morphisms: The induction functor sends a  morphism $\phi \colon    B\to    B'$ to the morphism
$\Ind_{H}^{G}(\phi) \colon \Ind_{H}^{G}(   B)\to \Ind_{H}^{G}(   B')$ given by
$\Ind_{H}^{G}(\phi)(f)(g) \coloneqq \phi(f(g))$.
\end{enumerate} 
  We have a natural transformation of functors
  \begin{equation}\label{wervwihiowecwecdscac}
  \iota \colon \id\to \Res^{G}_{H}\circ \Ind^{G}_{H} \colon \Fun(BH,\nCalg)\to \Fun(BH,\nCalg)\, .
  \end{equation}
  Its evaluation  at  $  B$   in $\Fun(BH ,\nCalg)$ is given by the homomorphism
  \[
  \iota_{  B} \colon   B\to  \Res^{G}_{H} ( \Ind^{G}_{H}(  B))
  \]  
which sends $b$ in $  B$ to the function
\[
G\to   B\ , \quad g\mapsto \begin{cases} g^{-1}(b) & \text{ if }g \in H\,,\\0 & \text{ if }g \notin H\,,\end{cases}
\]
on $G$.
\hB
\end{construction}

We frequently need the following well-known result.
Let $B$ be in $\Fun(BH,\nCalg)$ and $A$ be in $\Fun(BG,\nCalg)$.
\begin{lem}\label{rehoijerthtrhtrehtrheg}
For $?$ in $\{\min,\max\}$
we have a canonical isomorphism
\begin{equation}\label{ergqrqffewfqfef}
 \Ind_{H}^{G}(B)\otimes_{?} A \cong \Ind_{H}^{G}(B\otimes_{?} \Res^{G}_{H}(A))  \, .
\end{equation}
\end{lem}
\begin{proof}

Let $\iota_{B} \colon \Ind_{H}^{G}(B)\to C_{b}(G,B)$ denote the canonical inclusion. By construction of the  induction functor we have the following  commutative square
in $\Fun(BG,\nCalg)$
$$\xymatrix{C_{b}(G,B)\otimes_{?}A\ar[r] &C_{b}(G,B\otimes_{?}A)\ \\\Ind_{H}^{G} (B)\otimes_{?}A\ar[r]\ar[u]^{\iota_{B}\otimes \id_{A}}&\Ind^{G}_{H} (B\otimes_{?}\Res^{G}_{H}(A))\ar[u]^{\iota_{B\otimes \id_{A}}}}$$
In order to show that the lower horizontal map is an isomorphism
   we choose a section $r$ of the projection map $G\to G/H$ and extend the diagram non-equivariantly as follows: 
  $$\xymatrix{C_{b}(G,B)\otimes_{?}A\ar[r]  \ar[d]^{r^{*}\otimes \id_{A}}&C_{b}(G,B\otimes_{?}A)\ar[d]^{r^{*}}  \\C_{b}(G/H,B)\otimes_{?}A \ar[r] &C_{b}(G/H,B\otimes_{?}A)\\C_{0}(G/H,B)\otimes_{?}A  \ar[u]\ar[r]^{!}&C_{0}(G/H,B\otimes_{?}A)\ar[u]\\ \Ind_{H}^{G} (B)\otimes_{?}A\ar@/^2.5cm/[uuu]^{\iota_{B}\otimes \id_{A}}\ar[u]_{\cong}\ar[r]&\ar@/^-2.5cm/[uuu]_{\iota_{B \otimes \id_{A}}}\ar[u]_{\cong}\Ind^{G}_{H} (B\otimes_{?}\Res^{G}_{H}(A))}\ .$$
  The lower vertical isomorphisms are 
  immediate consequences of the construction of the induction functor.
  The  arrow marked by $!$ is an isomorphism
  because of
  $$C_{0}(G/H,B)\otimes_{?}A\cong C_{0}(G/H)\otimes_{?} B\otimes_{?}A\cong C_{0}(G/H,B\otimes_{?}A)\, .$$
  Hence the lower horizontal morphism is an isomorphism, too. 
%
  \end{proof}

The following lemma verifies Theorem  \ref{reguergiweogwergegwergwerv}.\ref{reguergiweogwergegwergwerv2}.
 \begin{lem}\label{ewoigjeorgregerwgergew} 
 We have the following commutative diagram 
\begin{equation}\label{vweivhweviuhevievvwevevevewvevdddddd}
 \xymatrix{
 &&\KKHs \ar[d]_{\Ind^{G}_{H,s}}\ar@/^1cm/[ddd]^{y^{H}}\\
 \Fun(BH,\nCalg_{\sepa})\ar[urr]^-{\kkHs} \ar[r]_-{\Ind^{G}_{H,s}}\ar[d]^{\incl}&\ar[d]^{\incl}\Fun(BG,\nCalg_{\sepa})\ar[r]_-{\kkHs}&\KKGs\ar[d]^{y^{G}} \\
 \ar[drr]_-{\kkH}\Fun(BH,\nCalg )\ar[r]^{\Ind^{G}_{H}}&\Fun(BG,\nCalg )\ar[r]^-{\kkH}&\KKG\\
 &&\KKH\ar[u]^{\Ind^{G}_{H}}
 }
 \end{equation} 
 where
 \begin{enumerate}
 \item $\Ind^{G}_{H,s} {\colon \KKHs \to \KKGs}$ is exact {and preserves countable coproducts}.
 \item $\Ind^{G}_{H} {\colon \KKH \to \KKG}$ preserves colimits and compact objects.
 \end{enumerate}
%
%
%
\end{lem}
\begin{proof} 
The assertion will again follow from  Lemma \ref{qoi3rgjowrtgwegwergwerg}.
In comparison with the case of $\Res^{G}_{H}$ the verification of Assumption  \ref{qoi3rgjowrtgwegwergwerg}.\ref{wetoigwtgergrwegwreg}
 is considerably more complicated. It   will employ  non-formal results of Kasparov  \cite{kasparovinvent} which we will use in the form stated in \cite{MR2193334}  (see Remark \ref{wegjoiergergerwge}). 
 
It is easy to see that $\Ind^{G}_{H}$ preserves separable algebras {and}   sums.
{In order to see that it is $\Ind$-s-finitary we must check the Conditions  \ref{wtoijwrgergwrefwef}.\ref{egijegoergwergwreg1} and  \ref{wtoijwrgergwrefwef}.\ref{egijegoergwergwreg2}. 
Note that $\Ind_{H}^{G}$ preserves isometric inclusions.
Let $A$ be in $\Fun(BH,\nCalg)$ and $B'$ be a separable  
subalgebra of $\Ind_{H}^{G}(A)$. Then we let $A'$ be the $H$-invariant subalgebra of $A$ generated by the values $f(g)$ for all $f$ in $B$ and $g$ in $G$ (see Construction \ref{qroigjoerwgergrewgregweg} for notation). Since $G$ is countable $A'$ is separable. Furthermore $B'\subseteq \Ind_{H}^{G}(A')$. Finally, 
 Condition   \ref{wtoijwrgergwrefwef}.\ref{egijegoergwergwreg2} follows from 
 Lemma \ref{qoifjqoifewqqfwedwed} using again that $\Ind_{H}^{G}$ preserves isometric inclusions.}


{Next we show that}   $\Ind_{H}^{G}$  {preserves} semisplit {exact} sequences.   To this end let $   p \colon   A\to    B$ be a surjective morphism in $\Fun(BH,\nCalg)$ and $s \colon   B\to    A$ be an equivariant  completely positive contraction
such that $   p\circ    s=\id_{   B}$. 
Then we define (using notation from Construction \ref{qroigjoerwgergrewgregweg}) a map of vector spaces $\Ind_{H}^{G}(   s) \colon \Ind_{H}^{G}(   B)\to \Ind_{H}^{G}(   A)$
by $\Ind_{H}^{G}(   s)(f)(g) \coloneqq   s( f(g))$.  This map preserves the generators (i.e., functions satisfying the Conditions \ref{qroigjoerwgergrewgregweg}.\ref{ergjweiogergregwgr} and \ref{qroigjoerwgergrewgregweg}.\ref{ergjweiogergregwgr1})  by the linearity, continuity and equivariance of $   s$.
We now show that $\Ind_{H}^{G}(   s)$ extends by continuity  to a completely positive contraction.
We choose a section of the projection map $G\to G/H$. The restriction along this section
 induces an isomorphism of $C^{*}$-algebras 
$$\Res^{G}(\Ind_{H}^{G}(   B))\cong C_{0}(G/H)\otimes_{\min} \Res^{H}(   B)$$
and similarly for $   A$. Under this identification the map $\Res^{G}(\Ind_{H}^{G}(   s))$ acts as
  $\id_{C_{0}(G/H)}\otimes  s$. This map  extends by continuity to a completely positive contraction.
For this last step we use that completely positive contractions  are compatible with minimal tensor products {\cite[Thm.\ 3.5.3]{brown_ozawa}}.

 {We now show that t}he composition $\kkGs\circ \Ind_{H,s}^{G}\colon \Fun(BH,\nCalg_{\sepa})\to \KKGs$ inverts
  ${\kkHsk}$-equivalences.
 {We give} a short argument using results from the literature.
  Using the adjunction {$(\Res^{G}_{H},\Ind^{G}_{H})$ on the level of 
  triangulated categories}
  \cite[(20)]{MR2193334} and the Yoneda lemma we see that the composition 
  $$\ho\circ {\kkGs}\circ  \Ind_{H,{s}}^{G} \colon \Fun(BH,\nCalg_{\sepa})\to \KKGsk$$
  sends ${\kkHsk}$-equivalences to isomorphisms.
  Since $\ho$ detects equivalences {we 
  conclude} that  ${\kkGs}\circ  \Ind_{H,{s}}^{G}$
  sends    ${\kkHsk}$-equivalences to  equivalences.   
%
%
%
%
\end{proof}

 \color{black}
  \begin{rem}\label{wegjoiergergerwge}
  The details of the arguments leading to  \cite[(20)]{MR2193334} are not very well documented in the literature.\footnote{The text before (9) in {\cite{MR2193334}}     suggests that the authors only consider $\mathbb{K}$-stability while one needs $\mathbb{K}^{G}$-stability in order to apply the universal property {from Corollary} \ref{ioqerjgqergrqfefqewfq}.}
   Therefore we sketch now an alternative argument for  {the fact that $\kkGs\circ \Ind_{H,s}^{G}\colon \Fun(BH,\nCalg_{\sepa})\to \KKGs$ inverts
  $\KKHsk$-equivalences}
  which is close to the general philosophy of the present paper.  We consider the $G$-Hilbert spaces  $   V_{G} \coloneqq {L^{2}(G)}$ and  $   V_{G}' \coloneqq \C\oplus    V_{G}$, where $\C$ has the trivial $G$-action.
  We define $$\widetilde{\Ind}^{G}_{H} \colon \Fun(BH,{\nCalg})\to \Fun(BG,{\nCalg} )$$
  by
  \begin{equation}\label{erwglkgnolgwgergwerg}
\widetilde{\Ind}^{G}_{H}(   B) \coloneqq \Ind_{H }^{G}(   B\otimes_{\min} \Res^{G}_{H}K(    V_{G}))\, .
\end{equation}
  Since $K(    V_{G})$ is separable, the functor   $\widetilde{\Ind}^{G}_{H}$   preserves separable algebras. As before we indicate the restriction of an induction functor to separable algebras by a subscript `s', but in order to simplify the notation we drop this subscript at restriction functors.
  We claim:
  \begin{enumerate}
  \item \label{qerpgojgoerqlmfl} $\kkGs \circ \widetilde{\Ind}^{G}_{H {,s} } $ sends  ${\kkHsk}$-equivalences to equivalences.
  \item \label{wegoiuergregwegergw}We have an equivalence
  $ \kkGs \circ  \Ind^{G}_{H {,s}} \simeq {\kkGs}\circ \widetilde{\Ind}^{G}_{H,s}$.
  \end{enumerate}
 Both assertions together imply {that $\kkGs\circ \Ind_{H,s}^{G}$ inverts ${\kkHsk}$-equivalences.}


We {first} show Assertion \ref{wegoiuergregwegergw} of the claim as follows.
We consider the diagram
$$   B\to      B\otimes_{\min} \Res^{G}_{H}K(   V'_{G} )\leftarrow       B\otimes_{\min}  \Res^{G}_{H} K(   V_{G})$$
where the maps are induced by the obvious $G$-equivariant isometric embeddings   
$\C\to    V_{G}'$ and $    V_{G}\to     V'_{G}$.
We now assume that $   B$ is separable, apply $ {\kkGs}\circ \Ind_{H{,s}}^{G}$ and use \eqref{ergqrqffewfqfef} in order to get   
\begin{align*}
\mathclap{
{\kkGs}(\Ind^{G}_{H{,s}}(   B))\to {\kkGs}(\Ind_{H{,s}}^{G}  (    B) \otimes_{\min}  K(    V_{G}'))\leftarrow   {\kkGs}(   \Ind_{H{,s}}^{G}(    B) \otimes_{\min} K(    V_{G} )) \stackrel{\text{def.}}=  {\kkGs}(\widetilde{\Ind}_{H{,s}}^{G}(   B)) \, .
}
\end{align*}
By {$\mathbb{K}^{G}$-}stability of $ {\kkGs}$, the first two arrows are equivalences.  
The whole construction is natural in $   B$ and provides the equivalence claimed in  Assertion   \ref{wegoiuergregwegergw}.

In order to show Assertion \ref{qerpgojgoerqlmfl} we show that
\begin{equation}\label{jqwelkfjwelfkqrwrgrwgr}
 {\kkGs}\circ \widetilde{\Ind}_{H{,s}}^{G} \colon \Fun(BH,\nCalg_{\sepa})\to {\KKGs}
\end{equation}
is reduced, 
 $\mathbb{K}^{{H}}$-stable and semiexact.  Then we apply the universal property of $ \kkHs$ stated in {Theorem}  \ref{wtohwergerewgrewgregwrg1}.
 
 Since $\widetilde{\Ind}_{H{,s}}^G(0) = 0$ and $\kkGs$ is reduced, {the functor in  \eqref{jqwelkfjwelfkqrwrgrwgr}} is reduced as well.

%
%
%
%


Since the operations $-\otimes_{\min}  K(   V_{G})$ and
$\Ind_{H{,s}}^{G}$ (as seen {in the proof of Lemma \ref{ewoigjeorgregerwgergew}}) preserve   semisplit exact sequences 
we conclude that $\widetilde{\Ind}_{H{,s}}^{G}$ also preserves  semisplit exact sequences. 
Since ${\kkGs}$ is semiexact we conclude that  \eqref{jqwelkfjwelfkqrwrgrwgr}   is semiexact.


The most complicated part of the argument is $\mathbb{K}^{{H}}$-stability.  
Let $   V\to    V'$ be a unitary inclusion of non-trivial $H$-Hilbert spaces and $   B$ be in $\Fun(B{H},\nCalg_{\sepa})$.
We must show that the induced map
\begin{equation}\label{gwergrewgregwregg}
\kkGs(\widetilde{\Ind}^{G}_{H,s}(   B\otimes_{\min} K(   V)))\to \kkGs(\widetilde{\Ind}^{G}_{H,s}(   B\otimes_{\min} K(   V')))
\end{equation}
is an equivalence.

If $   W$ is any $H$-Hilbert space, then we have an   isomorphism of $H$-Hilbert spaces
 \begin{equation}\label{arwfiojofdafedrferferf}
   W\otimes L^{2}(H)\xrightarrow{\cong}  \Res^{H}{(W)}\otimes  {L^2}(H)\, , \quad w\otimes [h]\mapsto {h^{-1}}w\otimes [h]\, ,
   \end{equation} 
where $[h]$ in $L^{2}(H)$ denotes the basis element corresponding to $h$ in $H$.
Here $H$ acts  diagonally on the left hand side, and  only on the factor $L^{2}(H)$ on  the right-hand side 

Decomposing $G$ into $H$-orbits we obtain an ${H}$-equivariant isomorphism   
\begin{equation}\label{arwfiojofdafed}
\Res^{G}_{H}(L^{2}(G))\cong L^{2}(H)  \otimes  L^{2}(G/H)  \, , 
\end{equation}
where $H$ acts by left-translations on the first tensor factor  {and trivially on the second}. Combining \eqref{arwfiojofdafedrferferf} and \eqref{arwfiojofdafed}
we get an isomorphism of $H$-Hilbert spaces 
\begin{align*}   W\otimes \Res^{G}_{H}(L^{2}(G))&\stackrel{\eqref{arwfiojofdafed}}{\cong}
    W \otimes  L^{2}(H)  \otimes L^{2}(G/H) \stackrel{\eqref{arwfiojofdafedrferferf}}{\cong}
  \Res^{H}{(W)}\otimes  L^{2}(H)  \otimes L^{2}(G/H) \\& \stackrel{\eqref{arwfiojofdafed}}{\cong} 
  \Res^{G}_{H}(\Res^{H}{(W)}\otimes   L^{2}(G))\, .
  \end{align*}
This isomorphism is natural in the $H$-Hilbert space $   W$.
 
We conclude that   the  homomorphism of $H$-algebras  $$K(   V)\otimes_{\min}  \Res^{G}_{H}(K(L^{2}(G)))\to 
K(   V')\otimes_{\min}  \Res^{G}_{H}(K(L^{2}(G)))$$ is isomorphic to the homomorphism 
$$  \Res^{G}_{H}{(K( \Res^{H} (V)\otimes  L^{2}(G)))}\to 
  \Res^{G}_{H}{(K( \Res^{H}(V')\otimes  L^{2}(G)))}\, .$$
  In view of \eqref{erwglkgnolgwgergwerg} this implies that 
   the map
  $$\widetilde{\Ind}_{H,s}^{G}(   B\otimes_{\min} K(   V) )\to 
    \widetilde{\Ind}_{H,s}^{G}(B\otimes_{\min} K(V') )$$
 is isomorphic to the map 
\begin{align*}
\mathclap{
{\Ind}_{H,s}^{G}(   B\otimes_{\min}  \Res^{G}_{H}{(K( \Res^{H} (V)\otimes  L^{2}(G))))}\to 
\Ind_{H,s}^{G}(   B\otimes_{\min}  \Res^{G}_{H}{(K( \Res^{H} (V')\otimes  L^{2}(G))))}\, .
}
\end{align*}
Applying \eqref{ergqrqffewfqfef} we furthermore see that the latter   is isomorphic to 
 $$\Ind_{H,s}^{G}(   B)\otimes_{\min}  K( \Res^{H}{(V)}\otimes  L^{2}(G))\to 
\Ind_{H,s}^{G}(   B)\otimes_{\min}  K( \Res^{H}{(V')}\otimes  L^{2}(G))\, .$$
Since the functor $\kkGs$ is $\mathbb{K}^{G}$-stable,  it sends this map to an equivalence.
Consequently, \eqref{gwergrewgregwregg}
  is an equivalence.
  \hB \color{black}
\end{rem}

%
%
%
%
%
%
%

The following corollary generalizes the projection formula  \cite[(16)]{MR2193334} from the triangulated to the $\infty$-categorical level.  
\begin{kor} \label{evigjwelgwbfgbdb} For $?$ in $\{\min,\max\}$
we have an equivalence of functors 
$$\Ind_{H}^{G}(-)\otimes_{?} (-)\simeq \Ind_{H}^{G}((-)\otimes_{?} \Res^{G}_{H}(-)) \colon \KKH\times \KKG\to \KKG\, .$$
\end{kor}
\begin{proof}
This is
  an immediate consequence of Lemma \ref{rehoijerthtrhtrehtrheg},
Lemma \ref{erwoigjwegerregwefwefewfewfwefwefwefwefwefeg}, Lemma \ref{ewoigjeorgregerwgergew} and Proposition \ref{togjoigerggwgergwgr}.  
\end{proof}

 For $A$ in $\Fun(BG,\nCalg)$ we can form the maximal and reduced crossed products $A\rtimes_{\max}G$ and $A\rtimes_{r}G$ in $\nCalg$. In the arguments below we need some  details of their construction which we will therefore recall at this point. 
 \begin{construction}\label{qreogijqeroigergwgwegwegwergw}
Both crossed products are defined as completions of the algebraic crossed product $A\rtimes^{\alg}G$. The latter is the $*$-algebra generated by elements $(a,g)$ with $a$ in $A$ and $g$ in $G$ with   multiplication $(a',g')(a,g) \coloneqq (g^{-1}(a')a,g'g)$ and the involution 
$(a,g)^{*} \coloneqq (g(a^{*}),g^{-1})$  subject to the relations $(a,g)+\lambda(a',g)=(a+\lambda a',g)$ for all $a,a'$ in $A$, $\lambda$ in $\C$ and $g$ in $G$. The maximal crossed product  $   A\rtimes_{\max}G $ is the completion of $  A\rtimes^{\alg}G$  in the maximal norm,  and the reduced  crossed product  $A\rtimes_{r}G$ is the completion of 
$A\rtimes^{\alg}G$ in the norm induced by the canonical representation on  the $A$-Hilbert $C^{*}$-module $L^{2}(G,A)$, see e.g.\ \cite[{Constr.\ 12.20}]{cank}  for an explicit formula.

If $H$ is a subgroup of $G$, then  we have a canonical homomorphism
\begin{equation}\label{regerwgwgwergergwergerg}
\Res^{G}_{H}(A)\rtimes_{?} H\to  A\rtimes_{?}G
\end{equation}
given on generators by
$(a,h)\mapsto (a,h)$, where $h$ in the target is considered as an element of $G$.
\hB
\end{construction}

Below we will need the following properties of the crossed products.
 \begin{lem}\label{qeriogjowergwergefefwefw}
 \mbox{}
 \begin{enumerate}
 \item \label{trhijwothgwwergwergewg} $- \rtimes_{\max}G$ preserves filtered colimits and   sums.
\item \label{trhijwothgwwergwergewg1} $- \rtimes_{r}G$ preserves isometric inclusions and  sums.
 \end{enumerate}
 \end{lem}
\begin{proof}
In order to show Assertion \ref{trhijwothgwwergwergewg}
we use the fact 
that the inclusion  $\incl:\nCalg\to \nCcat$ preserves and detects filtered colimits.
Using  \cite[Prop.\ 7.3.2]{crosscat} we see that  the composition of $ -\rtimes_{\max}G$ with this embedding
is equivalent to  the composition 
\begin{equation}\label{qwefqewdqwedqwedewdqed}
\Fun(BG,\nCalg)       \stackrel{L\circ \incl}{\to} \Fun(BG,\nCcat)
\stackrel{\colim_{BG}}{\to} \nCcat\ ,
\end{equation} 
where $L$ is the functor from \cite[(7.1)]{crosscat}.
It follows from the explicit description of this functor that $L\circ \incl$ preserves
filtered colimits. Since $\colim_{BG}$ preserves all colimits
the composition  in \eqref{qwefqewdqwedqwedewdqed} preserves filtered
colimits. We conclude that $ -{\rtimes_{\max}} G$ preserves filtered colimits.
Since this functor preserves finite sums, it also preserves
arbitrary sums.

We now show Assertion \ref{trhijwothgwwergwergewg1}. It is well-known
and easy to check from the explicit description of 
$-\rtimes_{r}G$ that this functor preserves isometric inclusions (see  \cite[Prop.\ 12.24]{cank} for the more general case of $C^{*}$-categories). 
In order to see that this functor preserves sums we argue as follows. 
If $(A_{i})_{i\in I}$ is a family in $\Fun(BG,\nCalg)$, then considering this family as a $C^{*}$-category  $\bC((A_{i})_{i\in I})$ in $\Fun(BG,\nCcat)$ with the set of objects $I$ and no non-trivial morphisms
between different objects we have a canonical isomorphism
 \begin{equation}\label{qwefqwdeqdqwedq} A(\bC((A_{i})_{i\in I}))\simeq \bigoplus_{i\in I} A_{i}\end{equation} in $\Fun(BG,\nCalg)$,  where the functor $A$ is
as in Construction \ref{woitgjowrtgggergwergw}.   By \cite[Lem.\ 12.22]{cank}
the reduced crossed product functor for $C^{*}$-categories  $-\rtimes_{r}G:\Fun(BG,\nCcat)\to \nCcat$ from \cite[Thm.\ 12.1]{cank}
extends the reduced crossed product functor for $C^{*}$-algebras.  
Since this extension  also   sends  the isometric inclusions $A_{i}\to \bC((A_{i})_{i\in I})$ to isometric inclusions by  \cite[Prop.\ 12.24]{cank}
we can conclude that \begin{equation}\label{qefqwefqweqdddwqw}\bC((A_{i})_{i\in I})\rtimes_{r}G\cong \bC((A_{i}\times_{r}G)_{{i\in I}})\, .
\end{equation}
Using that the functor $A$ preserves reduced crossed products by \cite[Thm.\ 12.23]{cank}
we conclude that
\begin{align*}
\mathclap{
(\bigoplus_{i\in I} A_{i})\rtimes_{r}G\stackrel{\eqref{qwefqwdeqdqwedq}}{\cong}
A(\bC((A_{i})_{i\in I}))\rtimes_{r}G\cong  A(\bC((A_{i})_{i\in I})\rtimes_{r}G) \stackrel{\eqref{qefqwefqweqdddwqw}}{\cong}A(\bC((A_{i} \rtimes_{r}G)_{i\in I}))\stackrel{\eqref{qwefqwdeqdqwedq}}{\cong} \bigoplus_{i\in I} (A_{i} \rtimes_{r}G)\, .
}
\end{align*}
  %
%
%
%
%
%
\end{proof}

Note that it follows from Proposition \ref{wervervfsdfvsdfvsdfvfsdv}
that $-\rtimes_{r}G$ preserves filtered colimits of diagrams whose structure maps are isometric inclusions.

 The next lemma provides  the remaining  {Assertion \ref{reguergiweogwergegwergwerv3} of} Theorem  \ref{reguergiweogwergegwergwerv}.
%
%
 \begin{lem}\label{wterpohgjwprtegeergwg}
We have the following commutative diagram
\begin{equation}\label{eqrwoigjweorgwerergerw}
\xymatrix{
&&&\KKGs \ar[d]_{(-\rtimes_{?} G)_{s}}\ar@/^1cm/[ddd]^{y^{G}}\\
\Fun(BG,\nCalg_{\sepa})\ar[urrr]^-{\kkGs} \ar[rr]_-{(-\rtimes_{?} G)_{s}}\ar[d]^{\incl}&&\ar[d]^{\incl} \nCalg_{\sepa}\ar[r]_-{\kks}&{\KKs}\ar[d]^{y} \\
\ar[drrr]_-{\kkG}\Fun(BG,\nCalg )\ar[rr]^-{-\rtimes_{?} G}&& \nCalg \ar[r]^-{\kk}&\KK\\
&&&\KKG\ar[u]^{-\rtimes_{?} G}
}
\end{equation}
 where
 \begin{enumerate}
 \item $(-\rtimes_{?} G)_{s}{\colon \KKGs \to \KKs}$ is exact  {and preserves countable coproducts}. 
 \item $-\rtimes_{?} G{\colon \KKG \to \KK}$ preserves colimits and compact objects. 
 \end{enumerate}
%
%
%
%
\end{lem}
\begin{proof} 

The assertion will  follow from  Lemma \ref{qoi3rgjowrtgwegwergwerg}.
Using that $G$ is countable it  is easy to check that the  functors $ -\rtimes_{\max} G$ and 
$ -\rtimes_{r} G$  preserve separable algebras. 
By Lemma \ref{qeriogjowergwergefefwefw}  they preserve sums. 
In order to show that they are $\Ind$-s-finitary it remains to 
check the Conditions  \ref{wtoijwrgergwrefwef}.\ref{egijegoergwergwreg1} and  \ref{wtoijwrgergwrefwef}.\ref{egijegoergwergwreg2}. 

Let $A$ be in $\Fun(BG,\nCalg)$ and $B'$ be a separable subalgebra of
$A\rtimes_{\max}G$. We let $\tilde B$ be a countable dense subset of $B'$. 
Then every element $\tilde b$ in $\tilde B$  is given by a sum
$\tilde b=\sum_{g\in G}  (\tilde b_{g},g)$ converging in $A\rtimes_{\max}G$.
We let $A'$ be the $G$-invariant subalgebra of $A$ generated by the  elements $\tilde b_{g}$ for all $g$ in $G$ in $\tilde b$ in $\tilde B$. By construction it is separable and we have 
 $B'\subseteq \overline{A'\rtimes_{\max}G}^{A\rtimes_{\max}G}$.
This shows   Condition  \ref{wtoijwrgergwrefwef}.\ref{egijegoergwergwreg1} in the case of the maximal crossed product.
The case of $-\rtimes_{\min}G$ is analogous.

For $-\rtimes_{\max}G$ the 
  Condition  \ref{wtoijwrgergwrefwef}.\ref{egijegoergwergwreg2} now follows from 
Lemma \ref{qoifjqoifewqqfwedwed} since the functor preserves filtered colimits by Lemma \ref{qeriogjowergwergefefwefw}.\ref{trhijwothgwwergwergewg}.
For $-\rtimes_{\min}G$ we again use Lemma \ref{qoifjqoifewqqfwedwed} and the fact 
that the functor preserves isometric inclusions by Lemma
\ref{qeriogjowergwergefefwefw}.\ref{trhijwothgwwergwergewg1}.

{The} restriction $(-{\rtimes_{\max}} G)_{s}$ preserves semisplit exact sequences by \cite[Prop.\ 9]{Meyer:ab}. Note that the latter is also true for the reduced crossed product by the same reference.
     
{In order to verify Assumption  \ref{qoi3rgjowrtgwegwergwerg}.\ref{wetoigwtgergrwegwreg}
we will show that the functor    
\begin{equation}\label{egwegrergwegergewrgwergreg}
\kks\circ (-\rtimes_{\max} G)_{s} \colon \Fun(BG,\nCalg_{\sepa}) \to  \KK_\sepa
\end{equation}
 is    reduced, 
 $\mathbb{K}^{G}$-stable and semiexact.
 It then follows from
 the universal property stated in Theorem \ref{wtohwergerewgrewgregwrg1} that it  inverts ${\kkGsk}$-equivalences.}
 
 The functor \eqref{egwegrergwegergewrgwergreg} and its  variant for $\rtimes_{r}$ are clearly reduced. 

 {Using that 
     $(-\rtimes_{\max} G)_{s}$  sends  semisplit exact sequences in $\Fun(BG,\nCalg_{\sepa})$   to semisplit exact sequences in $\nCalgs$ and the semiexactness of  $\kks$ we see that  that the functor \eqref{egwegrergwegergewrgwergreg} is semiexact. A similar statement holds for the reduced crossed product.} 
     
 {The rest of this argument is devoted to the verification of $\mathbb{K}^{G}$-stability of the functor in \eqref{egwegrergwegergewrgwergreg} and its reduced version. 
 This will follow  from    Lemma   \ref{qeroigqfefewfqef} below and the $\mathbb{K}$-stability of $\kks$.
}


If $B$ is in $ {\nCalg}$ and $  A$ is in  $\Fun(BG, {\nCalg})$, then  we have a canonical isomorphism \begin{equation}\label{reoijeoivvffvvsvdfvsfd}
 (  A\rtimes_{!} G)\otimes_{?} B\cong  (  A\otimes_{?} B)\rtimes_{!} G \ ,
\end{equation} 
where $(?,!)\in \{(\max,\max),({\min},r)\}$.  
  On elementary tensors it is induced by the map
$(a,h)\otimes b\mapsto (a\otimes b,h)$.
%
For a proof in the case $(?,!)=(\max,\max)$ (also stated in \cite[(3)]{Meyer:ab}) we refer to \cite[Lem.\ 2.75]{williams}. 
For the case $(?,!)=(\min,r)$ see \cite[Lem.\ 4.1]{echter}.
In both cases the argument is similar to 
  the proof of Lemma \ref{qeroigqfefewfqef}  below.
The isomorphism \eqref{reoijeoivvffvvsvdfvsfd} {could be used to show $\mathbb{K}$-stability, but it} is not sufficient to show that
{\eqref{egwegrergwegergewrgwergreg}  is $\mathbb{K}^{G}$-stable} since there we would need it  for algebras of the form $K({   V})$ in place of $B$ which have a non-trivial $G$-action.
This problem  is settled  by the following lemma which is probably well-known to  experts. 

 \begin{lem}\label{qeroigqfefewfqef}\mbox{}
 If $  V$ is a separable $G$-Hilbert space, then {for $?\in \{\max, r\}$} we have a canonical isomorphism  $$(  A\otimes K(  V))\rtimes_{{?}} G\cong (  A\rtimes_{{?}} G)\otimes  {\Res^{G}(K(  V))}\, .$$
\end{lem}
\begin{proof}
{Since} $K(V)$ is nuclear we can omit the decoration $\min$ or $\max$ at the tensor products.
As a first step we construct an isomorphism
$$\phi \colon (  A\otimes^{\alg} K(  V))\rtimes^{\alg} G\xrightarrow{\simeq} (  A\rtimes^{\alg} G)\otimes^{\alg}  {\Res^{G}(K(  V))}\, .$$
We define
$$\phi((a\otimes k,h)) \coloneqq (a,h)\otimes hk\, .$$
It is straightforward to check that $\phi$ is a homomorphism
and compatible with the involution.
 The inverse of $\phi$ is given by
 $$\psi \colon  (   A\rtimes^{\alg} G)\otimes^{\alg}  {\Res^{G}(K(  V))}\xrightarrow {\simeq} (  A\otimes^{\alg} K(  V))\rtimes^{\alg} G   \, , \quad
\psi ((a,h)\otimes k) \coloneqq (a\otimes h^{-1}k,h)\, .$$
We must show that these isomorphisms extend to the corresponding completions. 
We first discuss the case of the maximal crossed product.


Assume that we have a  homomorphism of $*$-algebras 
$ (  A\otimes^{\alg} K(  V))\rtimes^{\alg} G\to B$ with $B$ a $C^{*}$-algebra. 
We can consider 
  $
  A\otimes^{\alg} K(  V)$  naturally as a subalgebra of $(  A\otimes^{\alg} K(  V))\rtimes^{\alg} G$ by $a\otimes h\mapsto (a\otimes h,e)$.
The restriction of the representation above extends to $A\otimes K(  V)$ and then further to $ (  A\otimes K(  V))\rtimes^{\alg} G$, {and hence to $ (  A\otimes K(  V))\rtimes  G$.}

Similarly, consider a  
 homomorphism of $*$-algebras 
$ (   A\rtimes^{\alg} G)\otimes^{\alg}  {\Res^{G}(K(  V))}\to B$. 
For every finite-dimensional subalgebra $E$ in $  {\Res^{G}(K(  V))}$
we get a representation of 
$(   A\rtimes^{\alg} G)\otimes E$
by restriction along 
$(   A\rtimes^{\alg} G)\otimes^{\alg} E \to  (   A\rtimes^{\alg} G)\otimes^{\alg}  {\Res^{G}(K(  V))}$.
These representations extend to
$(  A\rtimes G)\otimes E$.  Since every operator in $ {\Res^{G}(K(  V))}$ can be
approximated by finite-dimensional ones we can further extend
the representation above to a representation of
$(  A\rtimes G)\otimes^{\alg}  {\Res^{G}(K(  V))}$, {and hence again to $(  A\rtimes G)\otimes  {\Res^{G}(K(  V))}$}.

{The two observations above show that $\phi$ and $\psi$ are isometries with respect to the maximal norms. This finishes the case of the maximal crossed product.}

For the reduced crossed product we interpret $\otimes$ as the minimal, i.e.\ spatial, tensor product.  
 The norm on $(  A\rtimes^{\alg} G)\otimes^{\alg}  {\Res^{G}(K(  V))}$ is induced from the representation on
$L^{2}(G,{  A}) \otimes  {\Res^{G}(V)}$, 
and the norm 
on $(  A\otimes^{\alg}  K(  V))\rtimes^{\alg} G $ is induced from
the representation on $L^{2}(G,{  A}\otimes   V)$. We have an isometry
$U \colon L^{2}(G,{  A})\otimes   {\Res^{G}(V)}\to L^{2}(G,{  A}\otimes   V)$ which sends
$\alpha\otimes v$ for $\alpha$ in $L^{2}(G,{  A})$ and $v$ in $V$ to the function $h\mapsto \alpha(h)\otimes {h}v$. The isometry $U$ intertwines the representation of $(  A\rtimes^{\alg} G)\otimes^{\alg} {\Res^{G}(K(  V))}$ on $
L^{2}(G,{  A})\otimes   {\Res^{G}(V)}$ with its representation via $\psi$ on $L^{2}(G,{  A}\otimes   V)$.
 Hence $\psi$ is isometric for the reduced crossed product and minimal tensor product.
 This finishes the argument in the case of the reduced crossed product. \end{proof}

{This finishes the proof of Lemma \ref{wterpohgjwprtegeergwg}.}
 \end{proof}
 
%
%

\phantomsection \label{reogijweogerewg}
Thus we have completed the proof of Theorem \ref{reguergiweogwergegwergwerv} from the introduction.

 The existence of a 
 right-adjoint of $\Res^{G}_{H} \colon \KKG\to \KKH$  for abstract reasons has been observed in Remark \ref{wetoigjowergerfgrfw}.
 In the following we identify it   
 explicitly with a functor induced by a functor on the algebra level provided $H$ is a subgroup of $G$ of finite index.
\begin{construction}
Let $H$ be any subgroup of $G$.
We define the coinduction functor
$$\Coind_{H}^{G} \colon \Fun(BH,\nCalg)\to \Fun(BG,\nCalg)\, , \quad A\mapsto \Coind_{H}^{G}(A) \coloneqq C_{b}(G,A)^{H}\, ,$$
where $ C_{b}(G,A)^{H}$ is the subspace of the algebra
$C_{b}(G,A)$ of  bounded functions $f \colon G\to A$   which are $H$-invariant in the sense that $f(gh)=h^{-1}f(g)$ for all $g$ in $G$ and $h$ in $H$.  The $*$-algebra structure on $ C_{b}(G,A)^{H}$ is pointwise induced by the $*$-algebra structure on $A$, and the $G$-action is by left-translation of functions.
Note that $ C_{b}(G,A)^{H}$ is not separable, in general.
We furthermore define a natural transformation 
\begin{equation}\label{qedoiqjoiqwedqewdq}
\Res^{G}_{H} \circ \Coind_{H}^{G}\to  \id 
\end{equation}
whose value at $A$ is the $H$-equivariant evaluation map
$C_{b}(G,A)^{H} \ni  f\mapsto f(e) \in A$.
\hB
\end{construction}
It is straightforward to check that
\eqref{qedoiqjoiqwedqewdq} is a counit of an adjunction \begin{equation}\label{regepojwerkpgergergwef}
\Res_{H}^{G}:\Fun(BG,\nCalg)\rightleftarrows\Fun(BH,\nCalg):\Coind_{H}^{G}\, ,
\end{equation} {see also \cite[Lem.\ 4.2]{Balmer_2015}.}
 
\begin{prop}\label{efoiwegerregwerg} If $H$ is a subgroup of finite index in $G$, then  the adjunction \eqref{regepojwerkpgergergwef} descends to an adjunction 
\begin{equation}\label{regepojwerkpgergergwef1}
\Res_{H}^{G}:\KKG\rightleftarrows \KKH:\Coind_{H}^{G}\, .
\end{equation}
\end{prop}
\begin{proof}
By Lemma \ref{erwoigjwegerregwefwefewfewfwefwefwefwefwefeg}  we know that  $\Res_{H}^{G}$ descends to the stable $\infty$-categories.  
In order to show the proposition 
it therefore suffices to show that
$\Coind_{H}^{G}$ also descends. But this follows from 
Lemma \ref{ewoigjeorgregerwgergew} since  we have an isomorphism
$\Ind_{H}^{G}\cong \Coind_{H}^{G}$ for subgroups $H$ of finite index.
\end{proof}

\begin{rem}
 In this remark we explain  why it is not clear whether   Proposition \ref{efoiwegerregwerg} holds true
  in the case when the index of the subgroup $H$ of $G$ is not finite. 
  We would like to show that 
 $$\kkG\circ \Coind_{H}^{G} \colon \Fun(BH,\nCalg)\to \KKG$$
 is $s$-finitary,  reduced, 
 $\mathbb{K}^{H}$-stable and split exact. 
The properties    split exact, reduced  and
$s$-finitary are straightforward. The 
problematic {property is  
$\mathbb{K}^{H}$-stability. For example, in order to verify the special case of $\mathbb{K}$-stability,}
we 
would like to use that the canonical map 
$$ {K(\ell^{2})}\otimes 
  C_{b}(G,  A)^{H}   \to 
 C_{b}(G,  {K(\ell^{2})}\otimes A)^{H}$$ is an isomorphism. But this is wrong,
 see e.g.\ \cite{dwill} for a detailed discussion. 
In the verification of  $\mathbb{K}^{G}$-stablity we would need      the analog of  Lemma \ref{rehoijerthtrhtrehtrheg} for $\Coind_{H}^{G}$ in place of $\Ind_{H}^{G}$ which fails for a similar reason.
\hB %
\end{rem}

The natural transformation \eqref{wervwihiowecwecdscac} descends to a natural transformation of functors
\begin{equation}\label{frqfwefwefwefqef}
\kkG(\iota) \colon \id\to \Res^{G}_{H}\circ \Ind^{G}_{H} \colon \KKH\to \KKH\, .
 \end{equation}   
 
%
The following proposition {settles Assertions \ref{qrougoegwergreegerewg} and \ref{eogiowpergwrewrege}} of Theorem \ref{wtgijoogwrewegewgrg}.
\begin{prop}\label{fhfufweiqwfewfqffqwfewf1}
\mbox{}
\begin{enumerate} 
\item   \label{qioerjoiqregrqfqffewfq} The transformation  
 $\kkG(\iota)$ 
  is the unit of an adjunction  
$$\Ind^{G}_{H}: \KKH\rightleftarrows \KKG:\Res^{G}_{H}\, .$$
\item\label{rthoijrohrrhrtrehth} The transformation  \eqref{fqwfoiofqwefqewfqewf}  naturally induces an equivalence of functors 
$${-}\rtimes_{?} H\to \Ind_{H}^{G}(-)\rtimes_{?}G \colon \KKH\to \KK$$
{for $?$ in $\{{r},\max\}$.}
	\end{enumerate}
\end{prop}
\begin{proof}  
For Assertion \ref{qioerjoiqregrqfqffewfq} we
  must show that    the natural transformation   of functors $(\KKH)^{\op}\times  \KKG\to \Sp^{\la}$
 \begin{eqnarray}\label{qwfefqqefqffeqwffwffqfqewfef}
r^{G}_{H} \colon \KKG( \Ind^{G}_{H}(-),-)&\xrightarrow{\Res^{G}_{H}}&\KKH(  \Res^{G}_{H}(\Ind^{{G}}_{H}(-)),
 \Res^{G}_{H}(-))\\&\xrightarrow{\eqref{frqfwefwefwefqef}}& \KKH( -,
 \Res^{G}_{H}(-))\, .\nonumber
\end{eqnarray} 
is an equivalence.  The functors in the domain and the target of  \eqref{qwfefqqefqffeqwffwffqfqewfef} send filtered colimits in $\KKH$ to limits in $\Sp$. Since $\KKH$ is generated by $\KKHs$ under filtered colimits it suffices to show that the restriction of 
\eqref{qwfefqqefqffeqwffwffqfqewfef} to
$(\KKHs)^{\op} \times \KKG$ is an equivalence. Since now both sides preserve filtered colimits in $\KKG$  it suffices to consider the restriction to 
$(\KKHs)^{\op} \times \KKGs$. Finally,   
since both sides are compatible with fibre sequences
it actually  suffices to show that the transformation
$$\pi_{0}r^{G}_{H} \colon \pi_{0}\KKGs( \Ind^{G}_{H}(-),-)\to \pi_{0}\KKHs(-,\Res^{G}_{H}(-))$$ of group-valued functors on $(\KKHs)^{\op} \times \KKGs$
{is} an  isomorphism. 
But this is the same as that in  \cite[(20)]{MR2193334}, which is shown to be an isomorphism in loc.\,cit. This finishes the proof that \eqref{qwfefqqefqffeqwffwffqfqewfef} is an equivalence and hence of Assertion \ref{qioerjoiqregrqfqffewfq}.

To see the Assertion \ref{rthoijrohrrhrtrehth} first observe that  the domain and the target of 
the transformation are s-finitary. 
Hence it suffices to check the equivalence after restriction to $\Fun(BH,\nCalgs)$. 
We now use 
Green's Imprimitivity Theorem which states that  for $A$ in $\Fun(BH,\nCalgs)$ (and also more general $H$-$C^{*}$-algebras)
the map of $C^{*}$-algebras $$A\rtimes_{?}H\to \Ind_{H}^{G}(A)\rtimes_{?}G$$ induced by   \eqref{fqwfoiofqwefqewfqewf}
is a Morita equivalence  and hence a $\kk_{0}$-equivalence. We finally use that $\kk$ inverts $\kk_{0}$-equivalences.\footnote{The text before  \cite[(9)]{MR2193334}
suggests that the authors wanted to state the version of  Green's Imprimitivity Theorem
on the level of their triangulated categories. This version {follows} from our $\infty$-categorical version by going to the homotopy category.
Our proof is essentially the same argument  as envisaged by Meyer--Nest for the  justification of \cite[(9)]{MR2193334} with the crucial addendum, that the Morita equivalence is induced by a morphism on the level of $C^{*}$-algebras.}  
\end{proof}

\begin{rem}
The Proposition \ref{fhfufweiqwfewfqffqwfewf1}.\ref{qioerjoiqregrqfqffewfq} identifies the right adjoint of $\Ind_{H}^{G} \colon \KKH\to \KKG$  whose existence was predicted by Remark \ref{wetoigjowergerfgrfw} with $\Res^{G}_{H} \colon \KKG\to \KKH$.
\hB
\end{rem}

       Recall the internal morphism object functor $\kkG_{?}(-,-)$ from  \eqref{quhfiufhqewwefqwef}. 
    \begin{kor} For $?$ in $\{\min,\max\}$ we have an equivalence of functors  
    \begin{equation}\label{fvfdsvkmsfdklvmlsdvsfdvsfv} 
\Res^{G}_{H}\circ \kkG_{?}(-,-)\simeq \kkH_{?}(\Res^{G}_{H}(-),\Res^{G}_{H}(-))\colon{(\KKG)}^{\op} \times \KKG\to \KKH\, .
\end{equation}
 \end{kor}
 \begin{proof}
 This is  a formal consequence of the adjunction in Proposition \ref{fhfufweiqwfewfqffqwfewf1}.\ref{qioerjoiqregrqfqffewfq}  {together with} Corollary \ref{evigjwelgwbfgbdb}.
 \end{proof}



We assume that $H$ is a finite group. If $A$ is in $\nCalg$, then we
consider the homomorphism
$$\epsilon_{A} \colon A\to \Res_{H}(A)\rtimes H\,, \quad a\mapsto\frac{1}{|H|}  \sum_{h\in H} (a,h)\, .$$ 
{Recall that $\Res_{H}(A)$ denotes the $C^*$-algebra $A$ equipped with the trivial $H$-action.}
The family 
$\epsilon=(\epsilon_{A})_{A\in \nCalg}$ is a natural transformation
\begin{equation}\label{wregoijwoierfwerfwerf}
\epsilon \colon \id\to \Res_{H}(-)\rtimes H
\end{equation}
of endofunctors of $\nCalg$.
It gives rise to a natural transformation 
 \begin{equation}\label{efwqewfewdqed}
\GJ{H}\colon \KKH(\Res_{H}(-),-)\xrightarrow{-\rtimes_{H}} \KK(\Res_{H}(-)\rtimes H,-\rtimes H)\xrightarrow{\kk(\epsilon)^{*}} \KK(-,-\rtimes H)
\end{equation}
of functors from
$\KK^{\op}\times \KKH$ to $\Sp$. 
Here $\GJ{}$ stands for Green--Julg.  


 The first assertion of the following theorem is a spectrum-level generalization of the classical Green--Julg theorem.
For a finite group $ H$ it explicitly identifies the right-adjoint 
of $\Res_{H} \colon \KK\to \KK^{H}$   whose existence was predicted   by Remark \ref{wetoigjowergerfgrfw}.  
\begin{theorem}\label{fhfufweiqwfewfqffqwfewf}  \mbox{}
Let  $H$ be a finite group. Then  the transformation \eqref{wregoijwoierfwerfwerf} induces the unit of an adjunction
\begin{equation}\label{qwriogjqoigjofewffwfwefqffqef}
\Res_{H}(-):\KK\rightleftarrows  \KKH:- \rtimes H\, .
\end{equation} 
\end{theorem}
 \begin{proof}   
 We must show that \eqref{efwqewfewdqed} is an equivalence.
   Since the functors in the domain and the target of   \eqref{efwqewfewdqed} send filtered colimits  in   $\KK$  to limits and $\KK$ is generated by $\KKs$ under filtered colimits it suffices to show that the restriction of \eqref{efwqewfewdqed}
to $\KKs^{\op}\times \KKH$ is an equivalence. Since this restriction preserves filtered colimits in $\KKH$ it suffices to consider the restriction to $\KKs^{\op}\times \KKHs$. 
Since both  sides are compatible with suspensions it
suffices to show that we get an isomorphism of group-valued functors after applying $\pi_{0}$.
It thus suffices to show that we get an isomorphism 
\begin{equation}\label{qdewdewdqe}
\KKth^{H}(\Res_{H}(A),B)\xrightarrow{\cong} \KKth( A,B\rtimes H)
\end{equation}
for every $A$ in $\nCalg_{\sepa}$ and
$B$ in $\Fun(BG,\nCalg_{\sepa})$. This is the classical
Green--Julg theorem (see \cite[Thm.\ 20.2.7]{blackadar} for  the case $A=\C)$.

In the following we sketch the argument.
For $\KKth^{H}$ we work with Kasparov $(\Res_{H}(A),B)$-modules $((M,\rho),\phi,F)$ such that $F$ is $H$-equivariant, i.e.\ $\rho(h)F=F\rho(h)$ for all $h$ in $H$.
Here $\rho(h) \colon M\to M$  denotes the action of $h$ in $H$ on $M$    such that  
$$\langle \rho(h) (m),m'\rangle= h( \langle m, \rho(h^{-1})(m') \rangle)$$ for all $m,m'$ in $M$.  If $H$ acts non-trivially on $B$, then $\rho(h)$ is not right-$B$-linear in general.   Following the definitions {the} map \eqref{qdewdewdqe} sends 
the $(\Res_{H}(A),B)$-module
$((M,\rho),\phi,F)$ to the   $(A,B\rtimes H)$-module
\begin{equation}\label{qewfpokqwpofqewfewf}
(M\rtimes H,  ( \phi\rtimes H)\circ \epsilon_{A}, F\rtimes H)\, .
\end{equation} 
Here
 $M\rtimes H$ is the Hilbert $(B\rtimes H)$-module whose underlying $\C$-vector space is
given by $\bigoplus_{h\in H} M$. We write the generators, which have only one non-zero entry $m$ in the summand indexed by $h$,
  in the form $(m,h)$. Then the  right-action of $  B\rtimes H$ is given by
  $(m,h)(b,h')=(\rho(h^{\prime,-1})(m)b, hh')$. Furthermore, the $(B\rtimes H)$-valued scalar product
  is given by $\langle (m,h),(m',h') \rangle \coloneqq ( \langle \rho(h^{\prime,-1}h)(m),m'  \rangle_{E}, h^{-1}h')$. The operator
  $F\rtimes H$ acts as $F(m,h)=(F(m),h)$. 
  Finally, the representation $\phi\rtimes H$ of $A\rtimes H$ on $M\rtimes H$ is given by $(\phi\rtimes H)(a,h)(m,h')=(\phi(a)(m),hh')$.

In the following we find a simpler representative of the class of \eqref{qewfpokqwpofqewfewf} which is more directly related with $((M,\rho),\phi,F)$.
We consider the projection $\pi_{1}\in B(M\rtimes H)$ given by
$$\pi_{1}(m,h) \coloneqq \frac{1}{H}\sum_{h'\in H} (\rho(h^{\prime})(m),h'h) \, .$$
This projection commutes with $F\rtimes H$ and
$( \phi\rtimes H)(\epsilon_{A}(a))$ for every $a$ in $A$. Furthermore,
$$\pi_{1}\circ ( \phi\rtimes H)(\epsilon_{A}(a))=( \phi\rtimes H)(\epsilon_{A}(a)) \quad \text{and}\quad (1-\pi_{1})\circ ( \phi\rtimes H)( \epsilon_{A}(a))=0$$
for all $a$ in $A$.
We decompose  
\begin{align*}
(M\rtimes H,  ( \phi\rtimes H)\circ \epsilon_{A}, F\rtimes H)&\\
\cong  
(\pi_{1}(M\rtimes H),  ( \phi\rtimes H)\circ \epsilon_{A}&,  \pi_{1} \circ (F\rtimes H))\oplus ((1-\pi_{1})M\rtimes H,   0, (1-\pi_{1})\circ (F\rtimes H))\, .
\end{align*}
 The second summand is degenerate. Hence the image of 
$((M,\rho),\phi,F)$ under \eqref{qdewdewdqe} is also  represented by 
$$(\pi_{1}(M\rtimes H),  ( \phi\rtimes H)\circ \epsilon_{A},  \pi_{1}\circ (F\rtimes H))\, .$$
The map
$$ \psi \colon M\to M\rtimes H\, , \quad m\mapsto \frac{1}{|H|}\sum_{h\in H} (m,h)$$ induces a $\C$-linear isomorphism
of $M$ with $\pi_{1}(M\rtimes H)$.
For $a$ in $A$ and $m$ in $M$ we have
$$\psi(\phi(a)m)= ( \phi\rtimes H)(\epsilon_{A}(a))\psi(m)\, .$$
Furthermore,
$$\psi(F(m))=\pi_{1}( (F\rtimes H)(\psi(m)))\, .$$
 We equip $M$ with the right-$H$-action by
$m h \coloneqq \rho(h^{-1})(m)$, then we get a covariant right-representation of $B$ on $M$ and hence a right-$(B\rtimes H)$-module structure. 
We let $M'$ be $M$ with this  right $(B\rtimes H)$-module structure.
The map $\psi$ then becomes $(B\rtimes H)$-linear from $ M'$ to $\pi_{1}(M\rtimes H)$. If we finally define the $(B\rtimes H)$-valued scalar product on $M'$ by 
$$\langle m,m'\rangle' \coloneqq \frac{1}{|H|}\sum_{h\in H} (\langle \rho(h)(m),m'\rangle_{E} ,h^{-1})\, ,$$
 then $\psi$ becomes an isomorphism of right Hilbert $(B\rtimes H)$-modules from $M'$ to $\pi_{1}(M\rtimes H)$.
We conclude that   
 the image of 
$((M,\rho),\phi,F)$ under \eqref{qdewdewdqe}  is represented by the  Kasparov 
$(A,B\rtimes H)$-module 
$(M',\phi,F)$.

 For the inverse consider a class in
 $\KKth(A,B\rtimes H)$ represented by 
 $(M',  \phi, F)$.
 We consider $M'$ as  
a right Hilbert $B$-module  $M$ by restriction along $B\mapsto B\rtimes H$, $b\mapsto (b,1)$.  
We further define the map
 $$t \colon B\rtimes H\to B\, , \quad   \sum_{h\in H} (b_{h},h)\mapsto b_{e}$$ 
 and the  $B$-valued scalar product  on $M$ by 
 $$\langle m,m'\rangle \coloneqq t(\langle m,m'\rangle')\, .$$
 We note that the right $B$-module $M$ is essential so that 
we get a right $H$-action  $(m,h)\mapsto mh$ on $M$. We then define
$\rho(h)(m) \coloneqq m h$.
Then the  inverse sends the class of   $(M',  \phi, F)$ to the class of 
$((M,\rho),\phi,F)$.

  It is easy to see that these constructions are inverse to each other up to isomorphism.  Since they preserve degenerate modules  and
  are compatible with direct sums and  homotopies they induce inverse to each other isomorphisms between Kasparov groups.
 This finishes the verification that \eqref{qdewdewdqe} is an isomorphism.
\end{proof}

Let $B$ be in $\nCalg$. Then we have a canonical homomorphism 
 $$\lambda_{B} \colon \Res_{G}(B)\rtimes_{\max}G\to B$$
of $C^{*}$-algebras  which corresponds to the covariant representation 
$(\id_{B},\triv)$ consisting of the identity of $B$ and the trivial representation of $G$. The family 
$\lambda=(\lambda_{B})_{B\in \nCalg}$ is a natural transformation
$\lambda \colon \Res_{G}(-)\rtimes_{\max}G\to \id$ of endofunctors of $\nCalg$. 
The following theorem is the spectrum-version   of the dual Green-Julg theorem.\footnote{We thank J.\ Echterhoff for convincing us that the statement is true in this generality.} It explicitly identifies the right-adjoint 
of $-\rtimes_{\max}G \colon \KKG\to \KK$ whose existence was predicted   by Remark \ref{wetoigjowergerfgrfw}.
\begin{theorem}
The natural  transformation  $\lambda$ is the counit of an adjunction 
$$-\rtimes_{\max}G: \KKG \rightleftarrows \KK: {\Res_{G}}\, .$$
\end{theorem}
\begin{proof}
We must show that 
the  composition
\begin{equation}\label{werferfefwefref}
 \JG{G} \colon \KK^{G}(-, \Res_{G}(-))\xrightarrow{-\rtimes_{\max}G}
\KK(-\rtimes_{\max}G, \Res_{G}(-)\rtimes_{\max}G)
\xrightarrow{\kk(\lambda)_{*}} \KK(-\rtimes_{\max}G, -)
\end{equation} 
is an equivalence of functors from $(\KKG)^{\op}\times \KK$ to $\Sp$.
%
%
%
In the first step we use the fact that the functors in the domain and target of  \eqref{werferfefwefref}  send filtered colimits 
in $\KKG$ to limits. Since $\KKGs$ generates $\KKG$ under filtered colimits it suffices to show that the restriction of the transformation to $ (\KKGs)^{\op}\times \KK$ is an equivalence.   But
this restriction
preserves filtered colimits in $\KK$. Hence it suffices to consider the restriction 
to   $(\KKGs)^{\op}\times \KKs$.
  Since both sides  are compatible with  suspensions
it suffices to check 
 that we get an isomorphism  of group-valued functors  after applying $\pi_{0}$.
It thus suffices to show that 
 we get an isomorphism  of groups
 \begin{equation}\label{adsfasdfdfafda}
 \KKth^{G}(A,\Res_{G}(B))\xrightarrow{\cong}
 \KKth(A\rtimes_{\max}G,B )
\end{equation}   for every  
 $A$ in $\Fun(BG,\nCalg_{\sepa})$ and $B$ in $\nCalg_{\sepa}$.
 This is the classical dual Green--Julg theorem (see \cite[Thm.\ 20.2.7]{blackadar} for  the case $B=\C)$.

 For completeness we sketch the argument.
 On   level of Kasparov modules  the map   \eqref{adsfasdfdfafda}  sends the $(A,B)$-module 
 $((H,\rho),\phi,F)$ to the $(A\rtimes_{\max}G,B)$-module $(H,\tilde \phi,F)$, where
 $\tilde \phi \colon A\rtimes_{\max}G\to B(H)$ is the homomorphism canonically induced by the universal property of the maximal crossed product  by the covariant representation $(\phi,\rho)$, where  
  $\rho$ is    the unitary representation of $G$ on $H$.
   For the inverse we use that we can actually represent every 
   class in  $\KKth(A\rtimes_{\max}G,B )$ by an essential module \cite[Prop.\ 18.3.6]{blackadar}.    The 
    inverse sends such a module
$(H,\tilde \phi,F)$ to $((H,\rho),\phi,F)$, where 
    $\phi$ is the restriction of $\tilde \phi$ via the canonical embedding $A\to A\rtimes_{\max}G$ which exists since  $G$ is discrete,  and where $\rho$ is the unitary representation on 
$H$ induced by  
$\tilde \phi$ {using its essentialness.}
\end{proof}

\phantomsection \label{qwfefqqefqffeqwewrgpeoj2oigrtgwffwffqfqewfef}
This completes the proof of Theorem \ref{wtgijoogwrewegewgrg}.

\section{{Analytic} \texorpdfstring{$\boldsymbol{K}$}{K}-homology}
 
 Let $X$ be in $\ppGTop$ and $Y$ be a $G$-invariant closed subset {of $X$.}
 Then we have an exact sequence
 \begin{equation}\label{fqefweddqwedwd2}
0\to C_{0}(X\setminus Y)\to C_{0}(X)\to C_{0}(Y)\to 0
\end{equation}
 in $\Fun(BG,\nCalg)$. 
 Recall the Definition \ref{weotigjoewgfreggw9} of the notion of split-closedness of $Y$.

  \begin{prop} \label{qeroifjqeriofewqewdwedqwdqewdqwedwed}The closed invariant subset $Y$ of $X$ is split-closed in the following cases:
 \begin{enumerate}
  \item \label{wetgpowkgprgewfr}$G$ acts properly on an invariant neighbourhood of $Y$ in $X$ and $Y$ is second countable.
  \item  \label{wetgpowkgprgewfr1} $Y$ admits a $G$-invariant tubular neighbourhood.
 \end{enumerate}
 \end{prop} 
\begin{proof}
We start with the proof of Assertion \ref{wetgpowkgprgewfr}.
Let $U$ be an invariant open neighbouhood of $Y$ in $X$.
 It suffices to construct an equivariant cpc split
 $s \colon C_{0}(Y)\to C_{0}(U)$. Composing with the extension-by-zero map $C_{0}(U)\to C_{0}(X)$  we then get an equivariant  cpc split for {the sequence} in \eqref{fqefweddqwedwd2}. 
 
 Since $Y$ is second countable, the $C^{*}$-algebra $C_{0}(Y)$ is separable    (see e.g.\ \cite{chou}). 
Since it is also nuclear, we can apply the Choi--Effros lifting theorem  
  {\cite{choi-effros} (see also   \cite[IV.3.2.5]{blackadar_operator_algebras})} in order to get
 a cpc split $s \colon C_{0}(Y)\to C_{0}(U)$ which is not  necessarily equivariant. 
 Since here we work with commutative algebras such a split is cpc if and only if it is contractive and preserves positive functions.
 
Since $G$ acts properly on $U$ we can choose a function
$\chi$ in $C(U,[0,1])$ such that the action map
$G\times U\to U$ restricts to a proper map on $G\times \supp(\chi)$ and $\sum_{g\in G} g^{-1,*}\chi\equiv 1$.
We now define
$$\bar s \colon  C_{0}(Y)\to  C_{0}(U)\, ,  \quad \bar s(f) \coloneqq \sum_{g\in G}  g^{-1,*} [\chi s(g^{*}f)] \, .  $$
One   verifies that this {is 
an} equivariant cpc left inverse of the restriction map $C_{0}(U)\to C_{0}(Y)$. 

Under {Assumption     \ref{wetgpowkgprgewfr1}}
we have an invariant tubular neighbourhood $U$
of $Y$, an invariant  retraction map $r \colon U\to Y$, and an invariant  radial function
$\rho \colon U\to[0,\infty)$ such that  $r_{|\rho^{-1}([0,1])}$ is proper.
We choose a function $\chi \colon [0,\infty)\to [0,1]$ such that $\chi(0)=1$ and $\chi(t)=0$ for $t\ge 1$.
Then we define $s \colon C_{0}(Y)\to C_{0}(U)$ by 
$s(f) \coloneqq \rho^{*}\chi\cdot  r^{*}s$.
This is an equivariant cpc split of the restriction
$C_{0}(U)\to C_{0}(Y)$. 
As above, 
composing with the extension-by-zero map $C_{0}(U)\to C_{0}(X)$  we    get an equivariant  cpc split for
 \eqref{fqefweddqwedwd2}.
\end{proof}

 \renewcommand{\ppGTops}{G\mathbf{LCH}_{\mathrm{2nd},+}^{\mathrm{prop}}}

We consider the functor 
$$\kkGC(-) \coloneqq \kkG\circ C_{0}(-) \colon (\ppGTop)^{\op}\to \KKG\, .$$
We furthermore let 
$\ppGTops$ be the full subcategory of $\ppGTop$ consisting of second countable spaces and set 
$$\kkGCs(-) \coloneqq \kkGs\circ C_{0}(-) \colon (\ppGTops)^{\op}\to \KKGs\, .$$
 These functors can be considered as the universal versions of the {analytic} $K$-homology functors from \eqref{qwefiojfoiwfweqfqwefewfeqw} and Definition \ref{qffohfiuwehfiowefqewfqwfqwefqewf}. The following diagram commutes by definition
\begin{equation}\label{diagram_locfinKhom_sep_yG}
\xymatrix{(\ppGTops)^{\op}\ar[rr]^-{\kkGCs}\ar[d]^{\incl}&&\KKGs\ar[d]^{y^{G}}\\(\ppGTop)^{\op}\ar[rr]^-{\kkGC}&&\KKG}
\end{equation}
{where $y^G$ is the functor from \eqref{ewgfuqgwefughfqlefiewfqwe}.}

In the following, we list the basic properties of these functors.
The spaces in the statement belong to $\ppGTop$ introduced at the beginning of Section \ref{kophetrhrtgetrgtget}.
\begin{theorem}\label{wetoighjwtiogewgregregweg}\mbox{} \begin{enumerate}
\item\label{weroigowegrregerwg} The functor $\kkGC$ is homotopy invariant.
\item\label{weroigowegrregerwg1} If $Y$ is an invariant  split-closed subspace of $X$, then we have a natural fibre sequence
$$\kkGC(X\setminus Y)\to \kkGC(X)\to \kkGC(Y)
\,.$$
\item\label{weroigowegrregerwg2} We have $\kkGC([0,\infty)\times X)\simeq 0$.
\item\label{weroigowegrregerwg3} If $(X_{n})_{n\in \nat}$ is a family in $\ppGTops$, then we have a canonical equivalence
$$\bigoplus_{n\in \nat}\kkGCs(X_{n}) \xrightarrow{\simeq} \kkGCs\big(\bigsqcup_{n\in \nat} X_{n}\big)\, .$$
\item \label{weroigowegrregerwg4}The functor $\kkGC$ has a symmetric monoidal refinement
$$\kkGCtensor \colon {\ppGTop}^{{,\otimes}}\to \KKGtensor$$ for $?$ in $\{\min,\max\}$.
\end{enumerate}
\end{theorem}
\begin{proof}
We start with Assertion \ref{weroigowegrregerwg}.
The functor $C_{0}$ sends the projection $[0,1]\times X\to X$ to the  {embedding}
$$C_{0}(X)\to  C_{0}([0,1]\times X)\cong C([0,1])\otimes C_{0}(X)\, ,$$ which is an instance  of \eqref{ewgegreggergegeggegwergw}.
Since $\kkG$ 
is  homotopy invariant, the assertion follows.

We now show Assertion \ref{weroigowegrregerwg1}. If $Y$ is a split-closed subspace of $X$, then we have a semisplit exact sequence
\begin{equation}\label{rewfweoirjoiferfwefr}
0\to C_{0}(X\setminus Y)\to C_{0}(X)\to C_{0}(Y)\to 0\, ,
\end{equation}
where the first map is given by extension by zero, while the second map is the restriction of functions from $X$ to $Y$.
We   apply the semiexactness of $\kkG$ in order to get the desired fibre sequence.

{For Assertion  \ref{weroigowegrregerwg2} we consider
\begin{align*}h:C_{0}([0,\infty)\times X)&\to C([0,\infty])\otimes C_{0}([0,\infty)\times X)\cong C_{0}([0,\infty]\times [0,\infty)\times X)\\&f\mapsto \left( (u,t,x)\mapsto \left\{\begin{array}{cc} f(t+u,x) &u\not=\infty\\ 0&u=\infty\end{array}\right. \right)    \ .\end{align*}
By the homotopy invariance of $\kkG$  we have an equivalence
$\kkG(\ev_{u=0})\circ \kkG(h)\simeq  \kkG(\ev_{u=\infty})\circ \kkG(h) $. 
But $\ev_{u=\infty} \circ h=0$ and $\ev_{0}\circ h=\id_{C_{0}([0,\infty)\otimes X)}$. This implies that
$\kkG(\id_{C_{0}([0,\infty)\otimes X)})=0$ and hence the Assertion.}


For Assertion  \ref{weroigowegrregerwg3} we first observe that
$$C_{0}\big(\bigsqcup_{n\in \nat} X_{n}\big)\cong \bigoplus_{n\in \nat} C_{0}(X_{n})\, .$$
We then use that $\kkGs$ preserves countable sums by Theorem \ref{qroifjeriogerggergegegweg}.\ref{qoirwfjhqoierggrg5}.

For Assertion  \ref{weroigowegrregerwg4} we use that the functor 
$$C_{0} \colon ({\ppGTop}^{{,\otimes}})^{\op} \to \Fun(BG,\nCalg)^{\otimes_{?}}$$ is symmetric monoidal, where the structure map is induced by 
$$C_{0}(X)\otimes_{?} C_{0}(X')\to C_{0}(X\times X')\, , \quad f\otimes f'\mapsto ((x,x')\mapsto f(x)f(x'))\, ,$$ {see Remark \ref{rem_explain_thm_Khom} for the {symmetric monoidal} structure on $ \ppGTop$.}
This is true in both cases $?=\min$ and $?=\max$. We then use Proposition \ref{togjoigerggwgergwgr} stating that $\kkG$ is symmetric monoidal.
%
%
%
%
%
\end{proof}

\begin{rem}\label{troihiohhehtrhehth}
 In Assertion \ref{wetoighjwtiogewgregregweg}.\ref{weroigowegrregerwg1} we   require that  $Y$ is  split-closed. In fact, for   an arbitrary invariant closed subset $Y$ of $X$ 
  we do not know whether the sequence  \eqref{rewfweoirjoiferfwefr} is semisplit. 
Furthermore, note that in Assertion \ref{wetoighjwtiogewgregregweg}.\ref{weroigowegrregerwg3} we must restrict to countable  unions of 
second countable spaces since we only know that $\kkGs$ preserves countable sums, while $\kkG$ is not expected to  have this property.
Finally note that all the assertions
stated for $\kkGC$ have an obvious version for $\kkGCs$. \hB
\end{rem}

\begin{ex}
For every $X$ in $\ppGTops$ and $n$ in $\nat$ we have
an equivalence
\begin{equation}\label{rgwergegergregwef}
  \Sigma^{n} \kkGC(\R^{n}\times X)\simeq \kkGC(X)\, .
\end{equation}
In order to see this we argue by induction on $n$.
The assertion is evidently  true for $n=0$. We now assume that $n>0$ and consider the commutative square
$$\xymatrix{\kkGC(\R^{n}\times X)\ar[r]\ar[d]&\kkGC([0,\infty)\times \R^{n-1}\times X)\ar[d]\\\kkGC((-\infty,0]\times \R^{n-1}\times X)\ar[r]&
\kkGC(\R^{n-1}\times X)}$$
where the maps are induced by the obvious inclusions of  closed subspaces which are split-closed by  Proposition  \ref{qeroifjqeriofewqewdwedqwdqewdqwedwed}.\ref{wetgpowkgprgewfr1}. Since
$$((-\infty,0]\times \R^{n-1}\times X)\setminus( \R^{n-1}\times X) \cong( \R^{n}\times X)\setminus ([0,\infty)\times \R^{n-1}  {\times X})\cong (-\infty,0)\times \R^{n-1}\times X$$
we conclude, using Theorem \ref{wetoighjwtiogewgregregweg}.\ref{weroigowegrregerwg1}, that the vertical morphisms induce an equivalence between 
the cofibres of the horizontal maps. Therefore the square is a push-out square in $\KKG$.
We now use  Theorem \ref{wetoighjwtiogewgregregweg}.\ref{weroigowegrregerwg2} in order to see  that the lower-left and the upper-right corners are zero objects. The square therefore  yields an equivalence
$$ \kkGC(\R^{n-1}\times X)\simeq \Sigma \kkGC(\R^{n}\times X)\, .$$
By induction we now get \eqref{rgwergegergregwef}.
\hB
\end{ex}

The following is Proposition \ref{eroigowregwergregwgreg} from the introduction. Recall the Definition \ref{qrogiqoirgergwergergwergwergwrg}.\ref{ergwgergwergeg25t}
of $G$-proper {objects in $\KKG$.}
\begin{prop}\label{eroigowregwergregwgreg1neu}
If $X$  is in $\ppGTop$ and  proper  homotopy equivalent     to
a finite $G$-CW complex with finite stabilizers, then 
$\kkGC(X)$ is $G$-proper.
\end{prop}
\begin{proof} 
Since $\kkGC$ is homotopy invariant by Theorem \ref{wetoighjwtiogewgregregweg}.\ref{weroigowegrregerwg}
we can assume that $X$ is a  $G$-CW complex with finite stabilizers.
%
We  argue by a finite induction on the number of $G$-cells. 
Assume that $X$ is obtained from $X'$ by attaching a $G$-cell of the form $G/H\times D^{n}$, where $H$ is a finite subgroup of $G$. 
Then $X'$ is a split-closed subspace of $X$ by  Proposition \ref{qeroifjqeriofewqewdwedqwdqewdqwedwed}.\ref{wetgpowkgprgewfr1}  such that $X\setminus X'\cong G/H\times \R^{n}$.
 By Theorem \ref{wetoighjwtiogewgregregweg}.\ref{weroigowegrregerwg1} and \eqref{rgwergegergregwef}  we have a fibre sequence  
$$ \Sigma^{-n} \kkGC(G/H)   \to    \kkGC(   X)\to  \kkGC  (X')
\, .$$
Since $ \Sigma^{-n} \kkGC(G/H) $ is $G$-proper by definition and $\kkGC  (X')$ is $G$-proper by {the} induction hypothesis
we can conclude that $\kkGC(X)$ is $G$-proper, too.
%
\end{proof}

The following is Theorem \ref{wtoihgwgregegwerg} from the introduction.

\begin{theorem} \label{opcsjoisjdcoijsdcoiasdcasdcadc} \mbox{}
\begin{enumerate}
\item \label{qwfiojoegggegwergweg}
If $P$   is  ind-$G$-proper, then the functor
$$\KKG(P,-) \colon \Fun(BG,\nCalg)\to \Sp^{\la}$$
sends all exact sequences to fibre sequences.
\item  \label{erwgoijweogwgwergwergeg}If $P$ is $G$-proper, then the functor $$\KKG(P,-) \colon \Fun(BG,\nCalg)\to \Sp^{\la}$$
preserves filtered colimits. 
\end{enumerate}
\end{theorem}
\begin{proof}  
%
%
We first show Assertion \ref{qwfiojoegggegwergweg}.
The full subcategory of $\KKG$ of objects $P$ such that $\KKG(P,-)$ 
sends all exact sequences in $\Fun(BG,\nCalg)$  to fibre sequences is localizing.
In view of the  Definition \ref{qrogiqoirgergwergergwergwergwrg}.\ref{weroigjwegrfrw} of  ind-$G$-properness  it suffices to check  that  the functor $$\KKG( C_{0}(G/H),-) \colon \Fun(BG,\nCalg)\to \Sp^{\la}$$  sends exact sequences in $\Fun(BG,\nCalg)$ to fibre sequences for all finite subgroups $H$ of $G$.
 
 We will use that the usual $K$-theory functor    
 for $C^{*}$-algebras \eqref{erwfoijoiwefrweffwe} sends all exact sequences to fibre sequences. 
 By Corollary \ref{qeroigjwoegwergergwerg} 
 we have the first equivalence  in
 \begin{equation}\label{wervfervihieuvervdfvsfdv}
{\KKG}(C_{0}(G/H),{-})\simeq  \KK(\C,{\Res^{G}_{H}(-)}\rtimes H)\stackrel{\eqref{erwfoijoiwefrweffwe} }{=}  \Kast({\Res^{G}_{H}(-)}\rtimes H)\, .
\end{equation}  
 Since  ${\Res^{G}_{H}(-)}\rtimes H$ preserves exact sequences and $\Kast $ sends
 exact sequences in $\nCalg$ to fibre sequence we see that 
${\KKG}(C_{0}(G/H),{-})$ sends exact sequences in $\Fun(BG,\nCalg_{\sepa})$ to fibre sequences.

The proof of Assertion \ref{erwgoijweogwgwergwergeg} is completely analogous.
The full subcategory of $\KKG$ of objects $P$
such that $\KKG(P,-) \colon \Fun(BG,\nCalg)\to \Sp^{\la}$ preserves filtered colimits is thick.
Since the functors $\Kast$ and ${\Res^{G}_{H}(-)}\rtimes H$ preserve filtered colimits 
    we see that the category in question contains the algebras $C_{0}(G/H)$ for all finite subgroups $H$ of $G$. Consequently it contains
all $G$-proper objects.
 \end{proof}

\begin{proof}[Proof of Theorem   \ref{wergojowtpgwergregwreg}] \phantomsection \label{uihiueghwegeregewg}  
The proof is based on the defining relation
\begin{equation}\label{ewrgoeijoiwejgfvsdfvsfdvfdv}
K^{\lf}_{A}(X)\simeq \KKG(\kkGC(X),A)\, .
\end{equation}
Assertion \ref{qoeirjgwergwergweg} is an immediate consequence of Proposition \ref{eroigowregwergregwgreg1}.

Assertion \ref{qoeirjgwergwergweg1} follows from Theorem \ref{wetoighjwtiogewgregregweg}.\ref{weroigowegrregerwg}.

Assertion \ref{qoeirjgwergwergweg2} follows 
by applying the exact functor $\KKG(-,A)$ to the fibre sequence in  Theorem \ref{wetoighjwtiogewgregregweg}.\ref{weroigowegrregerwg1}  

For Assertion \ref{wetoigjwotigwgergwrerfcerwffsfvf} we consider {the cases separately.} If the sequence of algebras is semisplit, then
the assertion immediately follows from the fact that $\kkG$ is  semiexact by Theorem  \ref{qeroigjqergfqeewfqewfqewf1}.\ref{weroigjwergwergwrgwgre1}. If it is just exact, then we use Theorem \ref{opcsjoisjdcoijsdcoiasdcasdcadc}.\ref{qwfiojoegggegwergweg} and Proposition
\ref{eroigowregwergregwgreg1neu} instead.
For the last statement note that if 
 $X$ is  second countable, then $C_{0}(X)$ is separable, and hence $\kkGC(X)\simeq y^{G}(\kkGs(C_{0}(X)))$ is a compact object of $\KKG${, see \eqref{diagram_locfinKhom_sep_yG} and Remark \ref{weoiguheijwogergfregwerf}.}
 
The first equivalence in Assertion \ref{qoeirjgwergwergweg3} is an immediate consequence of Theorem \ref{wetoighjwtiogewgregregweg}.\ref{weroigowegrregerwg2}. The second equivalence uses Theorem \ref{wetoighjwtiogewgregregweg}.\ref{weroigowegrregerwg3}
and the equivalence $$K^{\lf}_{A}(-)\simeq  \map_{\KKGs}(\kkGCs(-), \kkGs(A))$$ for separable $A$
as functors on $\ppGTops$ which follows from the fact that $y^{G}$ in  \eqref{ewgfuqgwefughfqlefiewfqwe} is fully faithful.

In order to see Assertion \ref{jfqoifjoerfqewf} first note that for $Y$  in $\ppGTops$ and  $A$ in $\Fun(BG,\nCalgs)$ we have an equivalence
$$K^{\lf}_{A}(Y)\simeq \map_{\kkGs}(\kkGs(C_{0}(Y)),\kkGs(A))  \, .$$ 
We use this expression in terms of the mapping spectrum in order to make it obvious that a colimit in the first argument
can be pulled out as a limit.
We can express the intersection of the decreasing family of subspaces   in categorical terms as a limit
$\bigcap_{n\in \nat} X_{n}\cong \lim_{n\in \nat} X_{n}$
which is interpreted in $\ppGTop$. By Gelfand duality we get 
$$C_{0}\big(\bigcap_{n\in \nat} X_{n}\big)\cong \colim_{n\in \nat} C_{0}(X_{n})\, ,$$
where the colimit is taken in $\Fun(BG,\nCalgs)$.
We now show that the diagram $ n\mapsto C_{0}(X_{n})$ is admissible in the sense of  \cite[Def.\ 2.5]{MR2193334}.
According to   the criterion \cite[Lem.\ 2.7]{MR2193334}  it suffices to construct   a  family of equivariant cpc maps $(s_{n} \colon C_{0}(\bigcap_{n\in \nat} X_{n}) \to C_{0}(X_{n}))_{n\in \nat}$ such that we have $\lim_{n\to \infty} \iota_{n}\circ s_{n}=\id_{C_{0}(\bigcap_{n\in \nat} X_{n})}$
in the norm topology, where $(\iota_{n} \colon C_{0}(X_{n})\to C_{0}(\bigcap_{n\in \nat} X_{n}))_{n\in \nat}$ is the family of  restriction maps.  
By our assumption we can choose   an equivariant  cpc left-inverse  $s_{0} \colon C_{0}(\bigcap_{n\in \nat} X_{n})\to C_{0}(X_{0})$ of the surjection~$\iota_{0}$. Then we can define the sought equivariant  cpc maps $s_{n} $ as the composition of $s_{0}$ with the restrictions $C_{0}(X_{0})\to C_{0}(X_{n})$. Since $\iota_{n}\circ s_{n}=\id_{C_{0}(\bigcap_{n\in \nat} X_{n})}$ for every $n$ in $\nat$  the family $(s_{n})_{n\in \nat}$ has the required property.  
We finally use that
$\kkGs$ preserves filtered colimits of admissible diagrams by Theorem \ref{qroifjeriogerggergegegweg}.\ref{qoirwfjhqoierggrg6}.

The Assertion \ref{qoeirjgwergwergweg4} follows from  
Corollary \ref{qeroigjwoegwergergwerg}
 together
with \eqref{wervfervihieuvervdfvsfdv}.  

Finally, Assertion \ref{eqrgiheigowergergergerw} follows from Theorem \ref{wetoighjwtiogewgregregweg}.\ref{weroigowegrregerwg4}, Proposition \ref{togjoigerggwgergwgr}, and the general fact that for a  stable symmetric monoidal $\infty$-category $\cC$ the functor 
$\map \colon \cC^{\op}\times \cC\to \Sp$ has a natural lax symmetric monoidal refinement.
\end{proof}

\section{Extension to \texorpdfstring{$\boldsymbol{C^{*}}$}{Cstar}-categories}\label{ergoijwogergregwergwrg}

In this section we consider the extension $\kkGA$ (see  Definition \ref{jgorijgqoiqfewfqwfewf}) of the functor $\kkG$ to $C^{*}$-categories. 
Basic references for $C^{*}$-categories are \cite{ghr}, \cite{joachimcat}, \cite{mitchc}, \cite{DellAmbrogio:2010aa}, \cite{antoun_voigt}.
We will in particular use the  language introduced in \cite{startcats}, \cite{crosscat} which we recall in the following.
 
We start with the category $\nsCat$ of possibly non-unital $*$-categories which are $\C$-vector space enriched categories $\bC$ with an involution $* \colon \bC^{\op}\to \bC$ which fixes objects and acts anti-linearly on morphisms.  Morphisms in $\nsCat$ are functors which are compatible with the involution and the enrichment. 
The category $\nsCat$ contains the full subcategory $\nsAlg $ of $*$-algebras considered as categories with a single object. 
Using the uniqueness of the norm on a $C^{*}$-algebra and the automatic continuity of $*$-homomorphisms between $C^{*}$-algebras we view the category $\nCalg$ of $C^{*}$-algebras as a full subcategory of $\nsAlg$ so that we can talk about 
functors {(i.e.\ morphisms in $\nsCat$)} from $*$-categories to $C^{*}$-algebras.

Given a $*$-category $\bC$ we define a maximal norm on morphisms
by $\|f\|_{\max} \coloneqq \sup_{\rho}\|\rho(f)\|_{B}$, where the supremum runs over all functors $\rho$ {from $\bC$} to $C^{*}$-algebras $B$. This norm may be infinite.
We call $\bC$ a pre-$C^{*}$-category if all morphisms have a finite maximal norm. 
In this way we get the full subcategory $\npCat$ of $\nsCat$  of pre-$C^{*}$-categories. Its intersection with $\nsAlg$ is the category $\npAlg$ of pre-$C^{*}$-algebras. A $C^{*}$-category is a pre-$C^{*}$-category in which all morphism spaces are complete with respect to the
maximal norm. As explained in   \cite{startcats}, \cite{crosscat} this definition is equivalent {to} other definitions in the literature. Again, we get the category $\nCalg$ of $C^{*}$-algebras by intersecting $\nCcat$ with $\nsAlg$.

There is also a unital version of all {the above. The corresponding categories will be denoted in the same way but without the superscript '$\mathrm{nu}$'.} 

The categories explained above are connected by 
 adjunctions  
 \begin{equation}\label{dsoivjaosivvdsdsvasdvadsvad}
\compl:\npCat\rightleftarrows \nCcat:\incl \quad \text{and} \quad \incl:\npCat\rightleftarrows \nsCat:\Bd^{\infty}
\end{equation}
constructed in  \cite{crosscat}, 
where $\compl$ is the completion functor  and the right adjoint
$\Bd^{\infty}$ is the bounded morphisms functor {(since we only need the existence of $\Bd^{\infty}$ we refrain from explaining it more precisely).}
These adjunctions restrict {correspondingly to adjunctions}
\begin{equation}\label{refoihioewrferfwrferfwf}
\compl:\npAlg\rightleftarrows \nCalg:\incl \quad \text{and} \quad \incl:\npAlg\rightleftarrows \nsAlg:\Bd^{\infty}\,.
\end{equation}
Using that $\nsCat$ is complete and cocomplete we can conclude formally  from the existence of the 
adjunctions (which are localizations or colocalizations, respectively) that all categories introduced above are complete and cocomplete, as well.  We refer to \cite[Thm.\ 4.1]{crosscat} for details.

While we define the extension $\kkGA$ using the left adjoint functor $A^{f}$ from the adjunction \eqref{rwegkjnkjgngjerkjwgwergregegwrege}  the verification of most of the properties of $\kkGA$
uses another $C^{*}$-algebra $A(\bC)$ associated to a $C^{*}$-category $\bC$. 
The problem with $A(-)$ is that it is only functorial for functors between $C^{*}$-categories which are injective on objects 
so that $A(-)$ can not directly be used to define $\kkGA$.
 
 \begin{construction}\label{woitgjowrtgggergwergw}
Let  $\bC$ be in {$\nsCat$}.  Following  \cite{joachimcat} we can form
the  object 
\begin{equation}\label{cwecpoijcopiqwcqwecqec}
A^{{\alg}}(\bC) \coloneqq \bigoplus_{C,C'\in \bC} \Hom_{ \bC}(C,C')
\end{equation}
of $\nsAlg$
with the obvious matrix multiplication and involution.  {If $\bC$ is in $\npCat$, then}
  $A^{\alg}(\bC)$ actually belongs to $\npAlg$  so that we define the  object of  $\nCalg$
\begin{equation}\label{erwgoijwioergewrgergwreg}
A(\bC) \coloneqq \compl(A^{\alg}(\bC))
\end{equation}
   by applying the completion functor from \eqref{refoihioewrferfwrferfwf}.
  
 The construction of $A(\bC)$ from $\bC$ is {only} functorial for functors in $\nCcat$ which are injective on objects. {We therefore introduce the wide subcategory $\nCcat_{\inj}$ of $\nCcat$ of injective
  functors and obtain a functor
  \begin{equation}\label{}
A:\nCcat_{\inj}
 \to \nCalg\ .
\end{equation}}We have a canonical isometric functor  \begin{equation}\label{qfqwefqewdqdqwqd} \bC\to A(\bC)\end{equation}  which sends a morphism $f \colon C\to C'$ in $\bC$ to the corresponding one-entry matrix $[f_{C',C}]$.  
As further  shown in  \cite{joachimcat}  the functor  $\bC\to A(\bC)$ is  initial for functors $\rho \colon \bC\to B$  with $B$ a $C^{*}$-algebra
  with the property that
  \begin{equation}\label{eq_universal_property_A}
  \rho(f)  \rho(f')=\begin{cases}\rho(f\circ f')&\text{if the composition is defined}\\0&\text{else}\end{cases}
  \end{equation}
{for any morphisms $f,f'$ in $\bC$.}
\hB
\end{construction}
%
 
Let $\bC\to \bD$ be {a morphism in $\nCcat_{\inj}$.} \begin{lem}\label{wtrhiotrhterhthebg}
{If $\bC\to \bD$ is an isometric inclusion, then 
 $A(\bC)\to A(\bD)$ is an isometric inclusion.}
\end{lem}
\begin{proof}
Let $\bC'$ be a full subcategory of $\bC$ with finitely many objects, and  let $\bD'$ be the 
full subcategory of $\bD$ on the image of the  objects of $\bC'$.
As said above, for every object $C$ in $\bC$ the map  $\End_{\bC}(C)\to A(\bC)$ is an isometry, and similarly for objects of $\bD$.
This easily implies that the upper horizontal   and the  vertical maps  in 
$$\xymatrix{A^{\alg}(\bC')\ar[r]\ar[d] &A^{\alg}(\bD')\ar[d] \\ A(\bC)\ar[r]&A(\bD)}$$
are   isometries. 
 Since $A^{\alg}(\bC)$ is  the union of the subalgebras of the form $A^{\alg}(\bC')$ 
we conclude that the maximal norm on this $*$-algebra is the norm induced from the representation on $A(\bD)$.
This implies the assertion.
\end{proof}

Let $\bC \colon \bI\to \nCcat_{{\inj}}$ be a   diagram.    \begin{lem}\label{wegoihjiogewrgergegrewf}
 {If  $I$ is filtered, then}  the canonical map
$\colim_{\bI}A(\bC)\to A(\colim_{\bI}\bC)$ is an isomorphism.
\end{lem}
\begin{proof} 
As a  formal consequence of the adjunctions in \eqref{dsoivjaosivvdsdsvasdvadsvad}
 we have the following formula for  colimits  in $\nCcat$:
  \begin{equation}\label{qwefqfjoiqfwefqwefeww}
\colim_{\bI} \bC\cong \compl({\colim_{\bI}}^{\nsCat} \bC)\, .
\end{equation} Here 
  ${\colim_{\bI}}^{\nsCat} \bC$ stands for the colimit  interpreted in the category $\nsCat$.    The latter happens to belong to $\npCat$   so that we can apply  the completion functor $\compl$.

 We claim  that for any $\bD$  in $\npCat$ there is  canonical isomorphism
 \begin{equation}\label{wregoijowergwergrf}
A(\compl( \bD))\cong \compl(A^{\alg}(\bD))\, .
\end{equation} 
In order to show the claim we
 form the square
 $$\xymatrix{
 \bD \ar[r]^-{(1)}\ar[d]_-{(2)}&\compl(\bD)\ar[dd]^-{(3)}\\
 A^{\alg}(\bD)\ar@{..>}[dr] \ar[d]_-{(4)}&\\\compl(A^{\alg}(\bD))\ar@{-->}[r]&A(\compl(\bD))
 }$$
 The   maps $(1)$ and $(4)$ are the canonical   completion maps {and}  the maps 
 $(3)$ and $(2)$ 
   {are instances of \eqref{qfqwefqewdqdqwqd}.} 
 The dotted arrow is induced from the universal property of $(2)$ applied to the composition $(3)\circ (1)$. 
  Finally, the dashed map comes from the universal
 property of  $(4)$ applied to the dotted arrow. This dashed arrow induces the desired isomorphism:
 In order to construct an inverse   we consider the diagram
 $$\xymatrix{
 \bD \ar[r]^-{(1)}\ar[d]_-{(2)}&\ar@{..>}[ddl]\compl(\bD)\ar[dd]^-{(3)}\\
 A^{\alg}(\bD ) \ar[d]_-{(4)}&\\\compl(A^{\alg}(\bD))&\ar@{-->}[l] A(\compl(\bD))
 }$$
 We get the dotted arrow from the universal property of $(1)$ applied to $(4)\circ (2)$, and then the dashed arrow 
 from the universal property of $(3)$ applied to the dotted arrow. It is {straightforward} to check that the dashed arrows in the two diagrams are inverse to each other.
 This finishes the proof of the   isomorphism \eqref{wregoijowergwergrf}.
 
Since taking objects is a left adjoint \cite[Lem. 2.4]{crosscat} and therefore commutes with colimits we have  a bijection $$\Ob( {\colim_{\bI}}^{\nsCat} \bC)\cong {\colim_{\bI}}^{\Set} \Ob(\bC)\ .$$ Furthermore, if $C,C'$ are objects  $\colim^{\nsCat}_{\bI}\bC $, then we can find  $i$ in $ \bI$  and objects $\tilde C,\tilde C'$  in $\bC(i)$
such that $\iota_{i}(\tilde C)=C$ and $\iota_{i}(\tilde C')=C'$, where $\iota_{i}:\bC(i)\to {\colim_{\bI}}^{\nsCat} \bC$ is the canonical functor.
Then
$$\Hom_{ \colim^{\nsCat}_{\bI}\bC}(C,C')\cong \colim_{(\kappa:i\to i')\in \bI_{i/}} \Hom_{\bC(i')}(\bC(\kappa)(\tilde C),\bC(\kappa)(\tilde C'))\, .$$
From this description and the formula  \eqref{cwecpoijcopiqwcqwecqec} we easily conclude that
\begin{equation}\label{wqef09uqj09efqwef}
{\colim_{\bI}}^{\nsAlg} A^{\alg}(\bC)\stackrel{\cong}{\to} A^{\alg}({\colim_{\bI}}^{\nsCat}\bC)
\end{equation}
is an isomorphism.  
We get the  isomorphisms
\begin{align*}
\colim_{\bI} A(\bC)\stackrel{\eqref{erwgoijwioergewrgergwreg}}{\cong}   \colim_{\bI}   \compl(A^{\alg}(\bC)) & \stackrel{!}{\cong}  \compl({\colim_{\bI}}^{\nsAlg} A^{{\alg}}(\bC))\\
\stackrel{\eqref{wqef09uqj09efqwef}}{\cong}  \compl( A^{\alg}({\colim_{\bI}}^{\nsCat} \bC)) & \stackrel{\eqref{wregoijowergwergrf}}{\cong} A(\compl ({\colim_{\bI}}^{\nsCat} \bC))\stackrel{\eqref{qwefqfjoiqfwefqwefeww}}{\cong} A(\colim_{{\bI}} \bC)\, ,
\end{align*}
where for the marked isomorphism we use that $\compl$   and  the  
inclusion $\incl \colon \npAlg\to \nsAlg$  are left-adjoints (see \eqref{refoihioewrferfwrferfwf}) and   therefore commute with all colimits. 
The colimits without superscripts are interpreted in $C^{*}$-algebras or $C^{*}$-categories, respectively.
\end{proof}


Recall Definition \ref{dqciojsaoidcjasdoc} of a separable $C^{*}$-category.
For every  {$G$-}$C^{*}$-category $\bC$ we consider the inductive system $(\bC')_{\bC'{\subseteq_{\sepa}\bC}}$ of  {$G$-}invariant separable subcategories. Using that $G$ is countable we have an isomorphism
\begin{equation}\label{qwefqewfqwpop} \colim_{\bC'{\subseteq_{\sepa}\bC}} \bC'\cong \bC \ .\end{equation}  

Let  $H$ be a second group and $R\colon \Fun(BG,\nCcat_{{\inj}})\to \Fun(BH,\nCalg)$
be a functor. For an invariant separable subcategory $\bC'$ of $\bC$ we let
$R(\bC')^{R(\bC)}$  be the image of
$R(\bC')\to R(\bC)$. The following generalizes   Definition \ref{wtoijwrgergwrefwef}. {Assume that $R$  sends  separable  categories to  separable  algebras.}
  \begin{ddd}\label{wtoijwrgergwrefwef1}
 We say that $R$ is $\Ind$-s-finitary if it has the following properties:
 \begin{enumerate}
 \item\label{e2gijegoergwergwreg1} For every $\bC$ in $ \Fun(BG,\nCcat ) $ the inductive system $(R(\bC')^{R(\bC)})_{\bC'{\subseteq_{\sepa}\bC}}$ is cofinal in the  inductive system of all invariant  separable
 subcategories  of $R(\bC)$.
 \item\label{e2gijegoergwergwreg2} The canonical map $(R(\bC'))_{\bC'{\subseteq_{\sepa}\bC}}\to (R(\bC')^{R(\bC)})_{\bC'{\subseteq_{\sepa}\bC}}$ is an isomorphism in
 $\Ind(  \Fun(BH,\nCcat ))$.
 \end{enumerate}
 \end{ddd}  
Note that  Definition \ref{wergoijweiogerfrefe} of an $s$-finitary functor on $G$-$C^{*}$-categories makes sense for functors just defined on $\Fun(BG,\nCcat_{{\inj}})$. 
 If $R$ is $\Ind$-s-finitary functor,  then it  is $s$-finitary.  The converse of this statement is not true.

 The  Definition  \ref{wtoijwrgergwrefwef1} is again designed in order to ensure the following fact analogous to Lemma \ref{weiogwoegwer9}. 
\begin{lem}\label{weiogwoegwer91}
If $F$ is some $s$-finitary functor on  $\Fun(BH,\nCalg) $ and $R$ is $\Ind$-s-finitary, then the composition
$F\circ R$ is   an  $s$-finitary functor on  $\Fun(BG,\nCcat_{{\inj}} ) $.
\end{lem}
 
 Lemma \ref{qoifjqoifewqqfwedwed} has the following generalization for a
 functor $R\colon \Fun(BG,\nCcat_{{\inj}})\to \Fun(BH,\nCalg)$.  \begin{lem}\label{q2oifjqoifewqqfwedwed}
Assume that $R $ {sends  separable  categories to  separable  algebras,  satisfies the Condition \ref{wtoijwrgergwrefwef1}.\ref{e2gijegoergwergwreg1}} and  one of the following:
\begin{enumerate}
\item\label{w2egiowergwergwerf} $R$ preserves \countably filtered colimits.
\item\label{w2egiowergwergwerf1}  $R$ preserves isometric inclusions.
\end{enumerate}
 Then $R$ is $\Ind$-s-finitary.
\end{lem}
\begin{proof}
The proof is completely analoguous to the    proof of  Lemma \ref{qoifjqoifewqqfwedwed}
taking advantage of the fact that Example \ref{wegoijwoegwergwregwreg}
also shows that the poset invariant separable subcategories of $\bC$ in
$\Fun(BG,\nCcat)$ is \countably filtered. 
 \end{proof}

 \begin{lem}\label{qirotwrq9}
 The functors
 $A,A^{f} \colon \Fun(BG,\nCcat_{{\inj}})\to \Fun(BG,\nCalg)$ are $\Ind$-s-finitary.
 \end{lem}
\begin{proof}
It follows from their constructions that the functors
$A$ and $A^{f}$ send separable $G$-$C^{*}$-categories to separable $G$-$C^{*}$-algebras.

We now verify Condition  \ref{wtoijwrgergwrefwef1}.\ref{e2gijegoergwergwreg1}.
We consider the case of $A$. Let $\bC$ be in $\Fun(BG,\nCcat)$.   
Let $B'$ be a $G$-invariant  separable subalgebra of $A(\bC)$. Let $\tilde B$ be a countable dense subset of $B'$.
For every $\tilde b$ in $\tilde B$ we choose a sequence $(\tilde b_{n})_{n\in \nat}$ in $A^{\alg}(\bC)$ such that
$\lim_{n\to \infty} \tilde b_{n}=\tilde b$. The union of the  $G$-orbits of the matrix elements of $\tilde b_{n}$ for all $n$ and $\tilde b$ in $\tilde B$  together generate a $G$-invariant separable subcategory $\bC'$. By construction we have $B'\subseteq {A(\bC')^{A(\bC)}}$.

The argument in the case of $A^{f}$ is similar.   We use that $A^{f}(\bC)$ the constructed as the closure of the free algebra
$A^{f,\alg}(\bC)$ generated by the morphisms of $\bC$
subject to natural relations \cite[Def.\ 3.7]{joachimcat}. 
Let $B'$ be a separable subalgebra of $A^{f}(\bC)$.   Let $\tilde B$ be a countable dense subset of $B'$.
 For every $\tilde b$ in $\tilde B$ we choose a sequence $(\tilde b_{n})_{n\in \nat}$ in $A^{f,\alg}(\bC)$ such that $\lim_{n\to \infty} \tilde b_{n}=\tilde b$. 
 We write $\tilde b_{n}$ as a finite linear combination of finite products of morphisms from $\bC$. This finite set of  morphisms 
  will be called the set of  components of $\tilde b_{n}$  (it is irrelevant that this definition involves choices). 
  We let $\bC'$ be the $G$-invariant subcategory of $\bC$ generated by the union of $G$-orbits of the sets of components for all $n$ in $\nat$ and $\tilde b$ in $\tilde B$.
  Then $\bC'$ is separable by construction and we have 
   $B'\subseteq {A^{f}(\bC')^{A^{f}(\bC)}}$.

We now employ Lemma \ref{q2oifjqoifewqqfwedwed} in order to finish the argument.
In the case of $A$ we use that this functor preserves isometric inclusions by Lemma \ref{wtrhiotrhterhthebg}. 
In the case of $A^{f}$ we use that this functor preserves all colimits since it is the left-adjoint   of  the  adjunction in \eqref{rwegkjnkjgngjerkjwgwergregegwrege}.
\end{proof}

 The following corollary   proves  
Theorem \ref{qroihjqiofewfqwefqwefqwefqwefqewfq}.\ref{ergoijogwegefwerf}. 
\begin{kor}\label{wtroihjrthertherthtrhetrht}
The functor $\kkGA$ is s-finitary. 
\end{kor}
\begin{proof}
We combine the fact that $\kkG$ is $s$-finitary with 
Lemma \ref{weiogwoegwer91} and Lemma \ref{qirotwrq9} for $A^{f}$. 
 \end{proof}

If 
 $  \bC$ is in $\Fun(BG,\nCcat)$, then 
  the universal property of $A^{f}$ being the left-adjoint in \eqref{rwegkjnkjgngjerkjwgwergregegwrege} {applied to $\bC\to A(\bC)$  from \eqref{qfqwefqewdqdqwqd} provides the} canonical morphism 
 \begin{equation}\label{fwerfrewevevfdvsdfvsdfv}
\alpha_{\bC} \colon A^{f}(  \bC)\to A(  \bC) 
\end{equation} 
in $\Fun(BG,\nCalg)$ such that
$$\xymatrix{&  \bC\ar[dr]^{\eqref{qfqwefqewdqdqwqd}}\ar[dl]&\\A^{f}(  \bC)\ar[rr]^{\alpha_{\bC}}&&A(  \bC)}$$
commutes.  Here the left-down  arrow is the unit of the adjunction  \eqref{rwegkjnkjgngjerkjwgwergregegwrege}. {The family $\alpha=(\alpha_{\bC})_{\bC}$ is a natural transformation of functors on $\Fun(BG,\nCcat_{{\inj}})$.}     

The following proposition is the main technical result which makes all other arguments {further below work.} 
\begin{prop}\label{gihergioergregwefeerfwerf} For every  $  \bC$ in $\Fun(BG,\nCcat)$ the morphism
$$\kkG(\alpha_{\bC}) \colon \kkG(A^{f}(  \bC))\to \kkG(A(  \bC))$$ is an equivalence. 
\end{prop}
\begin{proof}
We first assume that $\bC$ {is separable.} 
In this case we {can repeat 
the} proof of \cite[Prop.\ 3.8] {joachimcat}.   We consider the separable $G$-Hilbert space $  H \coloneqq L^{2}(\{\C\}\cup \Ob_{\not=0}(\bC))$, where $ \Ob_{\not=0}(\bC)$ is the set of non-zero objects in $\bC$ and   the $G$-action is induced by the action of $G$ on the set of objects of $\bC$. The additional  $G$-fixed point $\{\C\}$ induces an embedding ${\eta} \colon   \C\to  K(  H)$.
The argument in the {cited} reference provides   a morphism $\beta \colon A(  \bC)\to A^{f}(  \bC)\otimes   K(   H)$. It furthermore shows that
the composition
$\beta\circ \alpha  $  (we omit the subscript $\bC$ for better readability) is homotopic to {$ \id_{A^{f}(  \bC)}\otimes \eta$}, and {that} the composition
$ ( \alpha\otimes \id_{K(   H)})\circ \beta$ is homotopic to $\id_{A(  \bC)}\otimes \eta$.
Using {$\mathbb{K}^{G}$}-stability and homotopy invariance of $\kkG$ we see that
$\kkG(\beta)\circ \kkG(\alpha)\simeq \kkG({  \id_{A^{f}(  \bC)}\otimes \eta })$ is an equivalence, and that
{$ \kkG(\alpha)\circ \kkG(\beta)  \simeq \kkG(\alpha\otimes \id_{ K(  H)})  \circ  \kkG(\beta)  \simeq \kkG(\id_{A(  \bC)}\otimes \eta)$} is an equivalence, too. Consequently,
$\kkG(\alpha)$ is an equivalence.

{We now drop the assumption that $\bC$ is separable. We then consider the poset $(\bC')_{\bC'{\subseteq_{\sepa}\bC}}$ of invariant  separable subcategories of $\bC$. We have the following commutative diagram. 
 \begin{equation}\label{rpogjoprtgwegerf}
\xymatrix{\colim_{\bC'{\subseteq_{\sepa}\bC}} \kkG(A^{f}(\bC'))\ar[rrr]^-{\colim_{\bC'{\subseteq_{\sepa}\bC}}\kkG(\alpha_{\bC'})}_{\simeq}\ar[d]^{\simeq} &&&\colim_{\bC'{\subseteq_{\sepa}\bC}} \kkG(A(\bC'))\ar[d]^{\simeq} \\%
\kkG(A^{{f}}(\bC))\ar[rrr]^-{\kkG(\alpha_{\bC})}&&&\kkG(A(\bC))}
\end{equation}
The upper horizontal arrow is an equivalence by the case discussed above.
The vertical morphisms are equivalences by a combination of Lemma \ref{qirotwrq9}, 
{Lemma} \ref{weiogwoegwer91} and the fact that $\kkG$ is $s$-finitary.
We   conclude   that $\kkG(\alpha_{\bC})$ is an equivalence.}
   \end{proof}

The following result will be used in \cite{bel-paschke}. 
Note that  if $\bC$  in $\nCcat$ {is separable}, 
then $A^{f}(\bC)$ and $A(\bC)$  are   separable $C^{*}$-algebras by {Lemma \ref{qirotwrq9}.}
Let $(\bC_{i})_{i\in I}$ be a countable family  of {separable categories  in $ \Fun(BG,\nCcat)$.} 
\begin{lem}\label{dsfasfsdfdsfadsf}
 We have a canonical equivalence 
 $$\bigoplus_{i\in I} \kkGs(A^{f}(\bC_{i})) \simeq  \kkGs(A^{f}(\coprod_{i\in I} \bC_{i}))\, .$$
\end{lem}
\begin{proof}
Since $y^{G}:\KKGs\to \KKG$ detects equivalences,  Proposition \ref{gihergioergregwefeerfwerf}
implies that for a {separable} $\bC$ in $\Fun(BG,\nCcat)$ 
we have an equivalence
$$\kkGs(\alpha_{\bC}):\kkGs(A^{f}(\bC))\stackrel{\simeq}{\to} \kkGs(A(\bC))\ .$$
We use this for the equivalences marked by (1) in the following chain:
\begin{eqnarray*}
\bigoplus_{i\in I} \kkGs(A^{f}(\bC_{i}))&\stackrel{(1)}{\simeq}&
\bigoplus_{i\in I} \kkGs(A(\bC_{i}))\\&\stackrel{(2)}{\simeq}& 
\kkGs(\bigoplus_{i\in I}A(\bC_{i}))\\&\stackrel{(3)}{\simeq}&
\kkGs( A(\coprod_{j\in I} \bC_{i}))\\&\stackrel{(1)}{\simeq}&\kkGs( A^{f}(\coprod_{j\in I} \bC_{i}))
\, .
\end{eqnarray*}
For $(2)$ we use that $\kkGs$   preserves countable sums by Theorem \ref{qroifjeriogerggergegegweg}.\ref{qoirwfjhqoierggrg5}.
For the equivalence marked by $ (3)$ we use the fact that 
 $A$  sends coproducts of $C^{*}$-categories to direct sums of algebras. 
 This is immediate from the Construction \ref{woitgjowrtgggergwergw} of $A$.
 \end{proof}

{Before we} prove  Theorem    \ref{qroihjqiofewfqwefqwefqwefqwefqewfq}.\ref{fiuqwehfiewfeewqfwefqwefqwefqewf}, {we first 
 recall} 
the notion of a weak Morita equivalence from \cite{cank}.
  %
%
Let $\bD$ be in $\nCcat$   and
{$S$ be a subset of objects   of    $\bD$.}
  
\begin{ddd}[{\cite[Def.\ 1{8}.1]{cank}}]\label{q3rioghoergerwgregwg9}
$S$ is weakly generating if
for every object $D$ in $\bD$, finite family
$(A_{i})_{i\in I}$ of morphisms $A_{i} \colon D_{i}\to D$ in $\bD$, and any $\epsilon$ in $(0,\infty)$ 
there exists a {multiplier}  isometry $u \colon C\to D$  such that $\|A_{i}-uu^{*}A_{i}\|\le \epsilon$ for all $i$ in $I$ and
$C$ is unitarily isomorphic {in $\bM\bD$}  to a finite orthogonal sum  {in $\bM\bD$}
of objects in $S$.
\end{ddd}

Let $\phi \colon \bC\to \bD$ be a morphism in $\Fun(BG,\nCcat)$.

\begin{ddd}[{\cite[Def.\ 1{8.3}]{cank}}]\label{igjweogierfwerfwerf} The functor
$\phi$ is called a weak Morita equivalence if:
\begin{enumerate}
\item $\phi$ is fully faithfull,
\item $\phi(\Ob(\bC))$ is weakly generating.
\end{enumerate}
\end{ddd}

Note that being a weak Morita equivalence only depends on the underlying morphism between $C^{*}$-categories obtained by forgetting the $G$-action. As shown in the argument for \cite[Cor.\ 18.7]{cank} a unitary equivalence is a weak Morita equivalence.

\begin{prop} \label{eriuhuweirgwgerwf}The functor $\kkGA$ 
sends weak Morita equivalences to equivalences.
\end{prop}
\begin{proof}
{Using the same method as at the end of the proof of \cite[Thm.\ {18.6}]{cank}  
we can reduce to the case of weak Morita equivalences  which are in addition
injective on objects. So from now one we assume  that $\phi$ is injective on objects.
We first assume that $\bC$ and $\bD$ are separable.}

By Proposition~\ref{gihergioergregwefeerfwerf} and the fact that the image of $y^{G}$ in \eqref{ewgfuqgwefughfqlefiewfqwe} generates $\KKG$, {we see that} it suffices to  show that
\begin{equation}\label{fvsdvfvfdvfdvsfdvsfdvsfdvsd}
\KKG(  A,A(\phi)) \colon \KKG(  A,A(  \bC))\to  \KKG(  A,A(  \bD))
\end{equation} 
is an equivalence for every $  A$ in $\Fun(BG,\nCalg_{\sepa})$. It furthermore suffices to show that this map induces an isomorphism on the level of homotopy groups.

Our {assumptions   imply} that  $A(  \bC)$ and $A(  \bD)$ belong to
$\Fun(BG,\nCalg_{\sepa})$.  
{By Proposition \ref{eroigowregwergregwgreg1}} it suffices to show that 
$$\KKth^{G}_{*}(  A,A(\phi)) \colon \KKth^{G}_{*}(  A, A(  \bC))\to \KKth^{G}_{*}(  A, A(  \bD))$$
is an isomorphism.  
In the proof of \cite[Thm.\ 1{8.6}]{cank} we have constructed an equivariant Morita
$(A(  \bC),A(  \bD))$-bimodule (the $G$-action is induced by naturality of the construction) which induces the map $\KKth^{G}_{*}(  A,A(\phi))$ on the level of Kasparov modules. Since there exists an inverse Morita
$(A(  \bD),A(  \bC))$-bimodule it is now clear that $\KKth^{G}_{*}(  A,A(\phi))$ is an isomorphism. 
We conclude that $\kkGA(\phi) \colon \kkGA(\bC)\to \kkGA(\bD)$ is an equivalence  {if  $\bC$ and $\bD$ are separable}.

In order to extend to the general case we again
 use that the functor $\kkGA $  is s-finitary
  by Corollary \ref{wtroihjrthertherthtrhetrht}. 
 {Let $\bC$ and $\bD$ be in $\Fun(BG,\nCcat)$ and $\phi:\bC\to \bD$ be a weak Morita equivalence which is also injective on objects. 
 In the following we show that we can obtain $\phi$ as a colimit of weak Morita equivalences between cofinal families of  separable invariant subcategories of $\bC$ and $\bD$.}
 
   Let $\bC'$ and $\bD'$ be  invariant {separable} subcategories of $\bC$ and $\bD$, respectively.
We choose a countable dense subset $M $ of the morphisms of $\bD$. 
For every non-zero $D'$ in $\bD'$ and  finite family $(A_{i})_{i}$  with $A_{i} \colon {D_{i}'\to D'}$ in    $M$ and $n$ in $\nat$  we choose a finite family $(C_{j})_{j\in J}$ of objects in $\bC$  and a {multiplier} isometry $u \colon \bigoplus_{j\in  J} {\phi(C_{j})}\to D'$ {(the sum is interpreted in $\bM\bD$)}
such that $\|A_{i}-u u^{*}A_{i}\|\le \frac{1}{n}$ for all $i$ in $I$. This is possible since $\phi$ is a weak Morita equivalence.

We define the set objects of the  invariant subcategory   $\bC''$ of $\bC$ as the set 
 of objects of $\bC'$ together with   all $G$-orbits of the objects $C_{j}$   appearing in  these families above. 
   We then let $\bD''$ be the  smallest $G$-invariant    $C^{*}$-subcategory of $\bM\bD $ 
 containing  $\bD'$, the images of the objects of $\bC''$ under $\phi$, the sums, their structure maps and the $u'$s chosen above.
 The subcategy   $\bD''$  is invariant and separable,
 and we have  $\bD'\subseteq \bD''$.

  Using that $\phi$ is fully faithful we then define the morphisms of $\bC''$ such that $\phi$  restricts to a fully faithful functor  $\phi_{|\bC''} \colon \bC''\to \bD''$
 {which is} a weak Morita equivalence. 
 {The subcategory $\bC''$ also   invariant and separable and satisfies  
  $\bC'\subseteq \bC''$.}

%
%
%
%
%
%
%
%
%
%
%
%
%
%
{Using these observations in a cofinality argument and the separable case already shown above we conclude that
the upper horizontal map in the  following commutative square 
is an eqivalence.
$$\xymatrix{ \colim_{\bC'{\subseteq_{\sepa}\bC}}\kkGA{(  \bC')} \ar[r]^{\simeq}\ar[d]^{\simeq}&\colim_{\bD'{\subseteq_{\sepa}\bD}}\kkGA{(  \bD')} \ar[d]^{\simeq}\\
\kkGA (  \bC) \ar[r]^{\kkGA(\phi)}& \kkGA(\bD)}$$ 
The colimits run over the posets of separable subcategories of $\bC$ and $\bD$, respectively.
The vertical morphisms are equivalence since $\kkGA$ is $s$-finitary by Corollary \ref{wtroihjrthertherthtrhetrht}. 
It implies that the lower horizontal map is an equivalence.} 
 \end{proof}

{The following corollary verifies Theorem    \ref{qroihjqiofewfqwefqwefqwefqwefqewfq}.\ref{oijqoifefdwefwedqwdqewdqe}.
\begin{kor}\label{weoigjwoegerrgwregrweg}
The functor $\kkGA$ sends unitary equivalences  to equivalences.
\end{kor}
\begin{proof}
We use that a unitary equivalence   is a weak Morita equivalence
and apply Proposition \ref{eriuhuweirgwgerwf}.
\end{proof}}

\begin{construction}
We let $\Ccat$ denote the subcategory of $\nCcat$ of unital $C^{*}$-categories
and unital functors. We form   the    {Dwyer--Kan localization}
\begin{equation}\label{adfasdfsxas}
\ell\colon \Ccat\to \Ccat_{\infty}
\end{equation}  {of} $\Ccat$ at the unitary equivalences.
As shown in \cite{DellAmbrogio:2010aa}, \cite{startcats} the $\infty$-category $\Ccat_{\infty}$ is complete and cocomplete and modelled by  a combinatorial model category. 
By \cite[Sec.\ 4.2.4]{htt} or  \cite[Prop.\ 7.9.2]{Cisinski:2017} the functor 
\begin{equation}\label{toihgjiojoervgfdvsfdvfsdvsfv}
{\ell_{BG}:=\ell \circ-} \colon \Fun(BG,\Ccat)\to \Fun(BG,\Ccat_{\infty})
\end{equation}
  exhibits $ \Fun(BG,\Ccat_{\infty})$ as the Dwyer--Kan localization 
 of $\Fun(BG,\Ccat)$ again at the unitary equivalences. 
 Since the functor
$\kkGA$ sends unitary equivalences  to equivalences, we obtain a natural factorization
\begin{equation}\label{tgwergergreewff}
\xymatrix{\Fun(BG,\Ccat)\ar[dr]_{{\ell_{BG}}}\ar[rr]^-{\kkGA}&&\KKG\\
&\Fun(BG,\Ccat_{\infty})\ar@{..>}[ur]_-{\quad\kkGAi}&}
\end{equation}
by the universal property of the Dwyer--Kan localization.
\hB
\end{construction}

The following proposition shows Theorem \ref{qroihjqiofewfqwefqwefqwefqwefqewfq}.\ref{wtigjoepgrgrggergwegwergrweg}.
\begin{prop} \label{wegoijoihggrgwergegwe} We have an equivalence
\[\kkA( -\rtimes_{ {?}} G)\simeq  (-\rtimes_{ {?}} G)\circ  \kkGA\]
of functors from
$\Fun(BG,\nCcat)\to \KK$ for $?\in \{r,\max\}$. 
\end{prop}
\begin{proof} We have the following chain of equivalences
\begin{eqnarray*}
\kkA(-\rtimes_? G)&\stackrel{\text{Def.~}\ref{jgorijgqoiqfewfqwfewf}}{=}& \kk(A^{f}(-\rtimes_? G))\\
&\stackrel{\text{Prop.~}\ref{gihergioergregwefeerfwerf}}{\simeq}& \kk(A(-\rtimes_? G))\\
&\stackrel{!}{\simeq}& \kk(A(-)\rtimes_? G)\\
&\stackrel{\eqref{eqrwoigjweorgwerergerw}}{\simeq} &(-\rtimes_? G)\circ  \kkG(A(-))\\
&\stackrel{\text{Prop.~}\ref{gihergioergregwefeerfwerf}}{\simeq}&(-\rtimes_? G)\circ  \kkG(A^{f}(-))\\
&\stackrel{\text{Def.~}\ref{jgorijgqoiqfewfqwfewf}}{=}& (-\rtimes_? G)\circ  \kkGA(-)\,,
\end{eqnarray*} 
where for $!$ we use the fact shown in \cite[Thm.~6.{10}]{crosscat} that $A(-)$ preserves the maximal crossed product, or  \cite[Thm.~{12.23}]{cank}  that $A(-)$ preserves the reduced crossed product, respectively.
\end{proof}

An exact sequence $0\to \bA\to \bB\to \bC\to 0$ in $\Fun(BG,\nCcat)$ is a 
sequence of functors which induce  bijections between  the sets of objects and exact sequences
$$0\to \Hom_{\bA}(C,C')\to \Hom_{\bB}(C,C')\to \Hom_{\bC}(C,C')\to 0$$
for all pairs of objects $C,C'$ in $\bC$ (which  will be considered also as objects of $\bA$ and $\bB$ in the natural way). {This is equivalent to the definition of an exact sequence of $G$-$C^*$-categories given in \cite[Def.\ 8.3]{crosscat}.} 

The following proposition shows  Theorem \ref{qroihjqiofewfqwefqwefqwefqwefqewfq}.\ref{efoijqweofiqofweewfewfqfqfefe} and \ref{qroihjqiofewfqwefqwefqwefqwefqewfq}.\ref{ergoijweiogegfrfwfrefwf}.
Recall Definition  \ref{qrogiqoirgergwergergwergwergwrg} of $G$-properness and  ind-$G$-properness.
\begin{prop}\label{giuigoerggregrgrweg}\mbox{}
\begin{enumerate}
\item \label{wrthopwergerwgwreferfefw}
If $  P$ in $\KKG$ is ind-$G$-proper, then the functor $\KKG( P, \kkGA(-))$
sends all exact sequences in $\Fun(BG,\nCcat)$ to fibre sequences.
\item \label{weogjwegerfrefwref}If $P$ in $\kkG$ is $G$-proper, then
$\KKG(P,\kkGA(-))$ preserves filtered colimits.
\end{enumerate}
\end{prop}
\begin{proof}   We start with Assertion \ref{wrthopwergerwgwreferfefw}.
 Let  $$0\to \bA\to \bB\to \bC\to 0$$ be an exact sequence in $\Fun(BG,\nCcat)$.
 Then by  \cite[Prop.\ 8.{9.2}]{crosscat}  we have an exact sequence
 $$0\to A(\bA)\to A(\bB)\to A(\bC)\to 0$$ in
 $\Fun(BG,\nCalg)$.
By Theorem \ref{wtoihgwgregegwerg}.\ref{ergowejworefrefwerfwerfrefw} and the assumption on $P$ we {then} get the fibre sequence 
$$\KKG(P,A(\bA))\to \KKG(P,A(\bB))\to \KKG(P,A(\bC)) 
\, .$$
Finally, we get the desired fibre sequence
$$\KKG(P, \kkGA(\bA))\to \KKG(P,\kkGA(\bB))\to \KKG(P,\kkGA(\bC{))} 
$$
from Proposition \ref{gihergioergregwefeerfwerf}.

In order to show Assertion \ref{weogjwegerfrefwref}
we note that $A^{f}$ preserves all colimits since it is a left adjoint. The assertion now immediately follows from Theorem \ref{wtoihgwgregegwerg}.\ref{ergowejworefrefwerfwerfrefw1}.
\end{proof}

In the following we need the notion of  weakly equivariant functors and  of  natural transformations between them from  \cite[Def.\ 7.10]{crosscat}. Furthermore we will use that  the
 orthogonal  sum of two weakly equivariant functors 
 is again weakly equivariant in a canonical way, see  \cite[{Prop.\ 11.4}.]{cank}. 
 Let $  \bC$ be $ \Fun(BG,\Ccat)$. The following definition generalizes 
\cite[Def.~{11}.3]{cank} from the non-equivariant to the equivariant case.
Recall that $\bC$ is called additive if the underlying $C^{*}$-category obtained by forgetting the $G$-action admits orthogonal sums for all finite families of objects \cite[Def.\ {5.5}]{cank}.
\begin{ddd} \label{weiogjegewrrferfwef}
$\bC$ is called flasque if it is additive  and admits a weakly equivariant endomorphism $  S \colon     \bC\to    \bC$ such that  we have a natural unitary isomorphism of weakly equivariant functors $S\cong S\oplus \id_{\bC}$. 
\end{ddd}

The following proposition {shows   Assertion} \ref{erogijqorgwefqfewfqfewfqef} of Theorem \ref{qroihjqiofewfqwefqwefqwefqwefqewfq}.

\begin{prop}\label{roigwjtogwerfferfw}
 If $  \bC$  in $\Fun(BG,\Ccat)$  is flasque, then  $\KKG(  P,   \kkGA(\bC))\simeq 0$ for all   
  ind-$G$-proper
$  P$ in $\KKG$.
\end{prop}
\begin{proof}  
 Assume that $  \bC$  in $\Fun(BG,\Ccat)$  is flasque. 
The full subcategory of objects $P$ in $\KKG$ such that $\KKG(P,  \kkGA(\bC))\simeq 0$ is localizing.
In view of Definition \ref{qrogiqoirgergwergergwergwergwrg}   it therefore suffices to show  that 
$\KKG(C_{0}(G/H), \kkGA( \bC))\simeq 0$  for all finite subgroups $H$ of $G$. In this case   by Definition \ref{jgorijgqoiqfewfqwfewf},   Proposition \ref{gihergioergregwefeerfwerf} and  Corollary \ref{qeroigjwoegwergergwerg} 
we have the following  equivalences (we omit   $\Res^{G}_{H}$ to simplify the notation):
\begin{eqnarray*}
 \KKG(  C_{0}(G/H),  \kkGA( \bC))&{\simeq}&\KKG(  C_{0}(G/H), A^{f}(  \bC)) \\&{\simeq}& \KKG(  C_{0}(G/H), A(  \bC)) \\&{\simeq} &\KK(\C,A(\bC)\rtimes H)\, . 
\end{eqnarray*}
We now use that  $A(-)$ commutes with crossed products \cite[Thm.\ 6.10]{crosscat} and  
 again  Proposition \ref{gihergioergregwefeerfwerf}
(for the trivial group) and the definition of $\Kcat \coloneqq \KKth(\C,A^{f}(-))$
 in order to get an  equivalence  
 $$\KKth(\C,A(\bC)\rtimes H) \simeq \KKth(\C,A(  \bC\rtimes H)) \simeq
 \KKth(\C,A^{f}(  \bC\rtimes H)) \simeq \Kcat (\bC\rtimes  H) \, .$$
 We claim that $  \bC\rtimes H$ is again flasque. 
 Let $  S \colon   \bC\to   \bC$ be the weakly equivariant  functor  implementing the flasqueness of $  \bC$ such that we have a  natural  unitary  $S \oplus \id_{\bC} \cong S$  of weakly equivariant functors.
 By \cite[Prop.\ 7.12]{crosscat} 
 the crossed product is functorial with respect to weakly equivariant functors  and unitary transformations between them. Hence  
 we get a unitary isomorphism
 $$ {S} \rtimes H\oplus  \id_{ \bC} \rtimes H \cong     ( {S \oplus \id_{\bC}})\rtimes H\cong  {S} \rtimes H\, .$$  
Hence  $  S\rtimes H$ implements  the flasqueness of   $  \bC\rtimes H$.
 By \cite[Prop.\ {13.13} \& Thm.\ 1{4}.4]{cank}  the functor
 $\Kcat$ annihilates flasques so that 
 $\Kcat(\bC\rtimes H)\simeq 0$.   By going back through the equivalences  above we conclude that
 $\KKG({C_{0}(G/H)}, \kkGA( \bC)) \simeq 0$ as desired.
\end{proof}

We finish the proof of Theorem \ref{qroihjqiofewfqwefqwefqwefqwefqewfq}
 by showing Assertion \ref{ifejhgiowergerdsfvs}.
We consider a morphism $\phi\colon \bC\to \bD$ in $\Fun(BG,\nCcat)$. Recall the notion \cite[Def.\ 17.1]{cank} of a relative Morita equivalence. Being a relative Morita equivalence is again a property of the underlying morphism obtained by forgetting the $G$-actions. 
\begin{prop}If $P$ in $\KKG$ is ind-$G$-proper and 
 $\phi\colon \bC\to \bD$ is a   relative Morita equivalence, then
$\KKG(P,\kkGA(\phi)) $ is an equivalence.
\end{prop}
\begin{proof}
The argument is very similar to the proof of Proposition \ref{roigwjtogwerfferfw}. 
The subcategory of all $P$ in $\KKG$ such that $\KKG(P,\kkGA(\phi))$ is an equivalence is localizing. We then observe as in the proof of  Proposition \ref{roigwjtogwerfferfw} that for any finite subgroup $H$ of $G$ the morphism 
$\KKG(C_{0}(G/H),\kkGA(\phi))$ is equivalent to 
$\Kcat(\phi\rtimes H)$. 

We then use that $-\rtimes  H$ preserves  
relative Morita equivalences  \cite[Prop.\ 17.2.1]{cank}. 
  We finally use that   
 $\Kcat$   sends     relative   Morita  equivalences to equivalences by  \cite[Prop.\ 17.3 \& Thm.\ 14.4]{cank}.  \end{proof}

{Recall the well-known result that the $K$-theory of stable multiplier algebras is trivial \cite[Sec.\ 12]{blackadar}. As an application of the techniques developed so far we will provide a generalization of this fact to the equivariant situation, see Corollary \ref{kor_KKG_stable_multiplier_vanishes} below.}

Let $H \coloneqq \ell^{2}\otimes L^{2}(G)$ be the standard $G$-Hilbert space.
Let $A$ be in $\Fun(BG,\nCalg)$. Then we consider the stable multiplier algebra $\cM(A\otimes K(H))$ of $A$ as a $G$-$C^{*}$-category 
{$ \bM(A\otimes K(H))$}
with a single object  {$[A\otimes K(H)]$.}
 
\begin{lem}\label{lem_stable_multiplier_flasque}
The category ${\bM}(A\otimes K(H))$ is flasque.
\end{lem}
\begin{proof} {We consider $A\otimes K(H)$ as a $G$-$C^{*}$-category with a single object $[A\otimes K(H)]$.
We will show that $A\otimes K(H)$ admits countable AV-sums \cite[Def.\ 7.1]{cank}.
Then $\bM(A\otimes K(H))$ is flasque by \cite[Ex.\ 11.5]{cank}.}
For any countable set $I$  we can choose 
 a pairwise orthogonal family of isometries $(u_{i})_{i\in I}$, $u_{i} \colon \ell^{2}\to \ell^{2}$, 
such that $\sum_{i\in I} u_{i}u_{i}^{*}=\id_{\ell^{2}}$ {in the strict topology of $B(\ell^{2})$ which we interpret as the multiplier algebra of $K(\ell^{2})$.}
 We have 
 an isomorphism $K(H)\cong K(\ell^{2})\otimes K(L^{2}(G))$
 and define the multiplier 
   $e_{i} $ in $\cM(A\otimes K(H))$  such that
  $e_{i}(a\otimes b\otimes c) \coloneqq a\otimes {u_i b} \otimes c$ for every $a$ in $A$, $b$ in $K(\ell^2)$ and $c$ in $K(L^2(G))$.
Then $(e_{i})_{i\in I}$ is a mutually orthogonal family of isometries such that
$\sum_{i\in I} e_{i}e_{i}^{*}=1_{{[A\otimes K(H)]}}$, where the sum converges strictly.
According to \cite[Prop.\  {7.10}]{cank} the pair   $({[A\otimes K(H)]},(e_{i})_{i\in I})$ represents the orthogonal sum of the family $({[A\otimes K(H)]})_{I}$ in the $C^{*}$-category ${\bM}(A\otimes K(H))$. The latter is therefore countably additive.
\end{proof}
 {Combining Lemma \ref{lem_stable_multiplier_flasque} with  Proposition \ref{roigwjtogwerfferfw} immediately implies:}
\begin{kor}\label{kor_KKG_stable_multiplier_vanishes} 
We have 
$$\KKG(P,\cM(A\otimes K(H)))\simeq 0$$
for all ind-$G$-proper $P$ in $\KKG$.
\end{kor}
If $G$ is trivial and $P=\kk(\C)$, then  this is{, as already noted above,} the well-known result that the $K$-theory of stable multiplier algebras is trivial.  The classical proof is different and shows that the unitary group of such an algebra is contractible.

\section{Tensor products of \texorpdfstring{$\boldsymbol{C^{*}}$}{Cstar}-categories}\label{regijweoijergregrgrev}
   
 The main goal of the present section is to prove  Theorem \ref{ioeghwjoigrgwergrggrgew} from the introduction stating that
 $\kkGA$ has   symmetric monoidal refinements for the maximal and minimal  tensor products on $C^{*}$-categories.  {This result features in our companion paper on equivariant Paschke duality \cite{bel-paschke}.}
 In order to show this result we give an essentially complete account for the maximal and minimal tensor products  on $\nCcat$,  {which we believe to be of independent use. For instance, it was also used in \cite{Nadig}.} As an easy consequence of the definitions in Corollary \ref{wrtohjwtrohtrgwgwgregwreg} we obtain an op-lax symmetric
 monoidal refinement of $\kkGA$ in both cases. So the main (and most complicated part) is the verification, stated as  Proposition \ref{weogijwegoijrregrggergwegerwg},  that this structure
 is actually symmetric monoidal. The crucial 
 technical result used in its proof  is the Proposition \ref{eroigjweogieogigw} asserting the compatibility of $A$ from \eqref{erwgoijwioergewrgergwreg} with the minimal and maximal tensor products of $C^{*}$-categories.

%
%
%
Our starting point is the symmetric monoidal structure on $\nsCat$   given by the algebraic tensor product.
\begin{ddd}For $\bC$ and $\bD$ in $\nsCat$  the algebraic  tensor product is characterized by the property that the  morphism
$$\bC\times \bD\to \bC\otimes^{\alg}\bD$$  in   $\nsttCat$ (possibly non-unital categories with involution)  is universal for morphisms  from 
$\bC\times \bD$ to objects from  $\nsCat$ which are bilinear on morphism spaces.
\end{ddd}
Here is an 
  description of the algebraic tensor product   of
  $\bC$ and $\bD$   in $\nsCat$:
\begin{enumerate}
\item objects: We have $\Ob(\bC\otimes^{\alg} \bD)\cong \Ob(\bC)\times \Ob(\bD)$.
\item morphisms: For objects $(C,D)$ and $(C',D')$ in $ \bC\otimes^{\alg} \bD$ we have
\[\Hom_{\bC\otimes^{\alg} \bD}((C,D),(C',D'))\cong \Hom_{\bC}(C,C'){\otimes^{\alg}}  \Hom_{\bD}(D,D')\, .\]
\item composition and involution:   These structures are defined in the obvious manner.
\end{enumerate}
 
The maximal tensor product in $\Ccat$  or $\nCcat$ has a similar description by a universal property:
\begin{ddd}\label{weogjgewgwrgwg}
For $\bC$ and $\bD$ in   $\nCcat$  the maximal tensor product is characterized by the property that the  morphism
$$\bC\times \bD\to \bC\otimes_{\max} \bD$$   in $\nsttCat$  is universal for morphisms  from 
$\bC\times \bD$ to objects from   $\nCcat$ which are bilinear on morphism spaces.
\end{ddd}
One must check that the maximal tensor product exists. In the unital case  this  {has} been shown in \cite[Prop.\ 3.12]{DellAmbrogio:2010aa}, but the proof given there explicitly uses identity morphisms and does not directly  apply in the non-unital case. 
The first step in the verification is the following lemma.
Assume that $ \bC$, $\bD$ and $\bE$ are in   $\nCcat$. 
 \begin{lem}\label{rgiojiogwgreewreg}
 The algebraic tensor products $\bC\otimes^{\alg} \bD$ and
 $\bC\otimes^{\alg} \bD\otimes^{\alg} \bE $ are pre-$C^{*}$-categories. 
 \end{lem}
\begin{proof} 
It suffices to check that for morphisms  
  $ f$ in $\bC$ and $g $ in $\bD$ or $\phi$ in $\bC\otimes^{\alg} \bD$ and $h$ in $\bE$ the  morphisms
  $f\otimes g$  in $\bC\otimes^{\alg} \bD$ or $\phi\otimes h$ in $\bC\otimes^{\alg} \bD\otimes^{\alg} \bE $ have  finite maximal norms.   
  We will show that 
  \begin{equation}\label{qergwergwregwregrwegw}
\|f\otimes g\|_{\max}\le \|f\|_{\bC}\|g\|_{\bD} \qquad \text{and} \qquad \| \phi\otimes h\|_{\max}\le \|\phi\|_{\max} \|h\|_{\bE}\, .
\end{equation}
Let $\rho \colon \bC\otimes^{\alg} \bD\to A$ be a  functor to a $C^{*}$-algebra (considered as a  morphism in $\nsCat$). We will show that $\|\rho(f\otimes g)\|_{A}\le \|f\|_{\bC}\|g\|_{\bD}$.  This fact is well-known for homomorphisms from   algebraic tensor products of $C^{*}$-algebras \cite[Cor.\ 6.3.6]{murphy}, see also Remark \ref{wrgporekfwporwefwerfref} below. Using the $C^{*}$-equality for the norm on $C^{*}$-categories, the case of $C^{*}$-categories can be reduced to  the case of $C^{*}$-algebras as follows.
 We have
 \[\|\rho(f\otimes g)\|_{A}^{2}=\|\rho(f^{*}\otimes g^{*})\rho(f\otimes g)\|_{A}=\|\rho(f^{*}f\otimes g^{*}g) \|_{A}\le \|f^{*}f\|_{\bC}\|g^{*}g\|_{\bD}=\|f\|_{\bC}^{2} \|g\|^{2}_{\bD}\,,\]
 {where} for the inequality we use that $\rho$ induces a representation of the algebraic tensor product of $C^{*}$-algebras  $\End_{\bC}(C)\otimes_{{\C}}^{\alg}\End_{\bD}(D)$ to $A$.
 Since $\rho$ is arbitrary the first inequality \eqref{qergwergwregwregrwegw} follows.
 
In order to show the second inequality we argue similarly
 using the corresponding fact for $C^{*}$-algebras: If $A,B,C$ are in $\nCalg$,  $\psi$ is in $A\otimes^{\alg}B$ and $c$ is in $C$, then  
 for every representation $\rho \colon A\otimes^{\alg}B\otimes^{\alg}C\to D$ with $D$ in $\nCalg$
 we have  
\begin{equation}\label{sdfvpokvpfovvsfd}\|\rho(\psi\otimes c)\|_{D}\le \|\psi\|_{\max}\|c\|_{C}\,,
\end{equation}
see Remark \ref{wrgporekfwporwefwerfref}.
 \end{proof}

 \begin{rem}\label{wrgporekfwporwefwerfref}
 The estimate  \eqref{sdfvpokvpfovvsfd} is clearly well-known, but since 
  we do not know any reference we provide the argument  in some detail.  
 We consider    $A,B,C,D$ in $\nCalg$ and  a homomorphism 
  $\rho \colon A\otimes^{\alg}B\otimes^{\alg}C\to D$.  
  For $\psi$ in $A\otimes^{\alg}B$ and $c$ in $C$ we then want to show that
  $\|\rho(\psi\otimes c)\|_{D}\le \|\psi\|_{\max}\|c\|_C$.
  
  In a first step, after replacing $D$ by a subalgebra, we can assume that 
  $\rho$ has dense range.
   In the following we construct homomorphisms 
  $$\rho_{A\otimes^{\alg}B} \colon A\otimes^{\alg}B\to \cM(D)\, , \quad \rho_{B\otimes^{\alg}C} \colon B\otimes^{\alg}C\to \cM(D)\, , \quad \rho_{A\otimes^{\alg}C} \colon A\otimes^{\alg }C\to \cM(D)$$
  such that 
  \begin{equation}\label{sfvfdpovkopwevfdsvfdv}\rho(aa'\otimes bb'\otimes cc')=\rho_{A\otimes^{\alg}B}(a\otimes b)\rho_{B\otimes^{\alg}C}(b'\otimes c)\rho_{A\otimes^{\alg}C}(a'\otimes c')
\end{equation} 
for all $a,a'$ in $A$, all $b,b'$ in $B$ and all $c,c'$ in $C$. We discuss  the  construction of
 $\rho_{A\otimes^{\alg}B} $  in detail. The other two cases are analogous.
  We identify elements in the multiplier algebra $\cM(D)$ with pairs $(l,r)$ of maps
  $l,r \colon D\to D$ satisfying the multiplier identity $r(d)d'=dl(d')$. For $a$ in $A$ and $b$ in $B$ with thus have to construct
  the pair 
   $$\rho_{A\otimes^{\alg}B}(a\otimes b)=(\rho_{A\otimes^{\alg}B}^{L}(a\otimes b), \rho_{A\otimes^{\alg}B}^{R}(a\otimes b))\, .$$
   We explain the construction of $\rho_{A\otimes^{\alg}B}^{L}(a\otimes b)$.
      It is based on the following diagram:
   \begin{equation}\label{asdvposadjvaodspacdscac}
   \xymatrix{
   \ker(\rho)\ar@{..>}[drrr]^-{0}\ar[dr]&&&\\
   & A\otimes^{\alg} B\otimes^{\alg}C\ar[rr]_-{a'\otimes b'\otimes c'\mapsto \rho(aa'\otimes bb'\otimes c')}\ar[dr]_-{\rho}&&D\\
   &&\rho(A\otimes^{\alg} B\otimes^{\alg} C)\ar@{..>}[ur]\ar[r]^-{\incl}&\ar@{-->}[u]_{\rho^{L}_{A\otimes^{\alg}B}(a\otimes b)}D
   }
\end{equation}     
We start from the bold part. In a first step we show that the upper dotted arrow is zero. This implies
        the existence of the lower dotted  linear map. The argument will furthermore provide an estimate
        showing that the latter is continuous. We then get the dashed arrow by continuous extension since $\incl$ has dense range by our first reduction step.

Here are the details. For $a'$ in $A$ and $b'$ in $B$ we consider the linear map 
$\rho_{a'\otimes b'} \colon C\to D$ defined by  $c\mapsto a'\otimes b'\otimes c$. If $a'$ and $b'$ are positive elements, then 
$\rho_{a'\otimes b'}$ is a positive map between $C^{*}$-algebras and hence continuous. General elements $a'$ and $b'$ can be written as finite linear combinations of positive elements. We conclude that $\rho_{a'\otimes b'}$ is continuous in general. 

We now consider        $t \coloneqq \sum_{i=1}^{n}a_{i}\otimes b_{i}\otimes c_{i}$ in $A\otimes^{\alg}B\otimes^{\alg}C$ and let 
  $(w)$ denote  a normalized approximate unit of $C$. Then
  \begin{align*}&\lim_{w}\rho_{a\otimes b}(w)\rho(t)=
\lim_{w} \rho(a\otimes b\otimes w)\rho(t)=
 \lim_{w}\rho\big(\sum_{i=1}^{n}aa_{i}\otimes bb_{i}\otimes wc_{i}\big) \\&=
 \sum_{i=1}^{n}  \lim_{w}\rho( aa_{i}\otimes bb_{i}\otimes wc_{i}) =
 \sum_{i=1}^{n} \lim_{w} \rho_{ aa_{i}\otimes bb_{i}}(wc_{i})\stackrel{!}{=}
  \sum_{i=1}^{n} \rho_{ aa_{i}\otimes bb_{i}}(c_{i})\\&=
   \rho \big( \sum_{i=1}^{n} aa_{i}\otimes bb_{i}\otimes c_{i} \big)  
\end{align*}
using the continuity of $\rho_{ aa_{i}\otimes bb_{i}}$ at the marked equality. 
This equality first of all implies that if $\rho(t)=0$, then $ \rho( \sum_{i=1}^{n} aa_{i}\otimes bb_{i}\otimes c_{i})  =0$, i.e., that 
the upper dotted arrow in \eqref{asdvposadjvaodspacdscac} vanishes.
It further implies the estimate
$$  \rho \big( \sum_{i=1}^{n} aa_{i}\otimes bb_{i}\otimes c_{i} \big)  \le \|\rho_{a\otimes b}\| \|\rho(t)\|\, ,$$
hence the continuity of the lower dotted arrow in  \eqref{asdvposadjvaodspacdscac}.

One constructs $\rho_{A\otimes^{\alg}B}^{R}(a\otimes b)$ in a similar manner. 
The multiplier identity $$\rho_{A\otimes^{\alg}B}^{R}(a\otimes b)(d)d'=d \rho_{A\otimes^{\alg}B}^{L}(a\otimes b)(d')$$ is easy to check for
$d,d'$ in $ \rho(A\otimes^{\alg} B\otimes^{\alg} C)$ and extends by continuity to all of $D$. 
 
 In order to check  \eqref{sfvfdpovkopwevfdsvfdv} one  observes that this equality holds if one multiplies it from the right or left  by an element of  
 $ \rho(A\otimes^{\alg} B\otimes^{\alg} C)$.

In order to derive the estimate \eqref{sdfvpokvpfovvsfd} we  choose normalized approximate units 
$(u)$  of $A$ and 
$(v)$  of $B$. We write $\psi=\sum_{i=1}^{n}a_{i}\otimes b_{i}$ and calculate
\begin{align*}
 & \lim_{u}\lim_{v}\lim_{w} \rho_{A\otimes^{\alg}B}(\psi) \rho_{A\otimes^{\alg}C}(u\otimes c)\rho_{B\otimes^{\alg}C}(v\otimes w) = 
 \lim_{u}\lim_{v}\lim_{w} \rho \big( \sum_{i=1}^{n}a_{i}u\otimes b_{i}v\otimes  cw \big) \\&=
\sum_{i=1}^{n} \lim_{u}\lim_{v}\lim_{w}  \rho_{a_{i}u\otimes b_{i}v}(cw) =
\sum_{i=1}^{n} \lim_{u}\lim_{v}   \rho_{a_{i}u\otimes b_{i}v}(c)=
\sum_{i=1}^{n} \lim_{u}\lim_{v}   \rho_{a_{i}u\otimes c}( b_{i}v)\\&=
\sum_{i=1}^{n} \lim_{u}     \rho_{a_{i}u\otimes c}( b_{i})  = 
\sum_{i=1}^{n} \lim_{u}     \rho_{ b_{i}\otimes c}( a_{i}u)=
\sum_{i=1}^{n}      \rho_{ b_{i}\otimes c}( a_{i})=
\rho \big( \sum_{i=1}^{n}    a_{i}\otimes b_{i}\otimes c \big)\\&=
\rho(\psi\otimes c)\,.
\end{align*}
This gives
$$\|\rho(\psi \otimes c)\|_{D}\le   \|\rho_{A\otimes^{\alg} B}(\psi)\|_{\cM(D)} \|c\|_{C}\le  \|\psi\|_{\max}\|c\|_{C}$$
as desired.
\hB
\end{rem}

 \begin{prop}\label{woigjwergwergegw9}
 The maximal tensor product  $\otimes_{\max}$ on $\nCcat$ exists and equips this category with a symmetric monoidal structure.
 \end{prop}
\begin{proof}
In view of  {Lemma \ref{rgiojiogwgreewreg}} 
the algebraic tensor product induces a symmetric monoidal functor
$\nCcat\to \npCat$. Using the completion functor we 
 define
$$\bC\otimes_{ {\max}} \bD \coloneqq \compl(\bC\otimes^{\alg} \bD)\, .$$
It remains to define the unit, associativity and symmetry constraints. Thereby only the associativity is not completely straightforward.
In order to construct it we consider the bold part of the commutative diagram 
 $$\xymatrix{(\bA\otimes^{\alg} \bB)\otimes^{\alg} \bC\ar[r]^{\cong}\ar[d]&\bA\otimes^{\alg} (\bB\otimes^{\alg} \bC)\ar[d]\\ (\bA\otimes_{\max}  \bB)\otimes^{\alg} \bC\ar@{..>}[dr]\ar[d]& \bA\otimes^{\alg} (\bB\otimes_{\max}  \bC)\ar[d]\\   (\bA\otimes_{\max}  \bB)\otimes_{\max}  \bC  \ar@{-->}[r]  & \bA\otimes_{\max} (\bB\otimes_{\max}  \bC)}$$
whose vertical morphisms are all given by the unit of the first  adjunction in \eqref{dsoivjaosivvdsdsvasdvadsvad} and the functoriality of the algebraic tensor product. The upper horizontal functor is the associativity constraint of the algebraic tensor product.
We obtain
the dotted arrow   from the universal property of the algebraic tensor product:
To this end we must show that the bilinear functor
$$(\bA\otimes^{\alg} \bB)\times^{\alg} \bC \to \bA\otimes_{\max} (\bB\otimes_{\max}  \bC)$$
induced by the right-down composition extends by continuity  to a bilinear functor 
$$(\bA\otimes_{\max} \bB)\times^{\alg} \bC \to \bA\otimes_{\max} (\bB\otimes_{\max}  \bC)\, .$$
For a morphism $\phi$ in $ \bA\otimes^{\alg} \bB$ and $h$ in $\bC$ we have by the second inequality in \eqref{qergwergwregwregrwegw} that
$$\| \phi\otimes h\|_{ \bA\otimes_{\max} (\bB\otimes_{\max}  \bC)}\le \|\phi\|_{\max}\|h\|_{\bC}\, .$$
 This estimate implies  that the bilinear functor extends as desired, and the existence of the dotted arrow follows.
 
 We finally get the dashed arrow from the universal property of the lower left vertical arrow applied to the dotted arrow. 
 In order to show that it is an isomorphism we construct an inverse by a similar argument starting from the inverse of 
 the upper horizontal arrow.
 \end{proof}

 It is clear from the universal property of $\otimes_{\max}$, or alternatively    from its construction, that the inclusion functor
 $\incl \colon \nCalg\to \nCcat$
has a canonical symmetric monoidal refinement for the maximal tensor structures on the domain and the target.

 We now turn to the minimal tensor product on $\nCcat$.

The category $\Hilb$ of small Hilbert spaces is a commutative algebra in $\nsCat$   such that the structure morphism
$$\Hilb\otimes^{\alg}\Hilb \to \Hilb$$ is induced by the universal property of $\otimes^{\alg}$ by the functor $\Hilb\times \Hilb\to \Hilb$ given as follows:
\begin{enumerate}
\item objects: A pair $(H,H')$ of Hilbert spaces is sent to $H\otimes H'$ (tensor product in the sense of Hilbert spaces).
\item morphisms: A pair of morphism $(f,g) \colon (H_{0},H_{0}')\to (H_{1},H_{1}')$ is sent to the morphism $f\otimes g \colon H_{0}\otimes H_{0}'\to H_{1}\otimes H_{1}'$.
\end{enumerate}
 The unit of this algebra is the inclusion functor $\C\to \Hilb$.

Let $\bC, \bD$ be in $\nCcat$ and
  $c \colon \bC\to \Hilb$ and $d \colon \bD\to \Hilb$ be functors. Then we can define
a functor 
$$c\otimes d \colon \bC\otimes^{\alg} \bD\to  \Hilb\otimes^{\alg}\Hilb\to \Hilb\, .$$ 

\begin{ddd}\label{toigwjeriogregergergre}
The minimal tensor product $\bC\otimes_{\min} \bD$ is defined as the  completion of the algebraic tensor product such that  
for every $c,d$ as above we have a factorization
$$\xymatrix{\bC\otimes^{\alg} \bD\ar[rr]^{c\otimes d}\ar[dr]&&\Hilb\,.\\& \bC\otimes_{\min}\bD\ar[ur]&}$$
\end{ddd}
In other words,
the minimal norm  of a morphism $\phi$ in $\bC\otimes^{\alg} \bD$ is given by 
\begin{equation}\label{qwefoiqjefoifwefefqwfqewfefqwef}
\|\phi\|_{\min} \coloneqq {\sup}_{c,d}\, \|(c\otimes d)(\phi)\|_{\Hilb}\, .
\end{equation} 

\begin{prop}
The minimal tensor product $\otimes_{\min}$ equips $\nCcat$ with a symmetric monoidal structure.
\end{prop}
\begin{proof}
We must provide the unit, associativity, and symmetry constraints. 
As in the case of the maximal tensor product only  the associativity constraint is non-straightforward.  
In order to construct it we consider the bold part of the commutative diagram 
 $$\xymatrix{(\bA\otimes^{\alg} \bB)\otimes^{\alg} \bC\ar[r]^{\cong}\ar[d]&\bA\otimes^{\alg} (\bB\otimes^{\alg} \bC)\ar[d]\\ (\bA\otimes_{\min}  \bB)\otimes^{\alg} \bC\ar@{..>}[dr]\ar[d]& \bA\otimes^{\alg} (\bB\otimes_{\min}  \bC)\ar[d]\\   (\bA\otimes_{\min}  \bB)\otimes_{\min}  \bC  \ar@{-->}[r]  & \bA\otimes_{\min} (\bB\otimes_{\min}  \bC)}$$ 
where the vertical maps are given by the canonical maps from the algebraic tensor products  to the respective completions.

As in the case of the maximal tensor product, in order to show the existence
of the dotted arrow we must show that the bilinear functor 
$$(\bA\otimes^{\alg} \bB)\times^{\alg} \bC \to \bA\otimes_{\min} (\bB\otimes_{\min}  \bC)$$
induced by the right-down composition extends by continuity  to a bilinear functor 
$$(\bA\otimes_{\min} \bB)\times^{\alg} \bC \to \bA\otimes_{\min} (\bB\otimes_{\min}  \bC)\, .$$
Let $a \colon \bA\to \Hilb$, $b \colon \bB\to \Hilb$ and $c \colon \bC\to \Hilb$ be representations. 
Let $\phi$ be in $(\bA\otimes^{\alg} \bB)$ and $h$ be in $\bC$.
Then we have the inequalities
$$\|(a\otimes b\otimes c)(\phi\otimes h)\|_{\Hilb}\le \|(a\otimes b){(\phi)} \|_{\Hilb} \|c(h)\|_{\Hilb}\le  \|{\phi}\|_{{\min}} \|h\|_{\bC}\,.$$
Since $a,b,c$ are arbitrary we conclude that
$\|\phi\otimes h\|_{ \bA\otimes_{\min} (\bB\otimes_{\min}  \bC)}\le \|\phi\|_{\min}\|h\|_{\bC}$.
This estimate implies that the bilinear functor extends as desired and that the dotted arrow exists.

 The first part of the estimate above shows that the dotted arrow further extends by continuity to the dashed arrow. 
An inverse of the dashed arrow can be constructed in a similar manner starting from the inverse of the upper horizontal arrow. 
  \end{proof}

It is again clear from the universal property of $\otimes_{\min}$, or  alternatively from the construction of the minimal norm in \eqref{qwefoiqjefoifwefefqwfqewfefqwef}, 
that the inclusion functor
 $\incl \colon \nCalg\to \nCcat$
has a canonical symmetric monoidal refinement for the minimal tensor structures on the domain and the target.

{Let us now} collect some facts about the minimal tensor product which we will use at various places in the present section.

If $A$ is in $\nCalg$, then a representation $\alpha \colon A\to \Hilb$ of $A$ is the same datum as a homomorphism $\alpha \colon A\to B(H)$ for some Hilbert space $H$. If $\beta \colon B\to B(H')$ is a second homomorphism, then  their tensor product in the sense of representations to $\Hilb$ is simply
 the tensor product 
 $$\alpha\otimes \beta \colon A\otimes^{\alg}B\to B(H)\otimes^{{\alg}} B(H')\to B(H\otimes H')\, .$$
It is known that if $\alpha$ and $\beta$ are faithful representations, then 
\begin{equation}\label{dcdcdascadscdascadc}
\|x\|_{\min}  =\| (\alpha\otimes \beta)(x)\|_{B(H\otimes H')}
\end{equation} 
for all $x$ in $A\otimes^{\alg} B$.
Thus {for $C^*$-algebras} the supremum in   
 \eqref{qwefoiqjefoifwefefqwfqewfefqwef} is realized by any pair of faithful representations.

\begin{kor}\label{wrtohjwtrohtrgwgwgregwreg}
The functor $\kkGA$ canonically refines to an op-lax symmetric monoidal functor 
$$ {\kkGAtensor} \colon \nCcattensor \to \KKGtensor$$
for $?\in \{\min,\max\}$.
\end{kor}
\begin{proof}
{As observed previously,} the inclusion functor in \eqref{rwegkjnkjgngjerkjwgwergregegwrege} has a symmetric monoidal refinement for the structures $\otimes_{?}$. Hence its left-adjoint 
 $A^{f}$  acquires a canonical op-lax symmetric monoidal structure. Since $\kkG$  is symmetric monoidal by Proposition \ref{togjoigerggwgergwgr} we conclude that the composition $\kkGA=\kkG\circ A^{f}$ has a canonical op-lax symmetric monoidal structure.
 \end{proof}

 The following proposition finishes the verification of Theorem \ref{ioeghwjoigrgwergrggrgew} from the introduction.
\begin{prop}\label{weogijwegoijrregrggergwegerwg}
For $?\in \{\min,\max\}$ the op-lax symmetric monoidal functor 
$$ {\kkGAtensor} \colon \nCcattensor \to \KKGtensor$$
is  symmetric monoidal.
 \end{prop}

  Let $\bC$, $\bD$ be in $\nCcat$. Then 
we have  functors $\bC\to A(\bC)$ and $\bD\to A(\bD)$.
We consider the composition 
$$\bC\times \bD\to A(\bC)\times A(\bD)\to  A(\bC)\otimes_{?} A(\bD)$$
in $\nsttCat$. 
This functor is  bilinear and hence, by the universal property of the respective tensor products, factorises uniquely over the  functor
$i$ in
$$\xymatrix{\bC\times \bD\ar[d]\ar[dr]&\\
\bC\otimes_{?} \bD\ar@{..>}[r]^-{i} \ar[d]_{{\eqref{qfqwefqewdqdqwqd}}}& A(\bC)\otimes_{?} A(\bD)\,.\\
A(\bC\otimes_{?} \bD)\ar@{-->}[ur]&
}$$
We now use the universal property of the functor {$A$}
 formulated in Construction \ref{woitgjowrtgggergwergw} (the conditions are straightforward to check),  
that the functor  $i$ further factorizes over  the dashed homomorphism as indicated. {W}e will call this the canonical homomorphism  in what follows.

 \begin{prop}\label{eroigjweogieogigw}
 For all $\bC,\bD$ in $\nCcat$ and $?$ in $\{\min,\max\}$ the
 canonical homomorphism
 \begin{equation}\label{sdfvoihiovsdfvfdvvsfdv}
A(\bC \otimes_{?}  \bD)\to A(\bC)\otimes_{?} {A}(\bD)
\end{equation}
 is an isomorphism 
 \end{prop}
  
  \begin{rem}
One can use  Proposition \ref{eroigjweogieogigw} in the case $?=\max$ in order to show that  the definition of the maximal tensor product of $C^{*}$-categories given in \cite[Sec.\ 3.1]{antoun_voigt}  is equivalent to the Definition \ref{weogjgewgwrgwg} used in the present paper. 
\hB
\end{rem}

\begin{proof}[Proof of Prop.\ \ref{weogijwegoijrregrggergwegerwg} assuming Prop.\ \ref{eroigjweogieogigw}]
Let $  \bC$ and $  \bD$  be in $\Fun(BG,\nCcat)$.
The structure map of the op-lax symmetric monoidal structure on $A^{f}$ is a homomorphism 
\begin{equation}\label{wgregw245t}
A^{f}(\bC\otimes_? \bD) \to A^{f}(\bC)\otimes_? A^{f}(\bD)\, .
\end{equation}
 We must show that the morphism
$$\kkG(A^{f}(  \bC\otimes_?   \bD))\xrightarrow{\kkG(\eqref{wgregw245t})} \kkG({A^{f}(\bC)\otimes_? A^{f}(\bD)})\simeq \kkG(A^{f}(  \bC))\otimes_? 
 \kkG(A^{f}(  \bD))$$
 is an equivalence in $\KKG$, where the second equivalence is the inverse of the structure map of the symmetric monoidal structure of   $\kkG$. It is easy to see that  we have the following commutative diagram
 $$\xymatrix{ A^{f}(  \bC\otimes_?   \bD)\ar[r]^-{\eqref{wgregw245t}}\ar[d]^{\alpha_{  \bC\otimes_?   \bD} }& A^{f}(  \bC)\otimes_?  A^{f}(  \bD) \ar[d]^{\alpha_{  \bC}\otimes_? \alpha_{  \bD}}\\
 A(  \bC\otimes_?   \bD)\ar[r]_-{\cong}^-{\eqref{sdfvoihiovsdfvfdvvsfdv}}&A (  \bC)\otimes_?  A (  \bD)}$$
 where the vertical morphisms are induced by instances of 
\eqref{fwerfrewevevfdvsdfvsdfv}, and the lower horizontal map is an isomorphism  {by 
Proposition} \ref{eroigjweogieogigw}.
 We now apply $\kkG$ and get 
 $$\xymatrix{ \kkG(A^{f}(  \bC\otimes_?   \bD))\ar[rr]^-{\kkG(\eqref{wgregw245t})}\ar[d]^{\kkG(\alpha_{  \bC\otimes_?   \bD} )} && \kkG( A^{f}(  \bC)\otimes_?  A^{f}(  \bD)) \ar[d]^{\kkG(\alpha_{  \bC}\otimes_? \alpha_{  \bD})}\\
 \kkG(A(  \bC\otimes_?   \bD))\ar[rr]_-{\simeq} && \kkG(A (  \bC)\otimes_?  A (  \bD))}$$
 Using Proposition \ref{gihergioergregwefeerfwerf}  {for the left vertical arrow} and Proposition \ref{togjoigerggwgergwgr} for the equivalence 
 $$ \kkG(\alpha_{  \bC}\otimes_? \alpha_{  \bD})\simeq \kkG(\alpha_{  \bC})\otimes_? \kkG(\alpha_{  \bD})$$ in order to deal with the right vertical arrow, we conclude that the vertical arrows are equivalences. We conclude that $\kkG(\eqref{wgregw245t})$ is an equivalence.
\end{proof}

{The following lemma is the first step of the proof of Proposition \ref{eroigjweogieogigw}.}
\begin{lem}\label{eroigjerwoigerwgergrwegwerg}
 If $\bC$, $\bD$  in $\nCcat$ have at most finitely many   objects, then the canonical homomorphism
 $A(\bC\otimes_{?} \bD)\to A(\bC)\otimes_{?} A(\bD)$ is an isomorphism for $?$ in $\{\min,\max\}$.
 \end{lem}
\begin{proof}
The assumption implies that $A^{\alg}(\bC) \to A(\bC)$ and $A^{\alg}(\bD)\to A(\bD)$ are isomorphisms,
where $A^{\alg}$ is as in \eqref{cwecpoijcopiqwcqwecqec}. Since the algebraic tensor product in $\nsAlg$ is formed
on the level of underlying complex vector spaces and direct sums commute with tensor products we furthermore conclude that
the algebraic analog 
$$A(\bC\otimes^{\alg} \bD)\to A(\bC)\otimes^{\alg} A(\bD)$$  of the canonical homomorphism is an isomorphism. 
We now consider the diagram 
$$\xymatrix{A(\bC\otimes^{\alg} \bD)\ar[r]\ar[d]_{\cong}&A(\bC\otimes_{?} \bD) \\A(\bC)\otimes^{\alg} A(\bD) \ar[r]&A(\bC)\otimes_{?}  A(\bD)\ar@{..>}[u]}$$
If $?=\max$, then we obtain the dotted arrow from the universal property of the lower horizontal arrow applied to the up-right composition.

In the case of $?=\min$ we argue as follows. In order to show that the dotted arrow exists, by the universal property of $\otimes_{\min}$ on $\nCalg$ we must show that for every pair of representations $c \colon A(\bC)\to \Hilb$ and $d \colon A(\bD)\to \Hilb$ we have a factorization as indicated by the dashed arrow in the  extended diagram 
$$\xymatrix{
\bC\otimes^{\alg} \bD\ar@{..>}[d]\ar@{..>}[r]&\bC\otimes_{?} \bD\ar@{.>}[dr]^{c'\otimes d'}\ar@{.>}[d]_{!}&\\
A(\bC\otimes^{\alg} \bD)\ar[r]\ar[d]_{\cong}&A(\bC\otimes_{?} \bD)\ar@{-->}[r]&\Hilb\,.\\
A(\bC)\otimes^{\alg} A(\bD) \ar[r]&A(\bC)\otimes_{?}  A(\bD) \ar[ur]_{c\otimes d}&
}$$
To this end we consider the dotted part of the diagram, where $c' \colon \bC\to \Hilb$ and $d' \colon \bD\to \Hilb$ are the restrictions of $c$ and $d$ along $\bC\to A(\bC)$ and $\bD\to A(\bD)$.
The arrow $c'\otimes d'$ exists by the universal property of the minimal tensor product on $\nCcat$. 
We get the dashed arrow from the universal property of the arrow marked by $!$. 

In order to see in both cases   {of ?} that the homomorphism
$A(\bC)\otimes_{?}A(\bD)\to A(\bC\otimes_{?} \bD)$ just constructed
is inverse to the canonical homomorphism $A(\bC\otimes_{?} \bD)\to A(\bC)\otimes_{?} A(\bD)$ one observes that this is the case by construction 
after restriction to the algebraic tensor products.
\end{proof}

In order to show Proposition \ref{eroigjweogieogigw} in general
we must extend Lemma \ref{eroigjerwoigerwgergrwegwerg} from $C^{*}$-categories with finitely may objects to arbitrary $C^{*}$-categories. {Our argument for this} will {depend on}  the following lemma {which already has been used earlier in the proof of Proposition \ref{togjoigerggwgergwgr1}.} 
\begin{lem}\label{eriguhwiegugwergwerg}
\mbox{}
\begin{enumerate}
\item \label{weoigjweorrerggwergr} The functor $\otimes_{\max}$ on $\nCalg$ preserves filtered colimits in each argument.
\item The functor $\otimes_{\min}$ on $\nCalg$ preserves filtered colimits in each argument whose structure maps are isometric inclusions.
\end{enumerate}
\end{lem}

The Assertion  \ref{eriguhwiegugwergwerg}.\ref{weoigjweorrerggwergr}  seems to be well-known \cite[II.9.6.5]{blackadar_operator_algebras}, but we could not find a detailed proof.
We will show both assertions besides some other interesting results  about tensor products of $C^{*}$-categories  below.

\begin{proof}[Proof of {Proposition}  \ref{eroigjweogieogigw} assuming {Lemma} \ref{eriguhwiegugwergwerg}]
If $\bC$ is in $\nCcat$, then we have a canonical isomorphism 
$${\colim}_{\bC'}\, \bC'\xrightarrow{\cong} \bC\, , $$
 where the colimit runs over the full subcategories of $\bC$ with  finitely  many objects. 
The structure maps of this system are  fully faithful functors which are injective on objects.
 By Lemma \ref{wegoihjiogewrgergegrewf} we therefore get an isomorphism 
\begin{equation}\label{evwojiovfsvsfdvsdfv}
{\colim}_{\bC'}\, A(\bC')\stackrel{\cong}{\to}A( \bC)\, , 
\end{equation} 
where the structure map of the system of $A(\bC')$ are isometric inclusions.
We now consider the diagram
$$\xymatrix{
\colim_{\bC',\bD'}  A(\bC')\otimes_{?} A(\bD')\ar[rr]^-{\cong}_-{\text{Lem.~}\ref{eroigjerwoigerwgergrwegwerg}} \ar[d]^{\cong}&&\colim_{\bC',\bD'} A(\bC'\otimes_{?}\bD')\ar[d] \\
A(\bC)\otimes_{?} A(\bD) \ar@{..>}[rr]&& A(\bC\otimes_{?}\bD)
}$$
where the left vertical isomorphism  {uses 
Lemma}  \ref{eriguhwiegugwergwerg} and \eqref{evwojiovfsvsfdvsdfv} (also for $\bD$),
and the right vertical arrow exists by the universal property of the colimit. 
One  checks that the dotted arrow defined by this square is inverse to 
the canonical homomorphism  since this is true
after restricting to the algebraic tensor products for the uncompleted version $A^{\alg}$ of {the functor} $A$.
\end{proof}

In the following discussion we show a couple of results which prepare the actual proof of Lemma \ref{eriguhwiegugwergwerg}. 
At the end we use  the material in order to derive some additional results  which will be used, e.g.\ in \cite{bel-paschke}.
Our presentation will be selfcontained with one exception: the exactness of the maximal tensor product for $C^{*}$-algebras {\cite[Prop.\ 3.7.1]{brown_ozawa}}, but this does not go into the proof of Lemma \ref{eriguhwiegugwergwerg}.

 We start with an explicit model for filtered colimits in $\nCcat$.  
 We consider a  small filtered  category $\bI$ and a functor $\bC \colon \bI\to \nCcat$. 
In the following construction together with Proposition \ref{roigjowegwregrgweg} we provide an explicit model for the $C^{*}$-category $\colim_{\bI} \bC$.   

\begin{construction}\label{weoigjwoegferwfewfwer}
  The colimit   ${\colim_{\bI}}^{\nsCat}\bC$ of the image of the diagram in $\nsCat$  has the  following  explicit description.  For every $i$ in $\bI$ we let $\iota_{i} \colon \bC_{i} \to    {\colim_{\bI}}^{\nsCat} \bC $ denote the canonical map.  
\begin{enumerate}
\item objects: The set objects of  ${\colim_{\bI}}^{\nsCat}\bC$   is  given by $  \colim_{\bI} \Ob(\bC) $.
    \item \label{rgoijewoigwergergw} morphisms:   For every two objects $\bar C$ and $\bar C'$ in ${\colim_{\bI}}^{\nsCat}\bC$ we can (since $\bI$ is filtered) find $i$ in $\bI$ and objects $C$ and $C'$ in  $\bC_{i}$ such that
 $\bar C=\iota_{i}(C)$ and $\bar C'=\iota_{i}(C')$. We then have  \begin{equation}\label{wefqpoju1o2irwefqewfqef}
\Hom_{{\colim_{\bI}}^{\nsCat}\bC}(\bar C, \bar C')\coloneqq \colim_{(i\to i')\in \bI_{i/}} \Hom_{\bC_{i'}}( \bC(i\to i')(C),\bC(i\to i')(C'))\, ,
\end{equation}
where $  \bI_{i/}$ denotes the slice category of objects under $i$ in $\bI$, and the colimit is taken in $\Vect_{\C}$.
\item composition and involution: These structures are defined in the canonical manner.
\end{enumerate}

We now define a norm on ${\colim_{\bI}}^{\nsCat}\bC$ as follows. 
If $\bar f\colon \bar C\to \bar C'$ is any morphism in ${\colim_{\bI}}^{\nsCat}\bC$, then there exists $i$  and $C,C'$ in $\bC_{i}$ as in Point \ref{rgoijewoigwergergw} above and 
$(i\to i')$ in $\bI_{i/}$ and a morphism $f$ in  $\bC_{i'} $ such that $\iota_{i'}(f)=\bar f$.
 We then define $$\|\bar f\| \coloneqq \lim_{(\phi \colon i'\to i'')\in \bI_{i'\!/}}  \|\phi(f)\|_{C_{i''}}\, .$$
 Since the map 
$(\phi \colon i'\to i'')\mapsto  \|\phi(f)\|_{C_{i''}}$ is decreasing and bounded below by $0$  the limit exists.
Since $\bI$ is filtered the right-hand side does not depend on the choices of $i'$ and $f$.

We form   the completion
$$\bD \coloneqq \overline{{{\colim}_{\bI}^{\nsCat}}\bC}$$
with respect to the norm defined above. This amounts to  forming the  {completion} of the morphism spaces and extending the composition and the involution by continuity. 
Note that this process {also} involves forming the quotient by the subcategory of morphisms with zero norm, {and hence the} map from the original category to its completion is not necessarily injective.
Since the norm in $\bC_{i}$ satisfies the $C^{*}$-equality and inequality for every $i$ in $\bI$ we conclude from the construction that also the norm on $\bD$ has these properties. Consequently,  $\bD$ is an object in ${\nCcat}$.

The family of  structure maps $(\iota_{i})_{i\in \bI}$ provides the first map of the composition
\begin{equation}\label{qfewffqewfeqwfqewfeqwf}
\bC\to \underline{{{\colim}_{\bI}^{\nsCat}}\bC}\to   \underline{\bD}
\end{equation}
in $\Fun(\bI,{\nClincat})$, where $\underline{\smash{-}}$ stands for forming the constant $\bI$-diagram on $-$. 
The second morphism is induced by the inclusion of the colimit into its completion.
Since  the inclusion functor
$\nCcat\to \nClincat$ is  fully faithful,  the 
composition  \eqref{qfewffqewfeqwfqewfeqwf}  is a morphism 
$\bC\to \underline{\bD}$ in $\Fun(\bI,{\nCcat})$, and hence by adjunction corresponds  a functor
\begin{equation}\label{eqwfpokqpowefqwefqwf}
\sigma\colon \colim_{\bI} \bC\to \bD\, .
\end{equation} 
{The following proposition shows that $\bD$ is an explicit model for the $C^{*}$-category $\colim_{\bI} \bC$.}
\hB
\end{construction}

\begin{prop}\label{roigjowegwregrgweg}
The functor $\sigma$ from \eqref{eqwfpokqpowefqwefqwf} is an isomorphism.
\end{prop}
\begin{proof} We construct an inverse.
We have 
a canonical functor $$\kappa\colon {{\colim}_{\bI}^{\nsCat}}\bC\to  \colim_{\bI}\bC$$
which in view of the formula \eqref{qwefqfjoiqfwefqwefeww} is just the completion map.  Using the notation from Point \ref{weoigjwoegferwfewfwer}.\ref{rgoijewoigwergergw}  we have
$$\|\kappa(\bar f)\|_{\colim_{\bI} \bC}=\|\kappa \iota_{i''}\phi(f)\|_{\colim_{\bI} \bC} \le  \| \phi(f)\|_{C_{i''}} $$ for all $(\phi \colon i'\to i'')\in \bI_{i'\!/}$. By considering the limit over $\bI_{i'\!/}$
we conclude that
$$\|\kappa(\bar f)\|_{\colim_{\bI} \bC}\le \|\bar f\|_{\bD}\, .$$ 
This shows that $\kappa$ extends by continuity to a functor 
$ \kappa \colon \bD\to \colim_{\bI}\bC$ which is necessarily inverse to $\sigma$.
\end{proof}
 
As a first application of Proposition \ref{roigjowegwregrgweg} we show  Lemma \ref{eroigjowregwrege9} below. Its specialization to $C^{*}$-algebras has been used in the proof of Proposition \ref{togjoigerggwgergwgr1}.
 
  Let $\bI$ be a   small filtered category.
  \begin{ddd}  We say that $\bI$ is \countably
 filtered if {for} every functor  $\bJ\to \bI$ from a countable category the inclusion $\bJ\to \bI$ extends to the cone over $\bJ$.
 \end{ddd}

\begin{ex}\label{wegoijwoegwergwregwreg}
If $A$ is a $C^{*}$-algebra, then the poset of all separable subalgebras of $A$ is \countably filtered. 
If $\bJ$ is a subset of this poset, then we can extend the inclusion to the cone over $\bJ$ by
sending the cone tip to the separable subalgebra $\overline{\bigcup_{j\in \bJ} A_{j}}$ of $A$, where the closure is taken in $A$.

{By the same argument, if $G$ is a countable group, then the poset of  separable $G$-invariant subalgebras of $A$ is \countably filtered. 
More generally, if $\bC$ is a $G$-$C^{*}$-category, then the poset of
separable $G$-invariant  subcategories is \countably filtered.}
 \hB
\end{ex}

Let $\bC \colon \bI\to \nCcat$ be a diagram indexed by a small filtered category $\bI$, $i$ be in $\bI$, and $\bD$ be a subcategory of $\bC_{i}$.

\begin{lem}\label{eroigjowregwrege9}
Assume:
\begin{enumerate}
\item \label{trohiijerotetrgtreg} $\bD$ {is separable.} 
\item \label{sdffsdfswer} The composition $\bD\to \bC_{i}\xrightarrow{\iota_{i}} \colim_{\bI} \bC$ is zero.
\item \label{werpogwergregweg} $\bI$ is \countably filtered.
\end{enumerate}
Then there exists  a morphism $\phi \colon i\to i'$ in $\bI$ such that the composition
$\bD\to \bC_{i}\xrightarrow{\phi(i\to i')} \bC_{i'}$ is zero.
 \end{lem}
\begin{proof}
Using the Assumption \ref{trohiijerotetrgtreg} on $\bD$ we can choose a countable set of morphisms $M$ in $\bD$ such that
$M\cap \Hom_{\bD}(D,D')$ is dense for every pair of objects $D,D'$ in $\bD$. By the description of the norm in $ \colim_{\bI} \bC$ given by Proposition \ref{roigjowegwregrgweg} and Assumption \ref{sdffsdfswer},
for every $m$ in $M$ and $n$ in $\nat$ we can find $\phi_{m,n} \colon i\to i_{m,n}$ in $\bI_{i/}$ such that
$\|\phi_{m,n}(m) \|_{C_{i''}}\le\frac{1}{n}$. We let $\bJ$ be the  subcategory of $\bI$   with the set of  objects $\{i_{m,n}\:|\: m\in M , n\in \nat\}$ and the non-identity morphisms $\phi_{m,n} \colon i\to i'$.  Using that $\bI$ is \countably filtered by Assumption \ref{werpogwergregweg} we can now extend the inclusion of $\bJ$ into $\bI$ to the cone over $\bJ$
such that the cone tip is sent to an object  $i'$ of $\bI$.  We let $\phi \colon i\to i'$ be the unique morphism  which factorizes as
$i\xrightarrow{\phi_{m,n}} i_{m,n}\to i'$ for every $m$ in $M$ and $n$ in $\nat$ such that the second morphism belongs to this extension.

For every $m$ in $M$ and $n$ in $\nat$ we have by construction  $\bC(\phi)(m)  =0$.
By the density assumption  on $M$ this implies that $C(\phi)$ annihilates all morphisms of  $\bD$.
\end{proof}
 
{A second application of Proposition \ref{roigjowegwregrgweg} ist the following.} 
{For $C^{*}$-categories, a faithful functor is  the analogue of an isometric inclusion of $C^{*}$-algebras.}
\begin{kor}\label{qoifqjgfefeqwfq}
{Filtered colimits  in $\nCcat$ preserve  {faithful functors}.}
 \end{kor}
\begin{proof}
  {Let $\bI$ be a 
  small filtered
   category  and 
   $(\bA\to \bB) \colon \bI\to \nCcat$
be a natural transformation of functors such that  $\bA_{i}\to \bB_{i}$ is  {faithful} for every $i$ in $\bI$. Then we must show that
the induced morphism  $\colim_{\bI}\bA\to \colim_{\bI} \bB$ is  {faithful}.
But this} immediately follows from the explicit description of the norm on the colimits given in Proposition \ref{roigjowegwregrgweg}.
\end{proof}

 \begin{lem}\label{wpogjopwergrregwerg9}
{Filtered colimits in $\nCcat$ preserve exact sequences.}
 \end{lem}
 \begin{proof}
 {Let $\bI$  be a   small filtered 
 category and consider a diagram of exact sequences
$$0\to \bA\to \bB\to \bC\to 0$$ in $\nCcat$
indexed by $\bI$.
 We must show that the sequence
$$0\to {\colim}_{\bI}\, \bA\to {\colim}_{\bI}\, \bB\to {\colim}_{\bI}\, \bC\to 0$$
is exact.}

 By the definition of an exact sequence in $\nCcat$
 we have bijections  $\Ob(\bA)\cong \Ob(\bB)\cong\Ob(\bC)$. For a set $X$ we let $0[X]$ denote the $C^{*}$-category with the set of objects $X$ and only zero morphisms.
 We write the diagram of exact sequences as a diagram
of squares
$$\xymatrix{\bA\ar[r]\ar[d]&\bB\ar[d]\\0[\Ob(\bB)]\ar[r]&\bC}$$ in $\nCcat$
which are cartesian and cocartesian.
The colimit of cocartesian squares
\begin{equation}\label{frwefqrffewfewfq}
\xymatrix{\colim_{\bI}\bA\ar[r]\ar[d]&\colim_{\bI}\bB\ar[d]\\0[\Ob(\colim_{\bI}\bB)]\ar[r]&\colim_{\bI}\bC}
\end{equation} 
is again cocartesian, where we exploit that the two functors  $0[-] \colon \Set \to \nCcat  $ and $\Ob(-) \colon \nCcat\to  \Set$ are  both left adjoints by \cite[Lem.\ 3.8.1 \& 3.8.2]{crosscat}  and therefore commute with colimits in order to calculate the lower left corner in the colimit. This in particular implies that the image of the functor $\colim_{\bI}\bA\to \colim_{\bI}\bB$ 
is the kernel of the functor $\colim_{\bI}\bB\to \colim_{\bI}\bC$. In order to show that the square in \eqref{frwefqrffewfewfq}  is also cartesian it therefore suffices to show that  $\colim_{\bI}\bA\to \colim_{\bI}\bB$  is isometric. This is exactly  the assertion of Corollary \ref{qoifqjgfefeqwfq} which is applicable here since the inclusions $\bA_{i}\to \bB_{i}$ are isometric for all $i$ in $\bI$.  
\end{proof}

 The following  technical lemma is used in the proof of Lemma \ref{eriguhwiegugwergwerg}.   

 \begin{lem}\label{roigwerogwregwergwregwreg}
 The minimal tensor product {on $\nCalg$} preserves isometric inclusions. 
 \end{lem}
 \begin{proof}
 If $A'\to A$ is an isometric inclusion, then we must show that 
${A'} \otimes_{\min} B\to {A}\otimes_{\min} B$ is again an isometric inclusion. We  choose faithful representations of $\alpha$ and $\beta$ of $A$ and $B$ as above, respectively.
 Then we can use $\alpha_{|A'}$ as a faithful representation of $A'$. 
 The assertion is now clear {from \eqref{dcdcdascadscdascadc}.}
 \end{proof}

We can now prove   Lemma \ref{eriguhwiegugwergwerg}  and hence complete the proof of Proposition \ref{eroigjweogieogigw}.\phantomsection\label{weoigjwoegffrefgrgwgf}
We discuss the cases $?=\min$ and $?=\max$ separately.

\begin{proof}[Proof of {Lemma} \ref{eriguhwiegugwergwerg} in the case $?=\min$]
 Let $\bI$ be a  filtered category, 
  $A \colon \bI\to \nCalg$ be a diagram, and $B$ in $\nCalg$.
Then we consider the canonical map
\begin{equation}\label{wregkgpokpewgeggergwe}
{\colim}_{\bI}\, (A\otimes_{\min} B)\to ({\colim}_{\bI}\, A)\otimes_{\min} B\, .
\end{equation}
We must show that it is an isomorphism. 

Since the structure maps of the diagram $A$ are assumed to be isometric inclusions it follows from the explicit 
description of the colimit given by Proposition \ref{roigjowegwregrgweg} that the canonical maps
$A_{i}\to  \colim_{\bI}A$ are isometric inclusions.  By the Lemma \ref{roigwerogwregwergwregwreg} the homomorphisms
$A_{i}\otimes_{\min}B\to   {(\colim_{\bI}A)}\otimes_{\min}B$  are isometric inclusions, too.
Similarly, the structure maps of the system $A\otimes_{\min} B$ are isometric inclusions, and hence
$A_{i}\otimes_{\min}B\to \colim_{\bI}(A\otimes_{\min}  B)$ is an isometric inclusion. 
This implies by the Proposition \ref{roigjowegwregrgweg} that the canonical map  \eqref{wregkgpokpewgeggergwe} 
is an isometry. Since its
image clearly contains the dense subset ${({\colim_{\bI}}^{\nsAlg}A)}\otimes^{\alg}B$ 
 we conclude that it is an isomorphism.
\end{proof}

 The argument for  proof of Lemma \ref{eriguhwiegugwergwerg}
 in the case of $\otimes_{\min}$ can be used to deduce 
 the following general statement.  We let $\bM,\bN$ be in $\{\nCcat,\nCalg\}$, and we consider two groups $G,H$ and a functor 
  $F\colon\Fun(BG,\bM)\to\Fun(BH,\bN)$.    
 \begin{prop}\label{wegiojwoegerwferwfwerfw}
 Assume:
 \begin{enumerate}
 \item \label{qeriojgegwefrreferwf}
 $F$ preserves  {faithful functors}.\footnote{In the case of $C^{*}$-algebras we interpret {\em faithful functor} as {\em isometric inclusion}.}
 \item  \label{oirjforfqewfqwefqd}For any filtered diagram $\bC \colon \bI\to \Fun(BG,\bM)$ 
 the images of $F(\bC(i))\to F(\colim_{\bI}\bC)$ for all $i$ in $\bI$ together generate 
 $F(\colim_{\bI}\bC)$.
 \end{enumerate}
 Then $F$
preserves filtered colimits 
of diagrams whose structure maps are  {faithful functors}.
 \end{prop}
 \begin{proof}
 One argues as in the proof of Lemma \ref{eriguhwiegugwergwerg}
 in the case of $\otimes_{\min}$ replacing $-\otimes_{\min}B$ by
 $C$ that $\colim_{\bI}F(\bC)\to F(\colim_{\bI}\bC)$
 is  {faithful}. Then Assumption \ref{oirjforfqewfqwefqd} implies that this map is an isomorphism.
 \end{proof}

\begin{proof}[Proof of Lem.\ \ref{eriguhwiegugwergwerg} in the case $?=\max$]
In this argument we use that $\otimes_{\max}$ preserves exact sequences of $C^{*}$-algebras in each argument, see e.g.\ \cite[Prop.\ 3.7.1]{brown_ozawa}.
 Let $\bI$ be a  filtered category, 
  $A \colon \bI\to \nCalg$ be a diagram, and $B$ in $\nCalg$.
Then we consider the canonical map
\begin{equation}\label{wregkgpokpewgeggergwe_max}
{\colim}_{\bI}\, (A\otimes_{\max} B)\to ({\colim}_{\bI}\, A)\otimes_{\max} B\, .
\end{equation}
We must show that it is an isomorphism. 

We first consider the special case that $B$ is  in $\Calg$ and $A \colon \bI\to \Calg$. The latter condition means   that $A_{i}$ is unital for every $i$, and  for every morphism $i\to i'$ in $\bI$ the structure map $A_{i}\to A_{i'}$ preserves units.  
To prove that \eqref{wregkgpokpewgeggergwe_max} is an isomorphism in this case, it suffices to show that every homomorphism 
$\colim_{\bI}(A\otimes_{\max} B)\to T$  for every  $T$ in ${\nCalg}$  factorizes over a homomorphism 
$(\colim_{\bI}A)\otimes_{\max} B\to T$.

Let 
$\rho \colon \colim_{\bI}(A\otimes_{\max} B)\to T$ be a homomorphism. 
It determines a compatible family of homomorphisms 
$(\rho_{i} \colon A_{i}\otimes_{\max} B\to T)_{i\in \bI}$.
Using  the  unit    of $B$ we construct 
a compatible family of homomorphisms $(\pi_{i} \colon A_{i}\to T)_{i\in \bI}$ by
$\pi_{i}(a) \coloneqq  \rho_{i}(a\otimes 1_{B} )$.
This family induces a homomorphism $\pi \colon \colim_{\bI} A\to T$.
 For every $i$ in $\bI$ we can construct a homomorphism $\kappa_{i} \colon B\to T$ 
by $\kappa_{i}(b) \coloneqq \rho_{i}(1_{A_{i}}\otimes b)$.  Because $\bI$ is connected, these homomorphisms are independent of $i$.
We will just write $\kappa(b) \coloneqq \kappa_{i}(b)$ for any choice.
Note that $\rho_{i}(a\otimes b)=\pi_{i}(a)\kappa(b)$ for all $i$ in $\bI$ and $a$ in $A_{i}$, $b$ in $B$.
Then we get a map
$\rho' \colon (\colim_{\bI}A)\otimes^{\alg} B\to T$ determined by $\rho'(a\otimes b) \coloneqq \pi(a)\kappa(b)$.
 By the universal property of the maximal tensor product {$\rho'$} extends to $\rho'' \colon (\colim_{\bI}A)\otimes_{\max}  B\to T$. The homomorphism  $\rho$ factorizes over $\rho''$ as desired.

 We now consider the case of a diagram $A \colon \bI\to \nCalg$ but still assume that $B$ is unital.  We have a  split unitalization exact sequence $0\to A\to A^{+}\to \C\to 0$. 
  Applying to this sequence $-\otimes_{\max}B$ we get again a diagram of exact sequences. 
We consider the diagram
$$\xymatrix{0\ar[r] &\colim_{\bI}(A\otimes_{\max} B)\ar[r]\ar[d]^{!}&\colim_{\bI}(A^{+}\otimes_{\max} B)\ar[r]\ar[d]^{\cong}&B \ar[r]\ar@{=}[d]&0\\ 0\ar[r]&(\colim_{\bI}A)\otimes_{\max} B\ar[r] &(\colim_{\bI}A^{+})\otimes_{\max} B\ar[r] &B\ar[r]&0}$$
 The upper horizontal sequence is exact by Proposition \ref{wpogjopwergrregwerg9}.
 The right horizontal maps are obtained from    the canonical maps $A^{+}_{i}\to \C$ by tensoring with $B$.
  The middle vertical map is an isomorphism by the unital case of this lemma shown above. 
 We now argue that the lower horizontal sequence is exact, which will imply that the   arrow marked by $!$ is an isomorphism.
  We have an equivalence {of categories} $\nCalg\xrightarrow{\simeq} \Calg_{/\C}$ given by   $A\mapsto (A^{+}\to \C)$, and whose inverse is given by $(\phi \colon B\to \C) \mapsto \ker(\phi)$. Furthermore, the canonical functor $\Calg_{/\C}\to \Calg$ preserves colimits in view of the adjunction
$$((A\to \C)\mapsto A):\Calg_{/\C}\rightleftarrows \Calg: (B\mapsto (B\oplus \C\to \C))\, .$$
Hence
$${\colim}_{\bI}\, A\cong \ker({\colim}_{\bI}\, (A^{+}\to \C))\cong \ker({\colim}_{\bI}\, A^{+}\to \C)$$
and we have the split exact sequence
$$0\to  {\colim}_{\bI}\,A \to {\colim}_{\bI}\,A^{+} \to \C \to 0\, .$$
We finally use that $-\otimes_{\max}B$ preserves exact sequences. 

%
 
We finally allow $B$ {to be} non-unital. Then we get a diagram 
$$\xymatrix{0\ar[r] &\colim_{\bI}(A\otimes_{\max} B)\ar[r]\ar[d]^{!}&\colim_{\bI}(A\otimes_{\max} B^{+})\ar[r]\ar[d]^{\cong}&\colim_{\bI}A  \ar[r]\ar@{=}[d] &0\\ 0\ar[r]&(\colim_{\bI}A)\otimes_{\max} B\ar[r] &(\colim_{\bI}A)\otimes_{\max} B^{+}\ar[r] &\colim_{\bI}A\ar[r]&0}$$
The lower horizontal sequence is exact since $(\colim_{\bI}A)\otimes_{\max}-$ 
preserves exact sequences.
The middle vertical morphism is an isomorphism by the case considered above since $B^{+}$ is unital.  
The upper horizontal sequence is again exact by Proposition \ref{wpogjopwergrregwerg9}.  
We again conclude that the arrow marked by $!$ is an isomorphism. 
\end{proof}

{We now use    Proposition \ref{eroigjweogieogigw} in order to extend various results from $C^{*}$-algebras to $C^{*}$-categories.}

{\begin{prop}\label{iregowegwgwreggwre}\mbox{}
\begin{enumerate}
\item\label{iregowegwgwreggwre1} The maximal tensor product on $\nCcat$ preserves exact sequences.
\item\label{iregowegwgwreggwre2} The minimal tensor product on $\nCcat$ preserves   faithful functors.
\end{enumerate}
\end{prop}}
\begin{proof} {The analogue of Assertion \ref{iregowegwgwreggwre1} for $C^{*}$-algebras is 
 well-known, see e.g.\ \cite[Prop.\ 3.7.1]{brown_ozawa}. Let
\begin{equation}\label{fsvfvvfsdvfsvfvfvs}
0\to \bA\to \bB\to \bC\to 0
\end{equation}
be an exact sequence in $\nCcat$, and let $\bD$ be in $\nCcat$. We must show that the
 sequence 
\begin{equation}\label{avadcdscadscadss}
0\to \bA\otimes_{\max} \bD\to \bB\otimes_{\max} \bD\to \bC\otimes_{\max} \bD \to 0
\end{equation}  is exact. 
 The functor 
$A \colon \nCcat_{i}\to \nCalg$ preserves exact sequences \cite[Prop.\ 8.9.2]{crosscat}. 
 Applying $A$ to the exact sequence \eqref{fsvfvvfsdvfsvfvfvs} 
 we get the exact sequence
 $$0\to A(\bA)\to A (\bB)\to  A(  \bC)\to 0$$   in $\nCalg$. 
Since the maximal tensor product  in $\nCalg$ preserves exact  sequences the sequence  
$$0\to A(\bA)\otimes_{\max} A(\bD)\to A (\bB)\otimes_{\max} A(\bD)\to  A(  \bC)\otimes_{\max} A(\bD)\to 0$$
is exact.  By
Proposition \ref{eroigjweogieogigw} in the case of $?=\max$ (whose proof has been completed already)  
the sequence 
$$0\to A(\bA\otimes_{\max} \bD)\to A (\bB\otimes_{\max} \bD)\to  A(  \bC \otimes_{\max} \bD)\to 0$$
is then also exact. {We now  {employ}   that the functor $A$  detects} exactness.  {In order to see this,  note that}  for} all   pairs of  objects $B,B'$ of $\bB$ and $D,D'$ of $\bD$ {we can conclude that the sequence of complex vector spaces}
\begin{align*}
\mathclap{
0\to \Hom_{\bA\otimes_{\max} \bD}((B,D),(B',D')) \to  \Hom_{\bB\otimes_{\max} \bD}((B,D),(B',D')) \to   \Hom_{\bC\otimes_{\max} \bD}((B,D),(B',D'))\to 0
}
\end{align*}
is exact, since these morphism spaces are direct summands of the corresponding algebras by  \cite[Lem.\ 3.6]{joachimcat} {or} \cite[Lem.\ 6.7]{crosscat}. Here we use that the maps in an exact seqence of $C^{*}$-categories are bijective on objects so that we can interpret, e.g., $B$ also as an object of $\bA$ or $\bC$. {Consequently, the sequence in \eqref{avadcdscadscadss} is exact, too.}

{For Assertion \ref{iregowegwgwreggwre2}, we consider a faithful functor   $\phi\colon \bC\to \bD$ and $\bE$ in $\nCcat$.  We first assume that $\phi$ is injective on objects.  
Then we can form the commutative  square
\begin{equation*}\label{}
 \xymatrix{A(\bC\otimes_{\min} \bE)\ar[r]\ar[d]_{Prop. \ref{eroigjweogieogigw}}^{\cong}&A(\bD\otimes _{\min}\bE)\ar[d]_{Prop. \ref{eroigjweogieogigw}}^{\cong }\\ A(\bC)\otimes_{\min} A(\bE)\ar[r]&A(\bD)\otimes_{\min} A(\bE)}
\end{equation*}
Since the functor $A$ preserves isometric inclusions by \cite[Lem. 6.8]{crosscat}, it follows from Lemma \ref{roigwerogwregwergwregwreg} that the lower horizontal map is isometric.  We conclude that the  upper horizontal arrow is isometric, too.
  We now use that $A$ detects isometric inclusions.
In detail, for any  objects $C,C'$ of $\bC$ and $ E,E'$ of $\bE$ 
we have a commutative square of Banach  spaces 
$$\xymatrix{\Hom_{\bC\otimes_{\min}\bE}((C,E),(C',E'))\ar[r]\ar[d]& \Hom_{\bD\otimes_{\min}\bE}((\phi(C),E),(\phi(C'),E'))\ar[d]\\A(\bC\otimes_{\min} \bE)\ar[r]&A(\bD\otimes_{\min} \bE)}$$
As seen above the  lower horizontal morphism is isometric. The
vertical morphisms are isometric by     \cite[Lem. 6.7]{crosscat}.
It follows that the upper horizontal morphism is  isometric.
Since $C,C'$ and $E,E'$ were arbitrary this shows Assertion \ref{iregowegwgwreggwre2}   for functors which are in addition injective on objects.}

Since $-\otimes_{\min}\bE$ sends unitary equivalences to unitary equivalences and hence to isometries we can remove the 
assumption that $\phi$ is injective on objects using the same construction as in the proof of \cite[Thm.\ 18.6]{cank}. 
More precisely we can find a diagram
$$\xymatrix{ &\bD'&\\ \ar[ur]^{(1)}\bC\ar[rr]&&\bD\ar[ul]_{(2)}}$$ where $(1)$ and $(2)$ are injective on objects, $(1)$ is faithful, 
$(2)$ is a unitary equivalence.
We  get $$\xymatrix{ &\bD'\otimes_{\min} \bE&\\ \ar[ur]^{(1)\otimes\bE}\bC\otimes_{\min}\bE\ar[rr]&&\bD\otimes_{\min}\bE\ar[ul]_{(2)\otimes \bE}}$$
Then $(1)\otimes \bE$ is faithful by the case considered above and $(2)\otimes \bE$ is a unitary equivalence.
We conclude that the lower horizontal map is faithful, too.
\end{proof}

The following proposition partially generalizes Lemma \ref{eriguhwiegugwergwerg} from $C^{*}$-algebras to $C^{*}$-categories.
Let $\bI$ be a small category {and}   $\bC \colon \bI\to {\nCcat}$  be a diagram.

\begin{prop}\label{prop_colim_CCat_tensor} {We assume that $\bI$ is filtered.}  \begin{enumerate}
\item \label{wrthokpweferfwerf} {If the structure maps of the diagram are injective on objects, then 
the canonical morphism}
\begin{equation}\label{fqewfoijoiqwefqwefewfq}
\colim_{\bI}(\bC\otimes_{\max} \bD)\to (\colim_{\bI}\bC)\otimes_{\max} \bD
\end{equation} is an isomorphism.
\item \label{hertrrr} If   the structure maps of the diagram are  {faithful functors}, then  
\begin{equation}\label{fqewfoijoidqwefqwefewfq}
\colim_{\bI}(\bC\otimes_{\min} \bD)\to (\colim_{\bI}\bC)\otimes_{\min} \bD
\end{equation}
is an isomorphism.
\end{enumerate}
\end{prop}
\begin{proof}
{We first consider the case of $\otimes_{\max}$.}
{We get}
\begin{eqnarray*}
A({\colim}_{\bI}\,(\bC\otimes_{{\max}} \bD ))& \stackrel{\text{Lem.~}\ref{wegoihjiogewrgergegrewf}}{\cong} & {\colim}_{\bI}\, A(\bC\otimes_{{\max}} \bD) \\
&\stackrel{\text{Prop.~}\ref{eroigjweogieogigw}}{\cong}& {\colim}_{\bI}\,(A(\bC)\otimes_{{\max}} A( \bD)) \\
&\stackrel{\text{Lem.~}\ref{eriguhwiegugwergwerg}.\ref{weoigjweorrerggwergr}}{\cong}& ( {\colim}_{\bI}\, A(\bC))\otimes _{{\max}}A( \bD) \\
& \stackrel{\text{Lem.~}\ref{wegoihjiogewrgergegrewf}}{\cong} &A( {\colim}_{\bI}\, \bC)\otimes_{{\max}} A( \bD) \\
&\stackrel{\text{Prop.~}\ref{eroigjweogieogigw}}{\cong}&
A({\colim}_{\bI}\, \bC \otimes_{{\max}}  \bD)\,.
\end{eqnarray*}
{
We  observe that the functor  in \eqref{fqewfoijoiqwefqwefewfq}
induces a bijection on the level of objects. We then use 
that  functor $A$  detects isomorphisms among functors which are bijections on objects. Therefore 
arguing} as in the proof of Lemma \ref{iregowegwgwreggwre} we remove $A$ {to} conclude that \eqref{fqewfoijoiqwefqwefewfq}
is an isomorphism. 

{In the case of $\otimes_{\min}$ we use that
$-\otimes_{\min} \bD$ preserves  faithful functors by Proposition \ref{iregowegwgwreggwre}.\ref{iregowegwgwreggwre2}
and that the images of $\bC(i) \otimes_{\min}\bD\to 
\colim_{\bI}\bC\otimes_{\min} \bD$
for all $i$ in $\bI$ together generate $\colim_{\bI}\bC\otimes_{\min} \bD$.
We can therefore apply Proposition \ref{wegiojwoegerwferwfwerfw}.}
%
%
 \end{proof}

Let $\bI$ again be a small category, $G$ be a group,  {and let }   $\bC \colon \bI\to \Fun(BG ,\nCcat)$  be a diagram.

\begin{prop}\label{wervervfsdfvsdfvsdfvfsdv} {We assume that $\bI$ is filtered.}  \begin{enumerate}
\item \label{gwergrefwwerfw} If the structure maps of the diagram are injective on objects, then
the canonical morphism 
\begin{equation}\label{fqewfoijoiqwefqwefewfq1}
\colim_{\bI}(\bC\rtimes_{\max} G) \to (\colim_{\bI}\bC)\rtimes_{\max} G
\end{equation} is an isomorphism.
\item  If   the structure maps of the diagram are  faithful functors, then 
\begin{equation}\label{fqewfoijoidqwefqwefewfq1}
\colim_{\bI}(\bC\rtimes_{r} G)\to (\colim_{\bI}\bC)\rtimes_{r} G
\end{equation}
is an isomorphism.
\end{enumerate}
\end{prop}
\begin{proof}
The proof is analoguous to the proof of Proposition \ref{prop_colim_CCat_tensor}. We replace $-\otimes_{?}\bD$ by $
-\rtimes_{?}G$. 

For $?=\max$ we use the compatibility of 
$A$ with   the maximal crossed product \cite[Thm 8.6]{crosscat}.
Furthermore we use  Lemma \ref{qeriogjowergwergefefwefw}.\ref{trhijwothgwwergwergewg}  instead of Lemma \ref{eriguhwiegugwergwerg}.\ref{weoigjweorrerggwergr}.

In the case of $?=r$ {we want to apply Proposition \ref{wegiojwoegerwferwfwerfw}. 
First note  
that the reduced crossed product preserves 
faithful functors \cite[Thm.\ 12.24]{cank} verifying Assumption \ref{wegiojwoegerwferwfwerfw}.\ref{qeriojgegwefrreferwf}.  In order to verify     Assumption \ref{wegiojwoegerwferwfwerfw}.\ref{oirjforfqewfqwefqd} 
note that the reduced crossed product $(\colim_{\bI}\bC)\rtimes_{r} G$
is generated by the image of the algebraic crossed product
$(\colim_{\bI}\bC)\rtimes^{\alg}G$. Since $\colim_{\bI}\bC$ in turn  is generated by the images of $C(i)$ for all $i$ in $\bI$ we can conclude that 
$(\colim_{\bI}\bC)\rtimes_{r} G$ is generated by the images of
$C(i)\rtimes^{\alg}G$. Hence
$(\colim_{\bI}\bC)\rtimes_{r} G$ is  in particular generated  by the images of $C(i)\rtimes_{r}G$ for all $ i$ in $\bI$. This verifies Assumption \ref{wegiojwoegerwferwfwerfw}.\ref{oirjforfqewfqwefqd} and   Proposition \ref{wegiojwoegerwferwfwerfw} implies our assertion.}
 \end{proof}

 \begin{rem}
 We do not know whether  in  {Propositions \ref{prop_colim_CCat_tensor}.\ref{wrthokpweferfwerf} and \ref{wervervfsdfvsdfvsdfvfsdv}.\ref{gwergrefwwerfw}}
  the assumption  that the structure maps of the diagram are injective on objects  is really  necessary. 
  \hB
 \end{rem}

\appendix 
 \section{Applications to assembly maps}\label{weoigjegregewfr}

We  now  introduce a $\KK$-valued version of the Davis--L\"uck assembly map by specializing the general constructions from   \cite[Sec.\ 19]{cank}.   We furthermore explain its relation with the classical assembly map appearing in the Baum--Connes conjecture, thereby previewing some of the  results from \cite{bel-paschke}.

We start with $\bC$ in $\Fun(BG,\Ccat)$ 
and let $\ell $ 
be the localization map  from 
{\eqref{adfasdfsxas}}. We let $j^G \colon BG\to G\Orb$ denote the {fully faithful} inclusion of $BG$ into the orbit category of $G$ which sends
the unique object of $BG$ to the orbit $G$.  We then form the left Kan extension 
$j^{G}_{!}\ell (\bC)$ 
as indicated by the dotted arrow in the diagram 
$$\xymatrix{
BG\ar[rr]^{\ell(\bC)}\ar[dr]_{j^{G}}& &\Ccat_{\infty}\,.\\
&G\Orb\ar@{..>}[ur]_{j^{G}_{!} \ell (\bC)}&
}$$
We compose this left Kan extension with the functor $\kkAi$ from \eqref{tgwergergreewff} {(for the trivial group $G$)} and obtain the functor 
\begin{equation}\label{ergwergrefvcdcs}
k^{G}_{\bC} \coloneqq \kkAi(j^{G}_{!}\ell(\bC)) \colon G\Orb\to \KK
\end{equation} which
is an instance   of the functor in \cite[Def.\ 1{9}.3]{cank}. By Elmendorf's theorem this functor determines a 
$\KK$-valued equivariant homology theory 
\begin{equation}\label{gergwredcvdfv}
H(-,k^{G}_{\bC}) \colon G\Top\to \KK \, , \quad X\mapsto H(X,k^{G}_{\bC})
\end{equation}  
on  $G$-topological spaces.

We can calculate the values of the functor $k^{G}_{\bC}$ in \eqref{ergwergrefvcdcs}  explicitly. We consider a subgroup $H$ of $G$.
By the point-wise formula for the left Kan extension, the equivalence
$BH\simeq BG_{/(G/H)}$, by \cite[Thm.\ 7.8]{crosscat}  (expressing the colimit over $BH$ in terms of the maximal crossed product), and  using Proposition \ref{wegoijoihggrgwergegwe}   we get the equivalences 
\begin{eqnarray}
k^{G}_{\bC}(G/H)&\simeq & {\kkAi}(j^{G}_{!}\ell(\bC)(G/H))\label{ergoijerwogwergergrewf}\\
&\simeq & \kkAi(\smash{\colim_{BH}}\ \ell(  \bC))\nonumber
\\&\simeq& \kkA( \bC\rtimes_{\max} H)\nonumber\\
&\simeq & \kkHA(\bC)\rtimes_{\max} H\nonumber\, ,
\end{eqnarray}
where we omitted to write $\Res^{G}_{H}$ at various places.

Let $\cF$ be a family of subgroups of $G$ and denote by  $G_{\cF}\Orb$  the full subcategory of $G\Orb$ of $G$-orbits with stabilizers in $\cF$. {\begin{ddd} The Davis--L\"uck assembly map for the family $\cF$ and the functor $k^{G}_{\bC}$ is defined by
 \begin{equation}\label{fhefiuqhiufeffqwefqwqf}
\Ass_{\cF,k^{G}_{\bC}} \colon \colim_{G_{\cF}\Orb} k^{G}_{\bC}\to  k^{G}_{\bC}(*)\, .
\end{equation}
\end{ddd}}
Expressed in terms of the homology theory $H(-,k^{G}_{\bC})$ in \eqref{gergwredcvdfv} this map is equivalent to the map
$$\Ass_{\cF,k^{G}_{\bC}}\colon H( E_{\cF}G, k^{G}_{\bC}) \to H(*,k^{G}_{\bC})$$
induced by the map $E_{\cF}G\to *$, where $E_{\cF}G$ is a $G$-$CW$-complex representing the homotopy type of the classifying space of $G$ for the family $\cF$.

In the following we explain  the relation of the assembly map \eqref{fhefiuqhiufeffqwefqwqf} with the classical assembly map 
appearing in the Baum--Connes conjecture \cite{baum_connes_higson}.  The details of this comparison  will be developed in  \cite{bel-paschke}.

The constructions in  \cite{bel-paschke} depend {on the choice of a $C^{*}$-category $\bC$  in $\Fun(BG,\nCcat)$  which admits small AV-sums \cite[Def.\ 7.1]{cank}.}
We let {$\bC^{u}$} in $\Fun(BG,\Ccat)$ denote the invariant full subcategory of {$\bC$} of unital objects {\cite[Def.\ {2.14}]{cank}.} We furthermore consider the category  {$\bC^{(G)}_{\std}$} in  $\Fun(BG,\nCcat)$ defined in \cite[Def.\ 2.1{5}]{bel-paschke}.

We use the definition
$$\Kcat \coloneqq \KK(\C,\kkAi(-)) \colon \Ccat_{\infty}\to \Sp^{\la}$$
for the $K$-theory functor for $C^{*}$-categories. If we insert
$k^{G}_{{\bC}^{u}}$ from \eqref{ergwergrefvcdcs} into $  \KK({\C},-)$ we get the functor
$$K^{G,\topp}_{\bC} \coloneqq  \Kcat(j_{!}^{{G}}\ell({\bC}^{u})) \colon G\Orb\to \Sp^{\la}\, .$$
  On the other hand,  we can consider the functor 
\begin{equation}\label{saCOIJHOcsacc}
K^{G,{\An}}_{\bC} \coloneqq \Sigma K^{\lf}_{\kkGA({\bC}^{(G)}_{\std})} \colon G_{\Fin}\Orb\to \Sp^{\la}
\end{equation}
obtained by specializing the coefficients of the functor  in \eqref{wregiojoiwjerogwergrewgwerge} at  $\kkGA({\bC}^{(G)}_{\std})$ and applying the suspension functor $\Sigma$. 

If $H$ is a finite subgroup of $G$, then
we have   an equivalence (we again omitted $\Res^{G}_{H}$ at various places) 
\begin{eqnarray}
K^{G,\topp}_{\bC}(G/H)&\stackrel{\text{def}}{\simeq} &    \KK(\C, k_{{\bC}^{u}}^{G}(G/H))\label{wroighjorggewrgegfv}\\
& \stackrel{\eqref{ergoijerwogwergergrewf}}{\simeq} &    \KK(\C,   \kkHA({\bC}^{u})\rtimes_{\max} H)\nonumber\\
&\stackrel{\ref{wtgijoogwrewegewgrg}.\ref{gjeriogjreerwgwregwerg}}{\simeq} &    \KKH(\C,{\bC}^{u})\nonumber\\
&\stackrel{\eqref{eq_ind_res_adjunction_IC}}{\simeq} &    \KKG(C_{0}(G/H),{\bC}^{u})\nonumber\\&\stackrel{\text{def}}{\simeq}&  K^{\lf}_{\kkGA({\bC}^{u})}(G/H)
%
%
\, . \nonumber
\end{eqnarray}
The right-hand side of \eqref{wroighjorggewrgegfv} looks similar to the evaluation of \eqref{saCOIJHOcsacc} at $G/H$ (up to a suspension), the difference lies  in the coefficient categories
${\bC}^{u}$ and ${\bC}^{(G)}_{\std}$, respectively.
 But as a consequence of the general Paschke duality theorem which will be shown in \cite{bel-paschke} one can indeed show that these functors are {equivalent.} 
In the following we explain this in more detail. 
The comparison involves a third functor
$${K\bC^{G}}  \colon G\Orb  \to \Sp^{\la}$$
{defined in   \cite[Def.\ 12.2]{bel-paschke}. Equivalently, this functor is given by  \cite[Def.\ 19.{12}]{cank} if one sets 
   $\Homol=  \Kcat$ and    replaces  (in the reference)  $\bC$  by $\bC^{u}$.} 
 The notation for the functors 
$ K^{G,\topp}_{\bC}$ and $ {K\bC}^{G} $ which is used in   \cite{cank} is $  {\Kcat}^{G}_{{\bC}^{u},\max}$ and
$  {\Kcat}^{G}_{{\bC}^{u},r}$, respectively.
Because $\Kcat$ is Morita invariant, by \cite[Prop.\ {19.14.3}]{cank}
we have an equivalence
\begin{equation}\label{ffwqewfweddqwd}
(K^{G,\topp}_{\bC})_{|G_{\Fin}\Orb}\to 
({K\bC}^{G})_{|G_{\Fin}\Orb}\, .
\end{equation}
{The}   following equivalence is the left vertical {Paschke duality} equivalence in  \cite[(14.29)]{bel-paschke}:
 \begin{equation}\label{regwergegwer}
 ( {{\Sigma} K\bC}^{G})_{|G_{\Fin}\Orb}  \xrightarrow{\simeq} K^{G,{\An}}_{\bC} \, 
\end{equation}
(note that the r.h.s.\ is only {defined} on $G_{\Fin}\Orb$).
Combining \eqref{regwergegwer}
and \eqref{ffwqewfweddqwd} we obtain the equivalence 
\begin{equation}\label{wergoiweforefrfwe}
({\Sigma }K^{G,\topp}_{\bC})_{|G_{\Fin}\Orb}\xrightarrow{\simeq} K^{G,\An}_{\bC} \, .
\end{equation} 

%

The relation of the assembly map $\Ass_{\Fin,k^{G}_{{\bC}^{u}}}$ from \eqref{fhefiuqhiufeffqwefqwqf} {(note the superscript $u$)} with the classical Baum--Connes assembly map $\Ass^{BC}$ is now best explained by the following diagram: 
\begin{align}
\label{refoijoergwergrefw}\\
\mathclap{\xymatrix{
 \KK(\C,\colim_{G_{\Fin}\Orb} k_{{\bC}^{u}}^{G})\ar[rrr]^-{  \KK(\C,\Ass_{\Fin,k_{{\bC}^{u}}^{G}})} &&&  \KK(\C,k^{G}_{{\bC}^{u}}(*)) &\\
 \colim_{G_{\Fin}\Orb}    K^{G,\topp}_{\bC} \ar[rrr]^{\Ass_{\Fin,K^{G,\topp}_{\bC}}}\ar[u]^{(1)}_{\simeq}\ar[d]_{(2)}^{\simeq}&&&    K^{G,\topp}_{\bC}(*)\ar[u]^{\text{def's}}_{\simeq}\ar[d]^{(3)} &\ar[l]_-{(5)}^-{\simeq}   \Kcat({\bC}^{u}\rtimes_{\max}G)\ar[d]^{(4)}\\  
 \colim_{G_{\Fin}\Orb} {K\bC}^{G} \ar[d]^{\simeq}\ar[rrr]^{\Ass_{\Fin, {K\bC}^{G}}}&&& {K\bC}^{G}(*) \ar[d]^{\simeq}&\ar[l]_-{(6)}^-{\simeq}   \Kcat({\bC}^{u}\rtimes_{r}G)\\
 \ar[d]_{(8)}^{\simeq} H(E_{\Fin}G,  {K\bC}^{G})\ar[rrr]^{\Ass_{\Fin, {K\bC}^{G}}} &&& H(*, {K\bC}^{G}) \ar[d]_{(7)}^{{\simeq}} &\\
 \colim_{W\subseteq E_{\Fin G}} {\KKG}(C_{0}(W), {\bC}^{(G)}_{\std}) \ar[rrr]^-{\Ass^{BC}} &&   &\Kcat({\bC}^{(G)}_{\std}\rtimes_{r}G)&
}} \notag
\end{align}
where the colimit in the lower left corner runs over the $G$-finite subcomplexes of $E_{\Fin}G$.
The arrow  marked by $(1)$ is up to inserting  definitions the canonical map
$$\colim_{G_{\Fin}\Orb} \KK(\C,k^{G}_{{\bC}^{u}})\to  \KK(\C, \colim_{G_{\Fin}\Orb}k^{G}_{{\bC}^{u}})\, .$$
{Using stability of $\KK$ it} is an equivalence since $\kk(\C)$ is a compact object of $\KK$. The upper square commutes by construction.
The arrows marked by $(2)$ and $(3)$ are induced by the natural transformation $c$ from  \cite[Prop.\ {19.14.1}]{cank}. 
The corresponding square commutes by  the
 naturality of this transformation.
 The map $(4)$ is induced by the canonical map from the maximal to the reduced crossed product.
 The equivalence marked by $(5)$  is obtained by inserting the   calculation   \eqref{ergoijerwogwergergrewf}  for $G=H$
into the definitions. 
 The equivalence marked by $(6)$ is justified by
 \cite[Cor.\ 19.13]{cank}. The square involving these two maps commutes by an inspection of the construction of the natural transformation $c$ in \cite[(19.19)]{cank} denoted $\nu$ in loc.~cit.
 The equivalence $(7)$
is  {given by 
  \cite[Prop.\ 13.5]{bel-paschke}.}
 Finally, 
the equivalence $(8)$ 
involves the Paschke duality and is induced by the 
left vertical equivalence in 
   \cite[(14.29)]{bel-paschke}.
%
%
The map $\Ass^{BC}$  {is} defined in \cite[{Def.\ 12.8}]{bel-paschke}.
On the level of homotopy groups it 
 is equal to  the classical Baum--Connes assembly map
\cite{baum_connes_higson}, slightly extended to $G$-$C^{*}$-categories as coefficients.

The following is the second main result of {\cite[Thm.\ 1.9]{bel-paschke}}:
\begin{theorem}
The lower square in \eqref{refoijoergwergrefw} commutes after taking homotopy groups.
\end{theorem}
The commutative diagram in \eqref{refoijoergwergrefw}
 implies the compatibility
of the Baum--Connes and the Davis--L\"uck assembly maps. This fact has been stated in \cite{hamped}, and a complete proof   has recently been given in 
\cite{kranz}   {(see also \cite[Sec.\ 15]{bel-paschke} for  a detailed review)}
 using   methods which are completely different from the ones in \cite{bel-paschke}.

\bibliographystyle{alpha}
\bibliography{forschung}

\end{document}